\documentclass[12pt,letterpaper,oneside]{book}

\usepackage[headings]{thesis2e}

\usepackage{amsmath}
\usepackage{amsthm}
\usepackage{amssymb}
\usepackage{tikz}

\usepackage[utf8]{inputenc}	
\usepackage[english]{babel} 	
\usepackage[T1]{fontenc}	
\usepackage{txfonts} 		
\usepackage{textcomp}		

\usepackage{graphicx}		
\usepackage{verbatim}		
\usepackage{rotating}
\usepackage{framed}

\theoremstyle{plain}
\newtheorem{thm}{Theorem}[section]
\newtheorem{lem}[thm]{Lemma}
\newtheorem{prop}[thm]{Proposition}

\theoremstyle{definition}
\newtheorem{defn}{Definition}
\newtheorem{conj}{Conjecture}
\newtheorem{ex}{Example}[chapter]

\newcommand{\cA}{\ensuremath{\mathcal{A}}}

\newcommand{\cB}{\ensuremath{\mathcal{B}}}

\newcommand{\cC}{\ensuremath{\mathcal{C}}}

\newcommand{\CC}{\ensuremath{\mathbb{C}}}
\newcommand{\cD}{\ensuremath{\mathcal{D}}}

\newcommand{\cE}{\ensuremath{\mathcal{E}}}

\newcommand{\cF}{\ensuremath{\mathcal{F}}}

\newcommand{\cH}{\ensuremath{\mathcal{H}}}

\newcommand{\cM}{\ensuremath{\mathcal{M}}}

\newcommand{\NN}{\ensuremath{\mathbb{N}}}

\newcommand{\cP}{\ensuremath{\mathcal{P}}}

\newcommand{\cQ}{\ensuremath{\mathcal{Q}}}

\newcommand{\QQ}{\ensuremath{\mathbb{Q}}}

\newcommand{\RR}{\ensuremath{\mathbb{R}}}

\newcommand{\kS}{\ensuremath{\mathfrak{S}}}

\newcommand{\cT}{\ensuremath{\mathcal{T}}}

\newcommand{\kT}{\ensuremath{\mathfrak{T}}}

\newcommand{\kU}{\ensuremath{\mathfrak{U}}}

\newcommand{\cW}{\ensuremath{\mathcal{W}}}

\newcommand{\cZ}{\ensuremath{\mathcal{Z}}}

\newcommand{\ZZ}{\ensuremath{\mathbb{Z}}}

\DeclareFontEncoding{OT2}{}{} 

\makeatletter
\def\testb#1{\testb@i#1,,\@nil}%
\def\testb@i#1,#2,#3\@nil{%
  \draw[-, thick] (O) --++(#1);
  \ifx\relax#2\relax\else\testb@i#2,#3\@nil\fi}
\makeatother   

\newcommand{\makediag}[1]{
    \coordinate (O) at (0,0); \coordinate (N) at (0,1);
    \coordinate (NE) at (0.8,0.8); \coordinate (E) at (1,0);
    \coordinate (SE) at (0.8,-0.8); \coordinate (S) at (0,-1);
    \coordinate (SW) at (-0.8,-0.8);\coordinate (W) at (-1,0);
    \coordinate (NW) at (-0.8,0.8); \coordinate (B1) at (1.2,1.2);
    \coordinate (B2) at (-1.2,-1.2);
    
    \testb{#1}
} 

\newcommand{\diagr}[1]{
  \begin{tikzpicture}[scale=0.2]\makediag{#1}\end{tikzpicture}
}

\newcommand{\diagb}[1]{
  \begin{tikzpicture}[scale=0.5]\makediag{#1}\end{tikzpicture}
}

\newcommand{\drn}{\begin{tikzpicture}[scale = 2] \draw[->, very thick] (0,0) -- (0,0.2); \draw[->, very thick] (0,0) -- (0,0.2); \draw[->, very thick] (0,0) -- (0,0.2); \end{tikzpicture}}
\newcommand{\drs}{\begin{tikzpicture}[scale = 2] \draw[->, very thick] (0,0) -- (0,-0.2); \draw[->, very thick] (0,0) -- (0,-0.2); \draw[->, very thick] (0,0) -- (0,-0.2); \end{tikzpicture}}

\author{Samuel Johnson}
\qualifications{B. Math. (Hons.), University of Newcastle, 2010}
\title{Analytic combinatorics of\\planar lattice paths}
\degree{Master of Science}
\department{Mathematics\\Faculty of Science}
\submitdate{Summer 2012}

\fairdealingcopyright	

\chair{Dr. Matthew DeVos}
\committee{Dr. Marni Mishna\\Senior Supervisor\\Assistant Professor}
\committee{Dr. Karen Yeats\\Committee Member\\Assistant Professor}
\committee{Dr. Petr Lisonek\\Internal/External Examiner\\Associate Professor}

\approvaldate{15 June 2012}

\abstract{Lattice paths effectively model phenomena in chemistry, physics and probability theory. Asymptotic enumeration of lattice paths is linked with entropy in the physical systems being modeled. Lattice paths restricted to different regions of the plane are well suited to a functional equation approach for exact and asymptotic enumeration. This thesis surveys results on lattice paths under various restrictions, with an emphasis on lattice paths in the quarter plane. For these paths, we develop an original systematic combinatorial approach providing direct access to the exponential growth factors of the asymptotic expressions.}

\begin{document}
\frontmatter 
\maketitle 

\chapter[Dedication]{}
\begin{center}
{\Large
\emph{To Donald Anthony, Eric Sidney, TJ and Saxon Eric Stephen}}
\end{center}

\chapter{Acknowledgments}
First, I would like to acknowledge the funding provided throughout my degree. Research assistant funding was provided by the Mathematics of Computer Algebra and Analysis (MOCAA) project of mprime. Simon Fraser University also provided funding in the form of Teach Assistantships and a Graduate Fellowship.

Thanks to the examining committee for taking the time to read this, especially my supervisor Dr. Marni Mishna, without whose invaluable guidance I would never have finished on time with something I'm happy to put my name on.

Many thanks and much love also go to my loving family, who each contributed in their own way, but all held back the blank stare every time I spoke about my work and offered kind words of encouragement when stress began to fray the edges of my days.

And last but certainly not least, much love and thanks to Sweeney, my little yellow bird, for every little piece of support. Especially the Monday morning pots of coffee and putting up with my monotonous lack of any other conversation topic.

\newpage
\addcontentsline{toc}{chapter}{\contentsname}
\tableofcontents

\newpage 
\addcontentsline{toc}{chapter}{\listtablename}
\listoftables

\newpage
\addcontentsline{toc}{chapter}{\listfigurename}
\listoffigures


\mainmatter


\part{Background}\label{bg}

\chapter{Introduction}

In many applications, lattice paths are models of physical and chemical phenomena, with the enumerative results closely linked with physical properties. For statistical mechanical systems modeled using combinatorial structures, growth constants, known as exponential growth factors in the enumerative world, provide information about the entropy in the physical system \cite{va00}. 

The focus of this thesis is a family of planar lattice walks confined to different regions of the plane, with an emphasis on the quarter plane, a region for which explicit closed form enumeration is presently unknown for many members of this family. For many models of quarter plane walks, asymptotic expressions are known, but the present methods of finding them require heavy computation.

\begin{sidewaystable}
  \begin{center}
    \begin{tabular}{|c|r|l|l|}
      \hline
      Class & Functional Equation & Classification of GF & Enumerative Results \\
      \hline
      $\cW_\kS$ & $W_\kS(x,y;t) = 1 + tS(x,y)W_\kS(x,y;t)~~$ & ~~Rational &  \begin{tabular}{l} Exact and asymptotic \\ formulae for coefficients. \end{tabular} \\
      \hline
      $\cF_\kS$ & \begin{tabular}{r} $F_\kS(y;t) = 1 + tP(y)F_\kS(y;t) $ \\ $ - \{y^{<0}\}tP(y)F_\kS(0;t)$ \end{tabular} & ~~Algebraic & \begin{tabular}{l} Coefficient formulae from coefficient \\ extraction of explicit GF, \\ asymptotic results from Theorem \ref{dirasympt} \end{tabular} \\
      \hline
      $\cH_\kS$ & \begin{tabular}{r} $H_\kS(x,y;t) = 1 + tS(x,y)H_\kS(x,y;t) $ \\ $ - \{y^{<0}\}tS(x,y)H_\kS(x,0;t)$ \end{tabular} & ~~Algebraic & ~~Results via bijection with $\cF_\kS$ \\
      \hline
      $\cQ_\kS$ & \begin{tabular}{r} $Q_\kS(x,y;t) = ~~~~~ 1 + tS(x,y)Q_\kS(x,y;t) $ \\ $- t\left( \{x^{<0}\}S(x,y)Q_\kS(0,y;t)\right.$ \\ $\left. + \{y^{<0}\}S(x,y)Q_\kS(x,0;t) \right)$ \\ $ + \omega(\kS)\frac{t}{xy}Q_\kS(0,0;t)$ \end{tabular} & \begin{tabular}{l} 23 D-finite \\ 51 non D-finite \\ 5 conjectured non D-finite \end{tabular} & \begin{tabular}{l} For D-finite cases: \\ enumerative results for 19 \\ cases via orbit sums, some direct \\ counting results, experimental \\ asymptotics via long sum \\ generation, exponential  growth \\ factors from bounding arguments \end{tabular} \\
      \hline
    \end{tabular}
    \caption{A table summarising each class, its functional equation, classification of GFs and summary of enumerative results.}\label{fetabintro}
  \end{center}
\end{sidewaystable}

Motivated by finding a more intuitive strategy for proving the exponential growth factors combinatorially, we study planar lattice paths with small or nearest neighbour steps. These are paths that move from the origin to another point in a region $R \subseteq \ZZ^2$, using a set of steps $\kS$ taken as a subset of $\bar{\kS} = \{ -1,0,1 \}^2 \setminus \{(0,0)\}$, which we also may represent by the corresponding compass directions
\[
 \bar{\kS} = \{\textbf{N,~NE,~E,~SE,~S,~SW,~W,~NW}\}.
\]
The number of steps taken by a walk is the length of the walk, and we count the number of walks of a given length on a step set subject to some restrictions. Figure \ref{walkex1} is an unrestricted lattice path of length 1000 with steps taken from $\{ \textbf{N,~E,~S,~W}\}$, starting at the origin.

\begin{figure}[h]
 \begin{center}
    \begin{tikzpicture}[scale=.15]
      \draw[<->,dashed] (-30,0) -- (30,0);
      \draw[<->,dashed] (0,-30) -- (0,30);
      \pgfmathdeclarerandomlist{steps}{{(0,1)}{(1,0)}{(0,-1)}{(-1,0)}}
      \coordinate (current point) at (0,0);
      \draw[fill=black] (current point) circle (10pt);
      \foreach \t in {1,...,1000} {
	\pgfmathrandomitem{\step}{steps}
	\draw (current point) -- +\step coordinate (current point);
      }
      ;
  \end{tikzpicture}
  \caption{A walk of length 1000 on $\kS = \{ \textbf{N,~E,~S,~W}\}$. }\label{walkex1}
  \end{center}
\end{figure}

Our approach uses generating functions, that is, formal power series of the form
\[
 R_\kS(t) = \sum_{n \geq 0} r_\kS(n)t^n,
\]
where $r_\kS(n)$ is the number of walks of length $n$ taken with steps from $\kS$ in the desired region. We index the generating functions by the step set the lattice paths are taken on, and the region that the paths are restricted to will be clear from the context. The growth constants, or exponential growth factors, are found in the asymptotic expression for the sequence$~(r_\kS(n))$. For the cases considered here, these expressions take the form
\[
  r_\kS(n) \sim \kappa \beta^n \theta(n),
\]
Where $\kappa$ is a positive constant, $\beta$ is the exponential growth factor and $\theta(n)$ is some subexponential growth factor. We define this formally later, but in an intuitive the asymptotic expression is an approximation of the sequence $r_\kS(n)$, the quality of which increases with $n$. The form taken by $\theta(n)$ can vary, and is related to the kinds of equations that $R_\kS(t)$ satisfies. We focus on generating functions classified as rational, algebraic, transcendental D-finite and non-D-finite. For some classes of generating functions, the form that $\theta(n)$ takes is known: for example, for algebraic functions $\theta(n)$ is a Laurent monomial, and for rational functions $\kappa = \theta(n) = 1$. We give a summary of the current state of research, including our contribution, of the four classes of planar lattice paths discussed in this thesis in Table \ref{fetabintro}.

In \cite{BoMi10}, Bousquet-M\'{e}lou and Mishna show that there are 79 non-equivalent models of walks confined to the quarter plane. They then classified these according to the cardinality of a group associated with the step set. For 23 models, the groups are finite and the associated generating functions are proven to be D-finite, that is, they satisfy a linear differential equation. Unfortunately the proof does not give the differential equation explicitly, and even if it did, there would be no guarantee of finding a solution. They conjectured the remaining 56 to be non D-finite, and this was very recently proved for 51 cases by Kurkova and Raschel in \cite{KuRa11}. 

We provide here a combinatorial method which will confirm the exponential growth factors numerically obtained by Bostan and Kauers in \cite{BoKa09} for all of the 23 D-finite cases. The data provided by Bostan and Kauers is an invaluable guide (as it usually is when you know what you're trying to prove), but often the exponential growth factors can also be hypothesised using simpler numerical methods. They show in \cite{BoKa09} that these generating functions are not only D-finite, but are also $G$-series. This means that the number of walks grows like
\[
 \kappa \beta^n n^\alpha (\log n)^\gamma
\]
for some $\kappa,\alpha, \beta, \gamma$. By generating the sequence using experimental techniques and taking the ratio of successive terms far along the sequence, we can approximate $\beta$. For instance, if we are looking at a counting sequence with an algebraic generating function, the ratio of two successive terms would be
\[
 \frac{w(n+1)}{w(n)} \sim \frac{\kappa \beta^{n+1} (n+1)^\alpha }{\kappa \beta^{n} n^\alpha} = \beta \left(\frac{n+1}{n} \right)^\alpha.
\]
Using this, we can get an approximation for $\beta$ that decreases in quality as $\alpha$ grows. In \cite{BoKa09}, Bostan and Kauers never encounter an exponent $\alpha$ larger than $-1/2$, so taking $\alpha = 1$ would provide a reasonable worst case scenario. For $n = 1000$ terms, this gives $1.001 \beta$ as a result, and for $n = 2000$, $1.0005 \beta$. In general, this is not precise enough that a computer could guess the growth constants, but if the result is close to the growth constant of a related model, then it seems clear to a human what these are converging on.

The method for the combinatorial proof of these 23 factors appeals to relaxations and restrictions of the models. By considering the right relaxation (respectively restriction), one can place an upper (resp. lower) bound on the exponential growth factor. Elementary analysis does the rest.

In Chapter \ref{def}, we introduce the formal machinery of lattice paths, and some supporting notation that we find useful throughout the text. These are at the most general level only, and more specific definitions will be given in later chapters.

Chapter \ref{statmech} will give an outline of the link between exponential growth constants and statistical mechanical properties. We also discuss the famous Self Avoiding Walk (SAW) model in three different configurations, and the current state of knowledge on bounds for the empirically known exponential growth factor.

In Chapter \ref{ac}, the necessary elements of analytic combinatorics are surveyed. This starts from an introduction of a symbolic method, and ends with a proof of the correspondence between the radius of convergence of a generating function and the exponential growth factor of the coefficient sequence.

Chapter \ref{undir} reviews planar lattice walks with small steps which are completely unrestricted. This is the classical case, easily enumerated with a rational generating function.

Chapter \ref{dir} is the first restriction to the general class, where we consider directed paths. These paths have a priveleged direction of increase, with the $x$-coordinate of the endpoint equal to the length of the path. We survey results from \cite{BaFl02}, with an emphasis on directed paths confined to the upper half plane. Generating functions are shown to be algebraic through the use of the kernel method, and asymptotic enumeration is given. This chapter is significant, as it contains the kind of result that we are working towards: exact enumerative formulas dependent on the drift parameter.

Chapter \ref{hpwalks} considers the class of walks confined to the upper half plane. A bijection between this class and the class of directed paths from Chapter \ref{dir} is given. This shows generating functions to be algebraic, and enumerative results follow immediately from the results of \cite{BaFl02} given in Chapter \ref{dir}

Chapter \ref{qpwalks} covers the main problem, planar lattice walks with small steps confined to the quarter plane. The class is split into two families, those known to be D-finite and those known or conjectured to be non D-finite. Here we give an original, systematic method of proving the exponential growth factors for 23 D-finite models of quarter plane lattice paths.

Chapter \ref{conclusion} concludes the thesis. We consider the generalisations of the technique to all non-equivalent quarter plane models. We also consider the prediction of the polynomial growth of the counting sequence, and how we can put this work into a single, general framework.

\chapter{Definitions and Notation} \label{def}
This chapter introduces formal definitions relating to planar lattice paths with small steps.

\begin{defn}
Fix a finite set of vectors in $\ZZ^2$, $\kS = \{(a_1,b_1),(a_2,b_2),...,(a_m,b_m)\}$. A \textit{lattice path} or \textit{walk} taken on $\kS$ is a sequence $v = (v_1,v_2,...,v_n)$ such that each $v_i$ is in $\kS$. We may realise this geometrically as the sequence of points $(p_0,p_1,...,p_n)$ such that $p_0 = (0,0)$ and the vector difference $p_{i} - p_{i-1} = v_i$. The elements of $\kS$ are known as the \textit{steps}, and the number of steps in the sequence $v$, $n$, is called the length of the path.
\end{defn}

We recall the definiton of a Laurent polynomial which encodes the step set $\kS$ in a useful way.

\begin{defn}
For a set of steps $\kS = \{(a_1,b_1),(a_2,b_2),...,(a_m,b_m)\}$ we define the \textit{inventory}, or \textit{characteristic polynomial}, of $\kS$ to be the Laurent polynomial denoted by $S(x,y)$ which is defined by
\[
 S(x,y) := \sum_{i=1}^m x^{a_i}y^{b_i}.
\]
\end{defn}

Sometimes it is useful to add weights to the steps of a walk. In combinatorial applications, these weights are often natural numbers. Here they correspond to different coloured steps in the same direction. In probabilistic applications, the the total weight of a step set (the sum of weights over the steps) would be equal to 1. In either case, we define a set of weights $\Pi = \{w_1,...,w_m\}$, and modify the inventory to include the weights:
\[
 S(x,y) = \sum_{i=1}^m w_ix^{a_i}y^{b_i}.
\]

In order to simplify analysis we restrict ourselves to what are known as small steps until the end of Chapter \ref{qpwalks}. This is just a way of saying that at each step, the walk moves only to nearest integral lattice points diagonally, horizontally and vertically. We discuss the generalisation to larger steps in Chapter \ref{conclusion}.

\begin{defn}
If $\kS$ is a subset of $\bar{\kS} = \{-1,0,1\}^2 \setminus \{ (0,0) \}$ then we say that $\kS$ is a set of \textit{small steps}.
\end{defn}

Sometimes, rather than using a vector representation it is convenient to represent a step set pictorially, as in Figure \ref{smallsteps}, or by the corresponding points on the compass. So, moving clockwise from the $(0,1)$ step in Figure \ref{smallsteps} these would be \textbf{N, NE, E, SE, S, SW, W, NW}, with corresponding inventory 
\[
  S(x,y) = y + xy + x + \frac{x}{y} + \frac1y + \frac1{xy} + \frac1x + \frac{y}{x}.
\]

\begin{figure}[h]
\begin{center}
 \diagb{N,NE,NW,W,E,SW,SE,S}
  \caption{The full set of small steps $\bar{\kS}$.}\label{smallsteps}
\end{center}
\end{figure}

We can also represent a walk with the corresponding polygonal line in $\ZZ^2$. Shown in Figure$~$\ref{walkex} is a geometric representation of a typical unrestricted walk on the set of small steps $\kS = \{\textbf{NE,~SE,~SW,~NW} \}$. By unrestricted we mean that we allow self intersections and the final point $p_n$ to be anywhere.

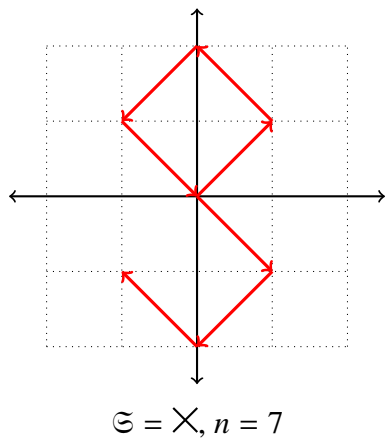
\begin{figure}[h]
 \begin{center}
    \begin{tikzpicture}[scale = 1]
    \draw[<->, thick] (-2.5,0) -- (2.5,0);
    \draw[<->, thick] (0,-2.5) -- (0,2.5);
    \draw[-, dotted] (-2,1) -- (2,1);
    \draw[-, dotted] (-2,2) -- (2,2);{itemize}
    \draw[-, dotted] (1,-2) -- (1,2);
    \draw[-, dotted] (2,-2) -- (2,2);
    \draw[-, dotted] (-1,-2) -- (-1,2);
    \draw[-, dotted] (-2,-2) -- (-2,2);
    \draw[-, dotted] (-2,-1) -- (2,-1);
    \draw[-, dotted] (-2,-2) -- (2,-2);
    \draw[->, very thick,red] (0,0) -- (1,1);
    \draw[->, very thick,red] (1,1) -- (0,2);
    \draw[->, very thick,red] (0,2) -- (-1,1);
    \draw[->, very thick,red] (-1,1) -- (0,0);
    \draw[->, very thick,red] (0,0) -- (1,-1);
    \draw[->, very thick,red] (1,-1) -- (0,-2);
    \draw[->, very thick,red] (0,-2) -- (-1,-1);
    \end{tikzpicture}
    
    $\kS = \diagr{NW,NE,SE,SW}$, $n=7$
    \caption{A walk of length 7 on a typical set of small steps $\kS$.}\label{walkex}
  \end{center}
\end{figure}

We find it useful at times to take the partial sum of a Laurent series or polynomial in which the exponents of a variable are negative. We recall some notation for this.

\begin{defn}
Let $H(y)$ be a Laurent series or polynomial in $y$ with coefficients taken from$~\ZZ(x)$. We denote by $\{y^{<0}\}H(y)$ the \textit{projection onto the pole part} of $H(y)$: the partial sum of $H(y)$ where all terms contain a negative index of~$y$.
\end{defn}

We illustrate this by example after the following definition.

\begin{defn}
We say that a step $(i,j)$ is $x$-positive if $i>0$. Similarly, $(i,j)$ is $x$-negative if$~i<0$, $y$-positive if $j > 0$ and $y$-negative if $j < 0$.
\end{defn}

\begin{ex}
Take $S(x,y)$ to be the inventory of the complete set of small steps. Then
\[
 \{y^{<0}\}S(x,y) = \{y^{<0}\} \left( y + xy + x + \frac{x}{y} + \frac1y + \frac1{xy} + \frac1x + \frac{y}{x} \right) =  \frac{x}{y} + \frac1y + \frac1{xy}
\]
is the Laurent polynomial representing all of the $y$-negative steps.
\end{ex}

Finally, we formally define what we mean by asymptotic growth. As mentioned in the introduction, intuitively it is a continuous function of $n$ which is an approximation of a sequence $(a_n)$, the quality of which improves as $n$ increases.

\begin{defn}
Let $a:\NN \rightarrow \NN$ be a function (such as the coefficient sequence of a generating function associated with a combinatorial class). We say that $a_n$ \textit{grows like} or has \textit{asymptotic growth} $g(n)$, $g: \RR^+ \rightarrow \RR^+$, denoted 
\[
 a(n) \sim g(n)
\]
if
\[
 \lim_{n \rightarrow \infty} \frac{a_n}{g(n)} = 1.
\]
\end{defn}

This concludes the general definitions for planar lattice paths with small steps. In Chapter \ref{ac}, more definitions are given as the formal framework of analytic combinatorics is built, and in Part \ref{plp}, more specific definitions will be made for each class, as restrictions are introduced and more machinery becomes necessary.

\chapter{Statistical Mechanics}\label{statmech}

Recent activity in lattice path combinatorics (for example, \cite{Bo10,BoKa09,BosKau09,BoMi10,MiRe09}) is due in part to the success with which they are applied to a wide variety of areas. Two often cited examples are applications to other combinatorial models through bijections, such as between lattice paths and plane partitions in the work of Alegri \textit{et al.} in \cite{AlBrSada11}, and queueing theory, as in B\"{o}hm's work \cite{Bo10} on lattice path counting and the theory of queues, but the example which will form the main focus of this chapter is the application of lattice models to the modelling of physical phenomena through statistical mechanics.

These applications take different forms, and a nice introduction to a number of these can be found in Janse van Rensburg's book `\textit{The statistical mechanics of interacting walks, polygons, animals and vesicles}' \cite{va00}. Of these applications, the most significant to us is the self avoiding walk (SAW). These walks are defined in the same way as our walks in Chapter \ref{def}, but we make the restriction that in the sequence of points $(p_0,...,p_n)$ there is no $i \neq j$ such that $p_i = p_j$. That is, the walk may not visit the same place twice. This makes these walks a good model for linear polymers, something we explain more in the next section. First, we should explain why combinatorial models are well suited to physical problems such as these.

Information about the entropy of a physical system is available in the form of a model's enumerative information. In statistical mechanics, entropy is a measure of the number of ways in which a system may be arranged, often taken to be a measure of `disorder' (the higher the entropy, the higher the disorder). This definition describes entropy as being proportional to the natural logarithm of the number of possible configurations of the individual atoms (with Boltzmann's constant as the proportionality constant). In order to find the number of possible configurations, we would be forced to count them. This is where enumerative combinatorics comes in. Since these are purely mathematical models we ignore the proportionality constant and any units of measurement, and focus solely on the enumeration.

So, if $\kS$ is a step set, and $\cP$ is a class of lattice paths on $\kS$ modeling some physical phenomena, we can find the combinatorial version of the entropy $E(n)$ by taking
\[
 E(n) = \log p(n),
\]
where $p(n)$ is the number of lattice paths in $\cP$ of length $n$. However, sometimes in statistical mechanics it is desirable to understand the behaviour of a model as the size approaches infinity. The sequence $p(n)$ is unbounded in general, so if this limit were taken the entropy would also be unbounded. There is another measurement though, related to entropy, called the \textit{limiting free energy} or the \textit{thermodynamic limit}. It is given as
\[
 F = \lim_{n \rightarrow \infty} \frac1n E(n).
\]
Now for many lattice paths (including those found in this thesis) the counting sequence has growth
\[
 p(n) \sim \kappa n^{\alpha} \beta^n,
\]
for $\kappa, \beta \in \RR^+$ and $\alpha \leq 0$, so the limiting free energy would be
\[
 F = \lim_{n \rightarrow \infty} \frac1n \log (\kappa n^{\alpha} \beta^n) = \beta,
\]
which is also known as the exponential growth factor of the number of states. So by finding exponential growth factors for models, we can give information about the entropy of a statistical mechanical system.

In the following section we survey a few different self-avoiding lattice paths and their applications. In light of the link between entropy and exponential growth, we also give the asymptotic expressions for each model.

\subsection*{Self-avoiding walks}\label{SAW}

A polymer is a connected arrangement (similar to a graph-theoretic structure) of molecules known as monomers. Polymers can take on tree-like structures, have internal polygonal structures, or simply be a linear chain of monomers. A linear polymer is a polymer where$~2$ monomers have only one neighbour and the remaining monomers have exactly 2 neighbours. As already mentioned, self avoiding walks (SAW) make very succesful models of linear polymers suspended in solution, see \cite{ReWhBrOw05,va10,BrDyLeOwPrReWh09} for some examples of this application in 2 and 3 dimensions. We focus on 2 dimensions here.

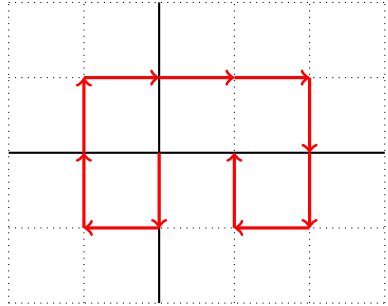
\begin{figure}
 \begin{center}
  \begin{tikzpicture}[scale = 1]
    \draw[thick,black] (-2,0) -- (3,0);
    \draw[thick,black] (0,-2) -- (0,2);
    \draw[-,dotted,black] (-2,1) -- (3,1);
    \draw[-,dotted,black] (-2,2) -- (3,2);
    \draw[-,dotted,black] (-2,-1) -- (3,-1);
    \draw[-,dotted,black] (-2,-2) -- (3,-2);
    \draw[-,dotted,black] (1,-2) -- (1,2);
    \draw[-,dotted,black] (2,-2) -- (2,2);
    \draw[-,dotted,black] (3,-2) -- (3,2);
    \draw[-,dotted,black] (-1,-2) -- (-1,2);
    \draw[-,dotted,black] (-2,-2) -- (-2,2);
    \draw[->,very thick,red] (0,0) -- (0,-1);
    \draw[->,very thick,red] (0,-1) -- (-1,-1);
    \draw[->,very thick,red] (-1,-1) -- (-1,0);
    \draw[->,very thick,red] (-1,0) -- (-1,1);
    \draw[->,very thick,red] (-1,1) -- (0,1);
    \draw[->,very thick,red] (0,1) -- (1,1);
    \draw[->,very thick,red] (1,1) -- (2,1);
    \draw[->,very thick,red] (2,1) -- (2,0);
    \draw[->,very thick,red] (2,0) -- (2,-1);
    \draw[->,very thick,red] (2,-1) -- (1,-1);
    \draw[->,very thick,red] (1,-1) -- (1,0);
  \end{tikzpicture}
  \caption{A SAW which cannot be extended to form a longer SAW.}\label{norec}
 \end{center}
\end{figure}

We call the class of self avoiding walks $\cA$, and they are defined as in the beginning of this chapter. There is a problem though: this class isn't amenable to recursion. That is, not all SAWs can be added to in order to create a SAW of longer length. Figure \ref{norec} shows an example of a SAW on $\kS = \{\textbf{N,~E,~S,~W}\}$ which cannot be extended to create a SAW of longer length: any extension of the path will violate the self avoidance restriction. A lack of a recursive structure means that we cannot build a specification and hence a functional equation satisfied by a generating function, as is done in Chapter \ref{ac}. Thus, many of the enumerative results on SAW are empirical, including the asymptotic expression for the number of walks of length $n$:
\[
 a(n) \sim \kappa (2.638)^nn^{11/32},
\]
where $\kappa$ is a positive constant \cite{GuCo01}.

If we make restrictions on the class, simplified models can be created which either  are automatically self avoiding (directed paths) or are able to be attacked recursively with self avoidance enforced (partially directed paths). These simplifications allow researchers to work on self avoiding combinatorial models in the usual way, finding rigourous results on related problems and forging new techniques for attacking an old, stubborn problem.

Even very simple combinatorial models have been succesfully applied in this way. One example is the application of directed paths to the statistical mechanics of linear polymers subject to a force, such as in an atomic force microscopy (AFM) experiment. We show an example of a directed path in Figure \ref{dirp}. Atomic force microscopes are able to interfere with a substance at an atomic scale, so in linear polymer experiments they would dip into a solution and hold on to a polymer by a single monomer, applying a force in a vertical or horizontal direction, making the polymer take a shape similar to one of these types of paths. The methodology is to analyse combinatorial models for phase transition and entropic information through asymptotic analysis of associated generating functions, and confirm with AFM experiments \cite{WhIl12}.

\begin{figure}[h]
 \begin{center}
  \begin{tikzpicture}[scale = 1]
  \draw[->, thick] (0,0) -- (8,0);
  \draw[<->, thick] (0,-2) -- (0,2);
  \draw[-, dotted] (0,1) -- (8,1);
  \draw[-, dotted] (0,2) -- (8,2);
  \draw[-, dotted] (1,-2) -- (1,2);
  \draw[-, dotted] (2,-2) -- (2,2);
  \draw[-, dotted] (3,-2) -- (3,2);
  \draw[-, dotted] (0,-1) -- (8,-1);
  \draw[-, dotted] (0,-2) -- (8,-2);
  \draw[-, dotted] (0,-2) -- (8,-2);
  \draw[-, dotted] (4,-2) -- (4,2);
  \draw[-, dotted] (5,-2) -- (5,2);
  \draw[-, dotted] (6,-2) -- (6,2);
  \draw[-, dotted] (7,-2) -- (7,2);
  \draw[-, dotted] (8,-2) -- (8,2);
  \draw[->, very thick,red] (0,0) -- (1,1);
  \draw[->, very thick,red] (1,1) -- (2,2);
  \draw[->, very thick,red] (2,2) -- (3,2);
  \draw[->, very thick,red] (3,2) -- (4,1);
  \draw[->, very thick,red] (4,1) -- (5,0);
  \draw[->, very thick,red] (5,0) -- (6,-1);
  \draw[->, very thick,red] (6,-1) -- (7,-2);
  \draw[->, very thick,red] (7,-2) -- (8,-1);
  \end{tikzpicture}
  \caption{A directed path on $\kS = \{ \textbf{NE,~E,~SE}\}$.}\label{dirp}
 \end{center}
\end{figure}

We also find rigorous bounds on the exponential growth of SAW by considering a subclass. Take the class $\cP$ of partially directed paths on the square lattice. This is a walk taken on the set$~\kS = \{ \textbf{N,~E,~S}\}$ with self avoidance enforced. An example is shown in Figure \ref{pdp}. 

\begin{figure}[]
 \begin{center}
  \begin{tikzpicture}[scale = 1]
  \draw[->, thick] (0,0) -- (6,0);
  \draw[<->, thick] (0,-3) -- (0,3);
  \draw[-, dotted] (0,1) -- (6,1);
  \draw[-, dotted] (0,2) -- (6,2);
  \draw[-, dotted] (0,3) -- (6,3);
  \draw[-, dotted] (1,-3) -- (1,3);
  \draw[-, dotted] (2,-3) -- (2,3);
  \draw[-, dotted] (3,-3) -- (3,3);
  \draw[-, dotted] (0,-1) -- (6,-1);
  \draw[-, dotted] (0,-2) -- (6,-2);
  \draw[-, dotted] (0,-3) -- (6,-3);
  \draw[-, dotted] (4,-3) -- (4,3);
  \draw[-, dotted] (5,-3) -- (5,3);
  \draw[-, dotted] (6,-3) -- (6,3);
  \draw[->, very thick,red] (0,0) -- (0,1);
  \draw[->, very thick,red] (0,1) -- (0,2);
  \draw[->, very thick,red] (0,2) -- (0,3);
  \draw[->, very thick,red] (0,3) -- (1,3);
  \draw[->, very thick,red] (1,3) -- (1,2);
  \draw[->, very thick,red] (1,2) -- (1,1);
  \draw[->, very thick,red] (1,1) -- (1,0);
  \draw[->, very thick,red] (1,0) -- (1,-1);
  \draw[->, very thick,red] (1,-1) -- (1,-2);
  \draw[->, very thick,red] (1,-2) -- (1,-3);
  \draw[->, very thick,red] (1,-3) -- (2,-3);
  \draw[->, very thick,red] (2,-3) -- (2,-2);
  \draw[->, very thick,red] (2,-2) -- (2,-1);
  \draw[->, very thick,red] (2,-1) -- (3,-1);
  \draw[->, very thick,red] (3,-1) -- (3,-2);
  \draw[->, very thick,red] (3,-2) -- (3,-3);
  \draw[->, very thick,red] (3,-3) -- (4,-3);
  \draw[->, very thick,red] (4,-3) -- (4,-2);
  \draw[->, very thick,red] (4,-2) -- (4,-1);
  \draw[->, very thick,red] (4,-1) -- (4,0);
  \draw[->, very thick,red] (4,0) -- (4,1);
  \draw[->, very thick,red] (4,1) -- (5,1);
  \draw[->, very thick,red] (5,1) -- (5,0);
  \draw[->, very thick,red] (5,0) -- (5,-1);
  \draw[->, very thick,red] (5,-1) -- (5,-2);
  \draw[->, very thick,red] (5,-2) -- (6,-2);
  \end{tikzpicture}
  \caption{A partially directed path provides a lower bound on the exponential growth of SAW.}\label{pdp}
  \end{center}
\end{figure}

The walks in $\cP$ form a subclass of general SAW on the square lattice, and can be solved explicitly, with generating function
\[
 P(t) = \frac{t(1-t)}{1 - 2t - t^2}.
\]
The dominant singularity for this generating function is a simple pole at $\sqrt2 - 1$. As will be explained in Chapter \ref{ac}, we can use singularity analysis to compute
\[
 p(n) \sim \kappa \left(\frac{1}{\sqrt2 - 1}\right)^n \approx \kappa (2.414)^n.
\]
Using the techniques we discuss throughout this thesis, we can use this asymptotic estimate to provide a rigorous lower bound on the exponential growth factor for all SAW on the square lattice. 

Combinatorics could be used to analyse other thermodynamic properties of a system. At any moment, a system is distributed across an ensemble of $N$ microstates, each denoted by $i$ and having a probability of occupation $p_i$, and an energy $E_i$. These microstates form a discrete set as defined by quantum statistical mechanics, and $E_i$ is an energy level of the system. One may then look at the internal energy of a system, which is the mean of the energies $E_i$, or the heat and work of a system. Moreover, lattice statistical mechanics is not restricted to systems modeled by lattice paths \cite{va00}. This makes the interface between combinatorics and statistical mechanics a rich junction, with combinatorial approaches not limited to enumeration.

\chapter{Analytic combinatorics}\label{ac}

When trying to enumerate a family of objects which is amenable to recursion, like lattice paths, a standard tool is the generating function. These were originally introduced as formal power series, or as Wilf colourfully puts it in \cite{Wi94}, a `clothesline to hang a sequence upon,' these tools have proved very useful in many applications. Stanley has written two volumes \cite{St99,St11} on the topic of enumerative combinatorics in general, covering applications of generating functions to algebraic structures as well as the algebraic nature of the objects themselves, providing in \cite{St99} a classification of formal power series which we shall employ a little later on.

In \cite{FlSe09}, Flajolet and Sedgewick take the field much further, considering the formal power series inside their radii of convergence as complex analytic functions: no longer just algebraically. This allows the complex analysis theorems of Cauchy to be applied with impressive enumerative results, obtaining asymptotic expressions for the coefficient sequences of the generating function with little effort (once the theorems have been proved). They describe their approach as an operational calculus for combinatorics organised around three components: symbolic methods, complex asymptotics and random structures. The following gives an introduction to the symbolic methods and complex asymptotics, drawing on elements from \cite{FlSe09,Wi94,St11}.

\section{Combinatorial classes and ordinary generating functions}
We follow the exposition of the symbolic methods given in \cite{FlSe09}, concentrating on the points useful to us. The first of these will be the combinatorial class. Earlier, we mentioned that we aim to enumerate, either exactly or asymptotically, subclasses of $\cW_\kS$, the class of planar lattice walks on a given step set $\kS$.

\begin{defn}
A \textit{combinatorial class} or simply \textit{class}, is a finite or denumerable set on which a size function is defined, satisfying the following conditions:
\begin{enumerate}
 \item the size of an element is a non-negative integer;
  \item the number of  elements of any given size is finite.
\end{enumerate}
\end{defn}

If $\cA$ is a class and $\alpha$ is an object in the class, we denote the size of $\alpha$ by $|\alpha|$ or $|\alpha|_\cA$ when more than one class is under consideration. For a lattice walk $w \in \cW$, $|w|$ is the length of the path, or the number of steps taken. For a class $\cA$, we usually denote by $\cA_n$ the set of objects in $\cA$ which have size $n$, and by $a_n$ the cardinality of $\cA_n$.

\begin{defn}
The sequence $(a_n)$ is known as the \textit{counting sequence} of the combinatorial class $\cA$.
\end{defn}

The notion of isomorphism presents itself in this context as follows.

\begin{defn}
Two combinatorial classes, $\cA$ and $\cB$ are said to be (combinatorially) \textit{isomorphic}, written $\cA \cong \cB$, if and only if their counting sequences are identical, that is, $a_n = b_n$ for every $n \geq 0$. This condition is equivalent to a bijection from $\cA$ to $\cB$ which preserves size, and we say that $\cA$ and $\cB$ are bijectively equivalent.
\end{defn}

It should be noted that while two classes may be equinumerous, and able to be shown to be isomorphic, it is not always easy or possible to find a nice or natural transformation between two classes that bijectively maps one to another. The enumerative information about a class $\cA$ is stored in a formal power series $A(t)$.

\begin{defn}
The \textit{ordinary generating function (OGF)} of a sequence $(a_n)_n$ is the formal power series
\[
 A(t) = \sum_{n=0}^{\infty} a_nt^n.
\]
The OGF for a combinatorial class $\cA$ is the generating function for the counting sequence $a_n = |\cA_n|$, with the variable $t$ marking size.
\end{defn}

There are two special classes: the empty class and the atomic class. The empty class, denoted by $\cE$, is the unique class (up to isomorphism) containing one object of size 0 and has OGF $E(t) = 1$. The atomic class, denoted sometimes by $\cZ$, is the class containing one object of size one with OGF $Z(t) = t$. In some applications we differentiate between atoms: for instance, to represent steps in different directions. When we differentiate a \textbf{NE} step and a \textbf{SE} step, we would use $\cZ_{NE}$ and $\cZ_{SE}$.

Just a brief note on terminology: since the counting sequence for the class $\cA$ is also the coefficient sequence for the generating function $A(t)$, we shall be using both terms when referring to the sequence $(a_n)$.

Note the naming convention at work here: a class $\cA$ has subclasses $\cA_n$ of elements of size $n$, each with cardinality $a_n$, and generating function $A(t)$. Also, sometimes we would like to find the coefficient of the $t^n$ term in the power series. We call this action coefficient extraction, and denote it by
\[
 [t^n]A(t) = a_n.
\]

Generating functions are manipulated algebraically as formal power series. Formally, they are elements of the ring of formal power series $\CC[\![x]\!]$. The operations of sum and Cauchy product in the ring of formal power series are interpreted quite naturally as operations on the combinatorial structures themselves. Let's investigate these operations: let $\cA$ and $\cB$ be a pair of combinatorial classes with no intersection. We define two binary operations on $\cA$ and $\cB$: combinatorial sum ($+$) and combinatorial product ($\bullet$). 

We define the sum of disjoint classes $\cA$ and $\cB$ like so:
\[
 \cA + \cB := \cA \cup \cB.
\]
Inside our new unified class, the class of objects of size $n$ is the union $\cA_n \cup  \cB_n$. Since $\cA \cap\cB = \emptyset$ the cardinality of $\cA_n \cup \cB_n$ is simply $a_n + b_n$. We call the new sum class $\cC$, which has OGF
\[
 C(t) = \sum_{n \geq 0} (a_n + b_n)t^n,
\]
which is clearly the sum of the OGFs $A(t)$ and $B(t)$.

Next, we define the product of $\cA$ and $\cB$ as the cartesian product:
\[
 \cA \bullet \cB := \cA \times \cB.
\]
The elements of the product class are ordered pairs $(\alpha,\beta)$, and the size of an ordered pair is defined to be the sum of the sizes of its entries. So if $|\alpha| = 2$ and $|\beta| = 4$, then $|(\alpha,\beta)| = 6$. We call the new product class $\cP$. Let us look again at the subclass $\cP_n$ of objects of size $n$. The elements are ordered pairs $(\alpha,\beta)$ with, for each $k = 0,...,n$, $|\alpha| = k$ and $|\beta| = n-k$. For each $k$ there are $a_kb_{n-k}$ such ordered pairs, so the generating function for this product class is
\[
 P(t) = \sum_{n \geq 0} \left( \sum_{k = 0}^n a_kb_{n-k} \right) t^n,
\]
which is precisely the Cauchy product of the generating functions $A(t)$ and $B(t)$.

We may now extend the definition of the sum to include classes that share some elements. Suppose that $\cA$ and $\cB$ did have a non-empty intersection. Then we would simply take
\[
 \cA + \cB := \cA \sqcup \cB,
\]
where $\sqcup$ is a disjoint union of the two classes. We take this by representing each class as a product: $\cA \cong \cA \times \cE$ and $\cB \cong \cE \times \cB$. Since $\cE$ is the empty class, this transformation adds no weight to the class (hence the congruence asserted above is legitimate) and the disjoint union still corresponds to the sum of the OGFs.

For many of our cases, we will not be finding a generating function, but a functional equation that the generating function satisfies. Take for example the geometric series
\[
 W(t) = \sum_{n \geq 0} 2^nt^n.
\]
This is the generating function for the number of unrestricted lattice walks on the step set~$\{\textbf{NE,~SE}\}$. This satisfies the algebraic functional equation
\[
 W(t) = 1 + 2tW(t),
\]
which is an algebraic equation with polynomial coefficients. Thus the solution, the generating function, is an algebraic function given by the sum of a geometric progression
\[
 W(t) = \frac{1}{1-2t}.
\]
Solutions to algebraic equations exist and are unique up to multiplicity of roots, and this fact gives us a quick way of checking an isomorphism between two classes. If two classes have generating functions which satisfy the same functional equation, then the coefficient sequences of each will satisfy the same recursion. To prove isomorphism, all that remains to be done is a comparison of the first few terms of each sequence. A good example where the isomorphism fails due to a different sequence are the neutral class,  with GF $E(t) = 1$, and atomic class, with GF $Z(t) = t$, which both satisfy the same algebraic equation $A(t)^2 = A(t)$ but are clearly not equinumerous.

\begin{ex}\label{dyck}
(Dyck paths.) We can use the theory developed so far to investigate what is probably the most famous example of a class of lattice paths, $\cD$, the class of Dyck paths. These are paths on the step set $\kS = \{\textbf{NE,~SE}\}$, beginning at the origin and ending on the $x$-axis, without leaving the first quadrant - in the nomenclature of \cite{BaFl02} they are an example of a directed excursion. A Dyck path of length 16 is shown in Figure \ref{D16}. 

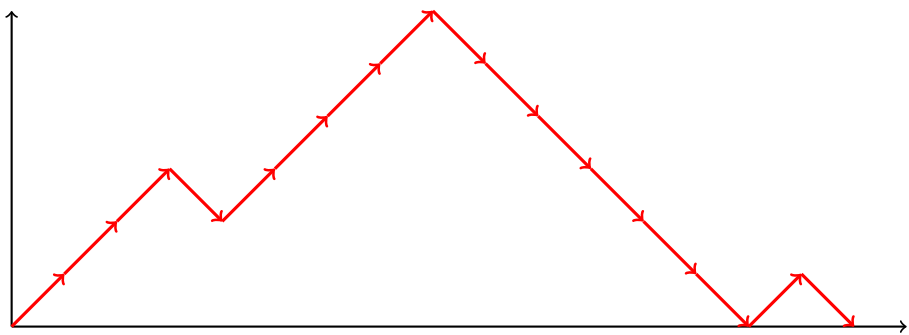
\begin{figure}[h]
 \begin{center}
  \begin{tikzpicture}[scale=0.7]
    \draw[->,thick] (0,0) -- (17,0);
    \draw[->,thick] (0,0) -- (0,6);
    \draw[->,very thick,red] (0,0) -- (1,1);
    \draw[->,very thick,red] (1,1) -- (2,2);
    \draw[->,very thick,red] (2,2) -- (3,3);
    \draw[->,very thick,red] (3,3) -- (4,2);
    \draw[->,very thick,red] (4,2) -- (5,3);
    \draw[->,very thick,red] (5,3) -- (6,4);
    \draw[->,very thick,red] (6,4) -- (7,5);
    \draw[->,very thick,red] (7,5) -- (8,6);
    \draw[->,very thick,red] (8,6) -- (9,5);
    \draw[->,very thick,red] (9,5) -- (10,4);
    \draw[->,very thick,red] (10,4) -- (11,3);
    \draw[->,very thick,red] (11,3) -- (12,2);
    \draw[->,very thick,red] (12,2) -- (13,1);
    \draw[->,very thick,red] (13,1) -- (14,0);
    \draw[->,very thick,red] (14,0) -- (15,1);
    \draw[->,very thick,red] (15,1) -- (16,0);
  \end{tikzpicture}
 \end{center}
  \caption{A Dyck path of length 16.} \label{D16}
\end{figure}

These paths are well known, and the number of Dyck paths of length $2n$ is known to be the $n$th Catalan number, 
\[
 C_n = \frac1{n+1} \binom{2n}{n},
\]
so the OGF for $\cD$ is given by
\[
 D(t) = \sum_{n \geq 0} \frac1{n+1}\binom{2n}{n} t^{2n}.
\]

We can use the operations of sum and product to find a functional equation satisfied by the OGF for $\cD$, which we may then solve and use a coefficient extraction to find the Catalan numbers. This is done by decomposing a Dyck path into 2 parts, using what is known as the inital pass decomposition. The decomposition is the following: a member of the class~$\cD$ is either the empty path or a path of non-zero length. If it is a path of non-zero length, after the point of contact with the $x$-axis formed by the beginning of the path, there will be at least one more point of contact. This is known as the initial pass, and is indicated in Figure \ref{Dip} by a blue circle.

Prior to the initial pass there are only two steps in contact with the $x$-axis, which are shown in blue in Figure \ref{Dip}. Between these steps is a shorter (possibly empty) Dyck path beginning and ending on the dashed blue line. After the initial pass there will be another (again, possibly empty) Dyck path beginning and ending on the $x$-axis.

\begin{figure}[h]
 \begin{center}
  \begin{tikzpicture}[scale=0.7]
    \draw[->,thick] (0,0) -- (17,0);
    \draw[->,thick] (0,0) -- (0,6);
    \draw[-,dashed,blue] (0,1) -- (14,1);
    \draw[fill=blue,blue] (14,0) circle (4pt);
    \draw[->,very thick,blue] (0,0) -- (1,1);
    \draw[->,very thick,red] (1,1) -- (2,2);
    \draw[->,very thick,red] (2,2) -- (3,3);
    \draw[->,very thick,red] (3,3) -- (4,2);
    \draw[->,very thick,red] (4,2) -- (5,3);
    \draw[->,very thick,red] (5,3) -- (6,4);
    \draw[->,very thick,red] (6,4) -- (7,5);
    \draw[->,very thick,red] (7,5) -- (8,6);
    \draw[->,very thick,red] (8,6) -- (9,5);
    \draw[->,very thick,red] (9,5) -- (10,4);
    \draw[->,very thick,red] (10,4) -- (11,3);
    \draw[->,very thick,red] (11,3) -- (12,2);
    \draw[->,very thick,red] (12,2) -- (13,1);
    \draw[->,very thick,blue] (13,1) -- (14,0);
    \draw[->,very thick,red] (14,0) -- (15,1);
    \draw[->,very thick,red] (15,1) -- (16,0);
  \end{tikzpicture}
 \end{center}
  \caption{The initial pass decomposition of our Dyck path} \label{Dip}
\end{figure}

This is summarised symbolically in terms of $\cD$, the empty class $\cE$ and the atomic class~$\cZ$ as
\[
 \cD = \cE + \cZ_{NE} \bullet \cD \bullet \cZ_{SE} \bullet \cD.
\]
We then use the definitions of combinatorial sum and product to interpret this in the ring of formal power series containing the OGF for $\cD$:
\[
 D(t) = 1 + tD(t)tD(t) = 1 + t^2D(t)^2.
\]

Still working with $D(t)$ as a formal algebraic object, we solve the equation, a quadratic in the OGF $D(t)$, using the quadratic formula:
\[
 D(t) = \frac{1 \pm \sqrt{1 - 4t^2}}{2t^2}.
\]
We may reject the branch
\[
 \frac{1 + \sqrt{1 - 4t^2}}{2t^2},
\]
since when we develop its power series (by taking Newton's expansion of $(1 - 4t^2)^{1/2}$) we find that the corresponding power series is singular at 0 and contains negative coefficients. Thus, we find that
\[
 D(t) = \frac{1 - \sqrt{1 - 4t^2}}{2t^2},
\]
which we may expand to its power series form by Newton's generalised binomial expansion.
\end{ex}

Finding generating functions as solutions to algebraic equations brings up the question of classification. As with numbers, formal power series which are the solution to algebraic equations are known as algebraic functions. Two examples are the formal power series $W(t)$ and $D(t)$ that we have seen already. In \cite{St99}, Stanley gives a classification of power series into a hierarchy, which we shall discuss presently.

\subsection*{Classification of OGFs}

Let $G$ denote an arbitrary formal power series with coefficients taken from a field $K$. Then~$G$ is an element of the ring of formal power series $K[\![t]\!]$. The power series $G$ is classified as either \textit{rational}, \textit{algebraic}, \textit{transcendental D-finite} or \textit{non D-finite}, terms which will be defined in due course. In Chapter \ref{qpwalks}, we find that a family of D-finite generating functions belong to a smaller class known as G-series, but we don't talk about these in this chapter.

The smallest class is that of rational power series, or those which can be represented as the ratio of two polynomials.

\begin{defn}\label{ratgf}
A formal power series $G \in K[\![t]\!]$ is \textit{rational} if there exist polynomials $Q(t),P(t) \in K[t]$, with $Q(t) \neq 0$, such that
\[
 G = \frac{P(t)}{Q(t)}.
\]
\end{defn}

\begin{ex}
The power series
\[
 W(t) = \frac1{1-2t}
\]
is an example of a rational series.
\end{ex}

Note that a rearrangement of the equation in Definition \ref{ratgf} gives a linear equation for which $G$ is the solution. The definition can be extended to equations that are of higher degree in $G$, giving the general definition of an algebraic power series.

\begin{defn}
A formal power series $G \in K[\![t]\!]$ is \textit{algebraic} if there exist polynomials~$P_0(t),P_1(t),...,P_d(t)\in K[t]$, not all identically 0, such that
\[
 P_0(t) + P_1(t)G + ... + P_d(t)G^d = 0.
\]
The minimal integer $d$ for which such an equation holds is the \textit{degree} of $G$.
\end{defn}

\begin{ex}
As we've already seen, the generating function
\[
 D(t) = \frac{1 - \sqrt{1 - 4t^2}}{2t^2}
\]
for Dyck paths satisfies the algebraic equation
\[
 1 - D(t) + t^2D(t)^2 = 0,
\]
making it an algebraic series of degree 2.
\end{ex}

The class of algebraic series sits inside a larger class, known as \textit{differentiably finite} power series, shortened to \textit{D-finite}. They are also known as \textit{holonomic} functions: a power series is holonomic if it satisfies a linear homogeneous differential equation, which we shall see is very similar to D-finite.

\begin{defn}
Let $G \in K[\![t]\!]$. If there exist polynomials $p_0(t),...,p_d(t) \in K[t]$ with $p_d(t) \neq 0$ such that
\begin{equation}\label{dfin}
 p_d(t)G^{(d)} + p_{d-1}(t)G^{(d-1)} + ... + p_1(t)G' + p_0(t)G = 0,
\end{equation}
where $G^{(j)} = \frac{d^jG}{dt^j}$, then we say that $G$ is a D-finite power series.
\end{defn}

Algebraic functions are D-finite, by the following theorem, taken from Stanley's \cite{St99}, along with a sketch of the proof.

\begin{thm}
\cite[Theorem 6.4.6]{St99}Let $u \in K[[x]]$ be algebraic of degree $d$. Then $u$ is D-finite.
\end{thm}

\begin{proof}
Let $u$ be as stated. Since $u$ is algebraic, there is some degree $d$ polynomial $P(x,y) \in K[x,y]$ such that $P(x,u) = 0$. Thus
\[
 0 = \frac{d}{dx} P(x,u) = \left. \frac{\partial P(x,y)}{\partial x}\right|_{y = u} +  u' \left. \frac{\partial P(x,y)}{\partial y}\right|_{y = u}.
\]
Since $P(x,y)$ is of minimal degree,
\[
 \left. \frac{\partial P(x,y)}{\partial y} \right|_{y=u} \neq 0,
\]
and so
\[
 u' = - \frac{\left. \frac{\partial P(x,y)}{\partial x}\right|_{y = u}}{\left. \frac{\partial P(x,y)}{\partial y}\right|_{y = u}}.
\]
That is, $u'$ is a rational function in $x$ and $u$, so is an element of the field $K(x,u)$. We can continue inductively to show that for any $k \geq 0$, $u^{(k)} \in K(x,u)$. Now, $u$ is algebraic of degree~$d$, so $\dim_{K(x)} K(x,u) = d$, and therefore $u,u',...,u^{(d)}$ are linearly dependent over~$K(x)$, giving an equation of the same form as Equation (\ref{dfin}), proving that $u$ is D-finite.
\end{proof}

Let's consider a brief example of a D-finite function that is not algebraic.

\begin{ex}
An example of a D-finite function is $G(t) = t^2e^t$, since $G'(t) = (1 + \frac2t)G$. However, $t^2e^t$ is not algebraic. This shows that the algebraic functions form a proper subclass of the D-finite functions.
\end{ex}

Finally, if a power series $G \in K[\![t]\!]$ is not D-finite, we simply say that $G$ is non D-finite. Proof that a generating function is non-D-finite can be difficult, see \cite{MiRe09} for two lattice path examples from a family for which classification is ongoing.

The two expressions
\begin{eqnarray*}
 W(t) &=& \frac{1}{1-2t}, \\
  D(t) &=& \frac{1 - \sqrt{1 - 4t^2}}{2t^2}
\end{eqnarray*}
raise an interesting point about these generating functions: we may view them as standard complex analytic objects by assigning complex values to the variable $t$. Both series will converge to analytic functions for $|t| < \frac12$. The exploitation of the complex analytic properties of the generating functions for enumerative results will form the focus of Section~\ref{asympt}.

The discussion so far has focused on ordinary generating functions, those with a single variable. Lattice paths require something more flexible than this, and Section~\ref{mgf} introduces the required ideas.

\section{Multivariate generating functions}\label{mgf}
The formal power series framework is easily adapted to a multivariate model, where multiple parameters are introduced. When working on the main problem, we use a trivariate power series with parameters keeping track of the end point of a walk. We show a simplified version: a bivariate generating function (BGF) which marks both the length of a walk and the $y$-coordinate of its endpoint, which we call the stopping height.

\begin{ex}
We look at the class $\cW$ of walks taken on the step set $\{ \textbf{NE,~SE}\}$, but unrestricted in where it may visit and end. Walks can be built recursively from shorter ones by simply appending a new step, and this can be done in two ways. Symbolically, we have
\[
 \cW = \cE + \cW \bullet \cZ_{NE}  + \cW \bullet \cZ_{SE},
\]
which we may give as the familiar functional equation
\[
 W(t) = 1 + 2tW(t).
\]
We keep track of the stopping height, and do so by introducing a parameter, $u$. Note that to keep track of height, we must increment the exponent of $u$ by 1 for an upward step, and decrement by 1 for a downward step. To capture this, we move straight into the ring of formal power series, rather than using a symbolic specification. The only difference is that now $K = \CC(u)$ because the walks are allowed to visit points with a negative $y$-value. The corresponding algebraic equation is
\[
 W^*(t,u) = 1 + utW^*(t,u) + \frac{t}{u}W^*(t,u),
\]
which we may solve to find the MGF as the geometric series
\[
 W^*(t,u) = \frac{1}{1 - t\left(u + \frac1u\right)},
\]
where $t$ marks length and $u$ marks stopping height. We can then perform a coefficient extraction with respect to $t$ to find a Laurent polynomial in $u$
\[
 [t^n]W^*(t,u) = \left(u + \frac1u\right)^n,
\]
where the coefficient of $u^k$ is the number of walks of length $n$ with stopping height $k$. To perform the coefficient extraction, we use the following fact:
\begin{equation}\label{coefftrans}
 [u^{k-t}] f(u) = [u^{k}]u^tf(u).
\end{equation}
This allows us to simplify things a little when expanding the binomial. Note that the range of values possible for $k$ is the set $\{-n,-n+1,...,n-1,n\}$, but to simplify calculations we multiply by $u^n$ (as in Equation \ref{coefftrans}) so the range of interest in the new function is $[[0,2n]]$.  So, for $k \in [[0,2n]]$ we have
\[
 [u^{k-n}] \left(u + \frac1u\right)^n = [u^k](u^2 + 1)^n.
\]
This is a binomial with a quadratic term, so for $k \equiv 1 \pmod{2}$, we have
\[
 [u^{k}](u^2 + 1)^n = 0,
\]
and for $k \equiv 0 \pmod{2}$, we have
\[
 [u^{k}](u^2 + 1)^n = \binom{n}{\frac{k}{2}}.
\]
So for $k \in \{0,1,2,...,2n\}$ the coefficients for our MGF are
\[
 [t^nu^{k - n}] W^*(t,u) = \left\{ \begin{array}{ll} 0, & \mbox{if } k \equiv 1 \pmod{2} \\ \binom{n}{k/2} & \mbox{if } k \equiv 0 \pmod{2}. \end{array}\right.
\]
\end{ex}

The classification of formal power series also applies in the multivariate case. Now we take $G \in K[[t_1,...,t_m]]$, and classify it at as rational, algebraic, D-finite or non D-finite over the field $K$. The first definition is the same as the single variable case: $G$ is rational if there exist polynomials~$P,Q \in K[t_1,...,t_n]$ with $Q \neq 0$ such that $QG = P$.

The definition for algebraic is also largely the same. If there exist polynomials~$P_0,...,P_d \in K[t_1,...,t_m]$ not all 0 such that $P_0 + P_1G + ... + P_dG^d = 0$, then we say that $G$ is algebraic. If~$d$ is the smallest integer such that this is true, then $G$ is algebraic of degree $d$.

Finally, the generalisation of the D-finite definition requires more. A power series $G \in K[[t_1,...,t_n]]$ is D-finite over $K(t_1,...,t_m)$ if for $1 \leq i \leq m$, $G$ satisfies a system of non-trivial partial differential equations of the form
\[
 \sum_{j = 0}^{d_i} P_{j,i} \frac{\partial^j G}{\partial t_i^j} = 0,
\]
where $P_{j,i} \in K[t_1,...,t_m]$ and $d_i$ is the order of the partial differential equation in $t_i$.

\section{Coefficient asymptotics}\label{asympt}

For classes for which an exact counting formula is known, real analysis is usually enough to obtain asymptotic expressions for the counting sequence, as we show with the Catalan numbers.

\begin{ex}\label{stirlingcat}
We show how to get asymptotic expressions from a formula by applying Stirling's approximation to the formula for the Catalan numbers. Recall from real analysis that Stirling's approximation says that
\[
 n! ~ \sim \sqrt{2 \pi n} \left( \frac{n}{e} \right)^n.
\]
We apply this to the formula
\[
 C_n = \frac{1}{n+1} \binom{2n}{n},
\]
while noting that $n+1 \sim n$. So, we get
\begin{eqnarray*}
 C_n 
  &=& \frac{1}{n+1} \binom{2n}{n}, \\
  &\sim& \frac1n \frac{(2n)^{2n} e^{-2n} \sqrt{4\pi n}}{n^{2n} e^{-2n} 2\pi n}, \\
  &\sim& \frac{2^{2n}}{\sqrt{\pi n^3}}.  
\end{eqnarray*}
\end{ex}

Next is the class of problems for which a generating function in a closed form is obtainable, but no closed form expression is available for the counting sequence. It usually suffices to have an expression of the generating function as an analytic function, not a formal power series, to obtain asymptotic expressions. The analysis depends on the local properties of the generating function at its dominant singularities. The exposition of this theory is the focus of this section. To continue, we recall the definition of asymptotic growth.

\begin{defn}
Let $a:\NN \rightarrow \NN$ be a function (such as the coefficient sequence of a generating function associated with a combinatorial class). We say that $a_n$ \textit{grows like} or has \textit{asymptotic growth} $g(n)$, $g: \RR^+ \rightarrow \RR^+$, denoted 
\[
 a(n) \sim g(n)
\]
if for every $0 < \epsilon < 1$ we can find a sufficiently large $N$ such that for $n > N$ the following holds
\[
 1 - \epsilon \leq \frac{a(n)}{g(n)} \leq 1 + \epsilon.
\]
\end{defn}

As we can see in Example \ref{stirlingcat} and for our rational function example, the coefficients $f_n$ of a generating function $F$ belong to a general asymptotic regime. That is
\[
 [t^n] F(t) \sim \beta^n \theta (n),
\]
where $\beta^n$ is an \textit{exponential growth} factor modulated by a \textit{subexponential} factor $\theta(n)$. For Dyck paths, the exponential growth factor is $4$, but recall that we count only even length paths so we may replace $4^n$ by $2^{2n}$, and for the rational function, the coefficients are $2^n$. Recall that in both cases, the radius of convergence is $\frac12$. This (tenuously) suggests a link between the radius of convergence of a generating function and the exponential growth of the coefficient sequence, and this turns out to be the case.

We begin by quoting from \cite{FlSe09} the two principles of coefficient asymptotics.
\begin{framed}
\begin{quote}
 \textbf{The first principle of coefficient asymptotics:} The \textit{location} of a function's singularities dictates the \textit{exponential growth} of its coefficients.
\\
  \textbf{The second principle of coefficient asymptotics:} The \textit{nature} of a function's singularities determines the associated \textit{subexponential factor}.
\end{quote}
\end{framed}

So, rather than viewing the generating functions as formal algebraic structures, we begin to view them as analytic functions of a complex variable. These can in turn be viewed in two equivalent forms: as convergent power series or as differentiable functions. The first form directly relates to the use of generating functions for enumeration.

We assume a basic understanding of complex analysis: the definitions of analytic, holomorphic and meromorphic functions are left to more focused texts \cite{La99}, as well as the definition of a singularity. In order to use information from the singularities in the asymptotic expressions for the coefficients of the generating function, we need to relate the global properties of the function near 0 - the asymptotic expansions - to the local properties of the function away from zero - the singularites.

Cauchy's residue theorem and Cauchy's coefficient formula are both examples of theorems which relate global and local properties of functions through the use of contour integrals, with contours taken around the singularities of a function. In practice, Cauchy's coefficient formula is often used to find bounds on the asymptotic growth of a generating functions coefficient sequence, as in \cite[Theorem 3]{BaFl02}. Given that we are interested in finding exponential growth factors, we need only consider singularities of a generating function which lie at the boundary of the disc of convergence. These are called \textit{dominant singularities}. The following theorem of Pringsheim significantly simplifies the search for these in the case of generating functions. We omit the proof, but it can be found in \cite{FlSe09}.

\begin{thm}
(Pringsheim's theorem). If $f(t)$ is representable at the origin by a convergent series expansion that has non-negative coefficients and a radius of convergence $R$, then the point $t = R$ is a singularity of $f(t)$.
\end{thm}

This tells us that a dominant singularity of a generating function will lie on the real line, and allows us to only consider the real singularities of a generating function when determining the radius of convergence. In order to show the relation between the radius of convergence (or location of the dominant real singularity) and the exponential growth of a counting sequence, we first need a definition.

\begin{defn}
We say that a number sequence $(a_n)$ is of \textit{exponential order} $K^n$ which is denoted as
\[
 a_n \bowtie K^n
\]
if and only if 
\[
  \limsup_{n \rightarrow \infty} |a_n|^{\frac1n} = K.
\]
The relation $a_n \bowtie K^n$ expresses both an upper and lower bound: for any $\epsilon > 0$ we have:
\begin{enumerate}
 \item $|a_n| >_{\mbox{i.o}} (K - \epsilon)^n$, that is, $|a_n|$ exceeds $(K - \epsilon)^n$ \textit{infinitely often};
  \item $|a_n| <_{\mbox{a.e}} (K + \epsilon)^n$, that is, $|a_n|$ is dominated by $(K + \epsilon)^n$ \textit{almost everywhere}.
\end{enumerate}
\end{defn}

Often, our sequences have a subexponential growth factor $\theta(n)$ as well. Note that such a factor would be dominated almost everywhere by an increasing exponential $(1 + \epsilon)^n$, and bounded from below infinitely often by any decaying exponential $(1 - \epsilon)^n$. Thus
\[
 \limsup_{n \rightarrow \infty} |\theta(n)|^{\frac1n} = 1,
\]
and if $a_n = K^n \theta(n)$, it still satisfies $a_n \bowtie K^n$. We are now ready to prove the exponential growth formula.

\begin{thm}\label{expgrow}
\cite[Theorem IV.7]{FlSe09}(Exponential growth formula). If $f(z)$ is a generating function, analytic at 0, with a power series representation involving only non-negative coefficients, and $R$ is the modulus of the dominant real singularity, that is
\[
 R = \sup \{ r \geq 0 ~ | ~ f \mbox{ is analytic at all points of } 0 \leq t < r \},
\]
then the coefficient $f_n = [t^n]f(t)$ satisfies
\[
 f_n \bowtie \left( \frac1R \right)^n.
\]
\end{thm}

\begin{proof}
Let $R$ be as in the statement of the theorem. Pringsheim's theorem tells us that $R$ is the radius of convergence. Now, we must prove the exponential growth relation. Since $R$ is the radius of convergence, the power series will converge at any point just to the left of $R$, that is, for $R > \epsilon > 0$
\[
 f_n(R - \epsilon)^n \rightarrow 0.
\]
In other words
\[
 f_n(R - \epsilon)^n < 1
\]
for all $n > N$, given $N$ sufficiently large. This gives
\[
 |f_n|^{\frac1n} < (R - \epsilon)^{-1}
\]
almost everywhere. For property $2$, let $\epsilon > 0$. Then $f_n(R + \epsilon)^n$ cannot be a bounded sequence, otherwise the power series
\[
 \sum_{n \geq 0} f_n (R + \epsilon/2)^n
\]
would converge. Indeed, suppose that $|f_n(R + \epsilon)^n| \leq B$ for some bound $B$. Then 
\[
|f_n|(R + \epsilon/2)^n = |f_n|(R + \epsilon)^n \left( \frac{R + \epsilon/2}{R + \epsilon} \right)^n \rightarrow 0,
\]
since $|f_n|(R + \epsilon)^n$ is bounded and
\[
 \left( \frac{R + \epsilon/2}{R + \epsilon} \right) < 1.
\]
Therefore $f_n^{1/n} > (R + \epsilon)^{-1}$ infinitely often, giving the relation
\[
 f_n \bowtie \left( \frac1R \right)^n.
\]
\end{proof}

This proves the first principle of coefficient asymptotics. The second principle, relating the subexponential growth of the counting sequence to the nature of the singularity, is dependent on the Newton expansion of the generating function near the singularity. We illustrate the previous theorem with some examples.

\begin{ex}
The first example is of lattice paths with the same set of steps as Dyck paths, but no longer restricted to lie in the first quadrant: they can now journey as far south as they would like and end at any place they can reach. We'll call the class of such paths $\cB$. It should be clear that since there are no restrictions, the generating function is 
\[
 B(t) = \sum_{n \geq 0} 2^nt^n = \frac{1}{1 - 2t}.
\]
This clearly has a simple pole at $t = \frac12$, and no other singularities. Applying the exponential growth formula, we find
\[
 b_n = 2^n \theta_B(n),
\]
where $\theta_B(n)$ is a subexponential growth factor. Note that in this case $\theta_B(n) = 1$, since the counting sequence is exactly $2^n$.

The second class is the class of Dyck paths, $\cD$. We already know the generating function:
\[
 D(t) = \sum_{n \geq 0} \frac{1}{n+1} \binom{2n}{n} t^{2n} = \frac{1 - \sqrt{1 - 4t^2}}{2t^2}.
\]
There is an apparent pole at 0, but development of the generating function as a series will show that this is removable. The dominant singularity in fact lies at $t = \frac12$, as in the previous case. This is the branch point of the square root portion of the function: the function $\sqrt{Z}$ cannot be unambiguously defined in a neighbourhood of $Z = 0$. Applying the exponential growth formula gives
\[
 d_{2n} = 2^{2n} \theta_D(2n).
\]
This time, due to the nature of the singularity the subexponential growth factor is a non-trivial function of $n$, which reflects the restriction of the paths to lie in the first quadrant. We already applied Stirling's approximation to the formula for Catalan numbers in Example \ref{stirlingcat}, so we know the asymptotic growth of the function $\theta$ to be
\[
 \theta_D(n) \sim \frac{1}{\sqrt{\pi n^3}}.
\]

The final class we consider is the class of Motzkin paths $\cM$. Like Dyck paths, these are restricted to lie in the first quadrant and begin and end on the $x$-axis, but use the set of steps shown in Figure \ref{motzsteps}.

\begin{figure}[h]
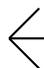

  \begin{center}
  \diagb{NE,E,SE}
  \caption{The steps for Motzkin paths}\label{motzsteps}
  \end{center}
\end{figure}

Like the Dyck paths, we can decompose Motzkin paths. Again, we either have an empty walk, or we take a step. The first step is either horizontal or an upward step. If the first step is horizontal, what follows will be a Motzkin path. If the first step is upwards, we use the first pass decomposition as before. This gives the functional equation satisfied by the generating function
\[
 M(t) = 1 + tM(t) + t^2M(t)^2.
\]
We solve this algebraically to find the algebraic function the generating series converges to inside its radius of convergence:
\[
 M(t) = \frac{1 - t + \sqrt{(1 + t)(1 - 3t)}}{2t^2}.
\]
Since the generating function is analytic at zero, we know that the apparent pole at $t = 0$ is removable. This makes the dominant singularity $t = \frac13$, giving
\[
 M_n = 3^n \theta_M(n).
\]
Again, this time $\theta_M$ is non-constant, and $o(1)$. An application of Stirling's approximation to one of the manifold Motzkin number formulas
\[
 M_n = \sum_{k \geq 0} \binom{n}{2k}C_k,
\]
where $C_k$ is the $k$th Catalan number, shows us that
\[
 \theta_M(n) \sim \frac{3\sqrt3}{2\sqrt{\pi n^3}}.
\]
\end{ex}

The second principle of coefficient asymptotics states that we must look to the nature of the singularities for the subexponential growth of a counting sequence. For many of the cases we consider, a closed form expression for the function is as yet unknown. This makes an application of Stirling's approximation as in Example \ref{stirlingcat} impossible. Therefore, we must turn to the classification of the generating function for asymptotic growth templates. For general algebraic series, Wimp and Zeilberger give us a method in \cite{WiZe85} for finding an asymptotic growth template for the coefficients:
\[
  a_n \sim \kappa\beta^nn^\alpha
\]
where $\beta = \frac1R$ and $\kappa$ are algebraic over $\QQ$ and $\alpha \in \QQ \setminus \{-1,-2,...\}$. That is, a constant, an exponential part and a polynomial part. The more general class of D-finite generating functions also has a template that we may find through the methods of \cite{WiZe85}. This one is more complicated:
\[
 a_n \sim \kappa\beta^ne^{P(n^{(1/r)})}n^{\alpha}(\log(n))^l,
\]
with $l,r \in \NN$, $\kappa,\alpha$ algebraic over $\QQ$ and $P$ a polynomial. 

As shown in \cite{BoKa09}, all the D-finite series we consider are examples of G-series, so as mentioned in the introduction we actually have the simpler growth template
\[
 a_n \sim \kappa \beta^n n^\alpha (\log(n))^\gamma.
\]
The only difference between this and the algebraic case is the logarithmic factor, the exponent of which, $\gamma$, is found to be zero in all D-finite cases that we consider (through methods given in both \cite{WiZe85,FlSe09}). Thus, the growth template for models of interest is the same as the algebraic case.

This is all we say on the subexponential factors, since the exponential factors $\beta$ are what we are more interested in. The reasoning behind this was given in Chapter \ref{statmech}: the exponential growth factors are closely linked to the entropy of the physical system being modeled. This concludes the background information for the thesis. The next section will survey four different lattice path models, beginning with the most simple to analyse in Chapter \ref{undir}.

\part{Planar lattice paths with small steps}\label{plp}
\chapter{Unrestricted walks}\label{undir}

We begin the second part of this thesis by considering the most general case of planar lattice paths with small steps, those completely unrestricted in where they visit. One can quite trivially enumerate this model for any step set, but we wish to pedagogically expose some of the developed machinery at work on something within the grasp of any reader.

Section \ref{udef} will introduce the class of unrestricted planar lattice paths, giving the ordinary and multivariate generating functions and producing the functional equation satisfied by them from a symbolic specification. The functional equation is quite easy to solve, and makes classification in Section \ref{uclass} simple. Finally, we exactly and asymptotically enumerate the class in Section \ref{uenum}.

\section{Definition}\label{udef}

Let $\kS \subseteq \bar{\kS}$ be a set of small steps, and consider lattice paths with steps from $\kS$, beginning at the origin and unrestricted in the regions of the plane that they may visit. This is the class~$\cW_\kS$ of unrestricted lattice paths on $\kS$. A typical example was shown in Figure \ref{walkex}. Taking $w(n)$ to be the number of walks in $\cW_\kS$ of length $n$, we define the ordinary length generating function as
\[
 W_\kS(t) = \sum_{n \geq 0} w(n)t^n.
\]

We can also keep track of the endpoint of a walk. If we let $w(i,j;n)$ be the number of walks of length $n$ ending at the point $(i,j)$, then we can define the MGF for $\cW_\kS$ as
\[
 W_\kS(x,y;t) = \sum_{\substack{(i,j) \in \ZZ^2 \\ n \geq 0}} w(i,j;n)x^iy^jt^n.
\]
Here, $x$ and $y$ mark the respective coordinates of the endpoint. Note that $W(t) \equiv W_\kS(1,1;t)$.

Since this class is completely unrestricted, the recursive structure of the walks is obvious: all walks of length $n+1$ can be created by appending a step from $\kS$ to the end of a walk of length $n$. We can translate this into a specification, using the operations of combinatorial sum and product:
\[
 \cW_\kS = \cE + \kS \cdot \cZ \cdot \cW_\kS,
\]
where the summation is combinatorial in nature. In the next section, we will translate this into the ring of formal power series and classify the generating function. In this ring, the product of the step set and a class will translate into the product of the inventory and the generating function for that class. 

\section{Classification}\label{uclass}

The MGF (and hence the OGF) for $\cW_\kS$ is easily shown to be a rational power series. Using the techniques from Chapter \ref{ac}, the specification given for $\cW_\kS$ in Section \ref{udef} is translated into a multivariate functional equation, translating a combinatorial sum and product into the operations of sum and Cauchy product in the ring of formal power series. This gives
\[
 W_\kS(x,y;t) = 1 + S(x,y) \cdot t \cdot W_\kS(x,y;t),
\]
which we may solve algebraically to find the rational function
\[
 W_\kS(x,y;t) = \frac{1}{1 - S(x,y)t}
\]
which the MGF will converge to inside its radius of convergence. Setting $x=y=1$ gives 
\[
 W_\kS(t) = \frac{1}{1 - |\kS|t}
\]
as the rational function for the OGF.

\begin{ex}
Let $\kS = \diagr{N,SW,SE}$. Then $S(x,y) = y + \frac{1}{xy} + \frac{x}{y}$, and the functional equation satisfied by $W_\kS(x,y;t)$ is
\[
 W_\kS(x,y;t) = 1 + \left(ty + \frac{t}{xy} + \frac{tx}{y} \right)W_\kS(x,y;t),
\]
which gives the rational MGF
\[
 W_\kS(x,y;t) = \frac{1}{1 - \left(ty + \frac{t}{xy} + \frac{tx}{y} \right)}
\]
and rational OGF
\[
 W_\kS(t) = \frac{1}{1 - 3t}.
\]
\end{ex}

\section{Enumeration}\label{uenum}

Note that we can count these directly using the same observation that gave us the recurrence in Section \ref{uclass}: at each step we have $|\kS|$ choices for the next step, so the number of walks is
\[
 w(n) = |\kS|^n.
\]
This allows us to define the OGF immediately as the power series
\[
 W_\kS(t) = \sum_{n \geq 0} |\kS|^nt^n,
\]
which is a geometric series, so has sum
\[
 W_\kS(t) = \frac{1}{1 - |\kS|t}
\]
for $t < \frac{1}{|\kS|}$. Incidentally, this is the radius of convergence of the power series $W_\kS(t)$, and Section \ref{asympt} gives us a method for finding the asymptotic growth of the sequence $w(n)$ using this radius of convergence. But since rational generating functions have geometric power series and thus their coefficient sequences grow purely exponentially, the asymptotic expression is the same as the exact expression.

\chapter{Directed walks}\label{dir}

Directed paths are those in which all steps have a privileged direction of increase. In this instance, we take all steps to be $x$-positive. Since this direction of increase forces self avoidance, this makes these paths naturally suitable for modeling of linear polymers subject to a force, as discussed in Chapter \ref{statmech}. We primarily use them in our analysis of quarter plane walks in Chapter \ref{qpwalks}. 

In this chapter, we discuss results from Banderier and Flajolet's paper \cite{BaFl02} on directed paths which stay above the $x$-axis, known as \textit{directed meanders}. The introduction of a boundary affects the enumerative results for walks on some step sets. This is due to a parameter known as the drift of the walk, which we discuss in Section \ref{direnum}. A common example of a directed meander, known as a Dyck prefix, is shown in Figure \ref{dirmfig}.

\begin{figure}[h]
  \begin{center}
    \begin{tikzpicture}[scale = 0.8]
      \draw[->,thick,black] (0,0) -- (10,0);
      \draw[->,thick,black] (0,0) -- (0,6);
      \draw[-,dotted,black] (0,1) -- (10,1);
      \draw[-,dotted,black] (0,2) -- (10,2);
      \draw[-,dotted,black] (0,3) -- (10,3);
      \draw[-,dotted,black] (0,4) -- (10,4);
      \draw[-,dotted,black] (0,5) -- (10,5);
      \draw[-,dotted,black] (0,6) -- (10,6);
      \draw[-,dotted,black] (1,0) -- (1,6);
      \draw[-,dotted,black] (2,0) -- (2,6);
      \draw[-,dotted,black] (3,0) -- (3,6);
      \draw[-,dotted,black] (4,0) -- (4,6);
      \draw[-,dotted,black] (5,0) -- (5,6);
      \draw[-,dotted,black] (6,0) -- (6,6);
      \draw[-,dotted,black] (7,0) -- (7,6);
      \draw[-,dotted,black] (8,0) -- (8,6);
      \draw[-,dotted,black] (9,0) -- (9,6);
      \draw[-,dotted,black] (10,0) -- (10,6);
      \draw[->,very thick,red] (0,0) -- (1,1);
      \draw[->,very thick,red] (1,1) -- (2,2);
      \draw[->,very thick,red] (2,2) -- (3,1);
      \draw[->,very thick,red] (3,1) -- (4,2);
      \draw[->,very thick,red] (4,2) -- (5,3);
      \draw[->,very thick,red] (5,3) -- (6,4);
      \draw[->,very thick,red] (6,4) -- (7,5);
      \draw[->,very thick,red] (7,5) -- (8,6);
      \draw[->,very thick,red] (8,6) -- (9,5);
      \draw[->,very thick,red] (9,5) -- (10,4);
    \end{tikzpicture}
    \caption{A Dyck prefix of length 10 on $\kS = \{ \textbf{NE,~SE}\}$.}\label{dirmfig}
  \end{center}
\end{figure}

First, a brief note about the step sets used in this chapter. In \cite{BaFl02}, walks are not restricted to just small steps. Steps may in fact have any $y$-coordinate as long as they step one unit to the right. Since the quarter plane models considered in Chapter \ref{qpwalks} are all on sets of small steps, we make the restriction here also. As in the other models, restricting to small steps significantly simplifies the computations necessary (for example, compare Theorem \ref{dmgf} here and Theorem 2 in~\cite{BaFl02}). Future work on these problems include a generalisation of the results to larger steps, see Chapter \ref{conclusion} for a brief discussion.

\section{Definition}\label{ddef}

We call the class of directed meanders $\cF$, as in \cite{BaFl02}. In order to modify our definition from Chapter \ref{def}, we add the restrictions that each step has an $x$-value of 1 and that each point visited by a path lies in the half plane $y \geq 0$. For walks with small steps, this shrinks the largest step set to that shown in Figure \ref{dirss}.

\begin{figure}[h]
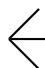

 \begin{center}
  \diagb{NE,E,SE}
 \end{center}
\caption{The set of directed, small steps}\label{dirss}
\end{figure}

For a given set of directed small steps $\kS$, the inventory $S(x,y)$ contains superfluous information: we know that the $x$ value of every step is $1$. Thus we may omit the $x$-coordinate information. In light of this, we define a new \textit{directed inventory}
\[
 P(y) := S(1,y).
\]

\begin{ex}
Let $\kS$ be the full set of directed small steps shown in Figure \ref{dirss}. Then $\kS$ has the inventory $S(x,y) = xy + x + x/y$ and directed inventory $P(y) = y + 1 + 1/y$.
\end{ex}

For a given set of directed small steps $\kS$, let $f(j;n)$ be the number of directed meanders with endpoints having $y$-coordinate, or final altitude, $j$ and length $n$. Then the MGF for directed meanders on $\kS$ is
\[
 F_\kS(y;t) = \sum_{j,n \geq 0} f(j;n)y^jt^n.
\]
Again, these walks can be built recursively: we create directed meanders of length $n+1$ by appending steps from $\kS$ in all possible ways to meanders of length $n$. The phrase `in all possible ways' is a little vague: what we mean by this is that if a walk of length $n$ has its endpoint on the $x$-axis, we may not create a walk of length $n+1$ by appending a $y$-negative step. This last point is difficult to reflect symbolically in terms of combinatorial sum and product, so we immediately define a functional equation based on the recursion:
\begin{equation}\label{fedm}
 F_\kS(y;t) = 1 + tP(y)F_\kS(y;t) - \{y^{<0}\}\left[P(y)\right]\cdot tF_\kS(0;t).
\end{equation}
Note that the final term represents removing the `impossible ways' to increase the length of a shorter walk. Rearrangement of the functional equation gives the expression
\begin{equation}\label{bgfalg}
 F_\kS(y;t) = \frac{1 - \{y^{<0}\}tP(y)F_\kS(0;t)}{1 - tP(y)}.
\end{equation}

\begin{ex}\label{dpfe}
Take $\kS = \{\textbf{NE,~SE}\}$, the steps of a Dyck prefix. Then the directed inventory of $\kS$ is
\[
 P(y) = y + \frac1y,
\]
and the functional equation satisfied by $F_\kS(y;t)$, the BGF for the class of Dyck prefixes, is
\[
 F_\kS(y;t) = 1 + t\left(y + \frac1y\right) F_\kS(y;t) - \frac{t}{y}F_\kS(0;t).
\]
As in Equation (\ref{bgfalg}), we may rearrange the functional equation to find the expression
\[
 F_\kS(y;t) = \frac{1 - \frac{t}{y}F_\kS(0;t)}{1 - t\left(y + \frac1y \right)}.
\]
\end{ex}

The expression (\ref{bgfalg}) is for the MGF of directed meanders on a step set $\kS$ in terms of the OGF for meanders ending on the $x$-axis. Unless we know the classification of $F_\kS(0;t)$, we are unable to immediately classify $F_\kS(y;t)$ as we could for the whole plane model. In Section \ref{dclass} we shall classify the generating functions for directed meanders with small steps by solving the functional equation explicitly through the kernel method.

\section{Classification}\label{dclass}
We continue with expression (\ref{bgfalg})
\[
 F_\kS(y;t) = \frac{1 - \{y^{<0}\}tP(y)F_\kS(0;t)}{1 - tP(y)}
\]
for the MGF of directed meanders. The quantity $1 - tP(y)$ found in the denominator of the above expression is important enough to warrant a name.

\begin{defn}
Let $\cF_\kS$ be a class of directed meanders taken on a set of directed small steps~$\kS$ with directed inventory $P(y)$. The quantity
\[
 1 - tP(y)
\]
found by rearrangement of Equation (\ref{fedm}) is called the \textit{kernel}. We define the \textit{kernel equation} to be
\[
 y - tyP(y) = 0,
\]
where the left hand side of the kernel equation is the kernel made entire.
\end{defn}

We work with directed small steps so the kernel equation is at most quadratic in $y$. Solving by the quadratic formula will yield two solutions, both of which are algebraic functions of $t$. When developed as power series, one shall be convergent at zero and have only non-negative integer coefficients and the other will be singular at 0. Denote the solution with non-negative integer coefficients as $y_0(t)$. This function plays a significant role in the solution of Equation (\ref{fedm}).

\begin{ex}\label{dyckpref}
We continue Example \ref{dpfe}, which showed us that the functional equation is
\[
 F_\kS(y;t) = 1 + t(y + 1/y)F_\kS(y;t) - t/yF_\kS(0;t),
\]
and rearrangment will give
\[
 F_\kS(y;t) = \frac{1 - t/yF_\kS(0;t)}{1 - t(y + 1/y)}.
\]
We make the kernel entire, and set it equal to zero to get the kernel equation
\[
 y - ty^2 - t = 0,
\]
which we may solve using the quadratic formula:
\[
 y = \frac{-1 \pm \sqrt{1 - 4(-t)(-t)}}{2(-t)} = \frac{1 \pm \sqrt{1 - 4t^2}}{2t}.
\]
Development of $\sqrt{1 - 4t^2}$ will show that
\[
 y(t) = \frac{1 + \sqrt{1 - 4t^2}}{2t}
\]
is not convergent at $t=0$, so
\[
 y_0(t) = \frac{1 - \sqrt{1 - 4t^2}}{2t}
\]
is the desired solution to the kernel equation.
\end{ex}

We now move on to a proof that for all directed meanders taken on small steps, the generating function is algebraic. In the process, we obtain a formula for $F_\kS(y;t)$ in terms of the kernel and the power series $y_0(t)$.

\begin{thm}\label{dmgf}
\cite[Theorem 2]{BaFl02} For a set of directed, small steps $\kS$, the BGF of meanders (with $t$ marking size and $y$ marking final altitude) relative to $\kS$ is algebraic. It is given in terms of the function $y_0(t)$, the solution to the kernel equation, by
\[
 F_\kS(y;t) = \frac{y - y_0(t)}{y(1 - tP(y))}.
\]
The counting generating function is found by taking $y = 1$, giving
\[
 F_\kS(1;t) \equiv F_\kS(t) = \frac{1 - y_0(t)}{1 - tP(1)}.
\]
\end{thm}

\begin{proof}
The fundamental equation in the form of Equation (\ref{fedm}) looks underdetermined. It involves two unknowns: the bivariate $F_\kS(y;t)$ and the univariate $F_\kS(0;t)$. The main idea of the kernel method is to bind $t$ and $y$ in such a way that the left hand side vanishes.

First, we remove the negative exponents of $y$ from Equation (\ref{fedm}):
\[
 yF_\kS(y;t) = y + tyP(y)F_\kS(y;t) - y \cdot (\{y^{<0}\}P(y))F_\kS(0;t).
\]
Since $P(y)$ has a negative degree of 1, $(\{y^{<0}\}P(y))$ is simply a Laurent monomial in $y$. Denote it by $r(y)$. We remove the pole at $y_0(t)$ by multiplying both sides by the kernel:
\[
 F_\kS(y;t)(y - tyP(y)) = y - ty r(y)F_\kS(0;t).
\]
We want to force the left hand side to vanish. By restricting $t$ to a small enough neighbourhood of the origin contained in $|t|<1/P(1)$, we can ensure that $|y_0(t)| < 1$ (recall $y_0$ has non-negative coefficients), ensuring that $F_\kS(y;t)$ converges upon substitution.

Thus, the left hand side of the equation vanishes providing us with an equation in a single unknown, the function $F_\kS(0;t)$:
\[
 0 = y_0(t) -  ty_0(t) \cdot r(y_0)F_\kS(0;t).
\]
Thus $F_\kS(0;t)$ is an algebraic function expressible rationally in terms of the algebraic function~$y_0(t)$.

Next, we make use of an observation of Mireille Bousquet-M\'{e}lou \cite{BoPe00}. Define
\[
 N(t,y) := y - t y r(y) F_\kS(0;t).
\]
This is by the substitution in the previous paragraph a linear polynomial in $y$ with its root at the small branch, $y_0(t)$. Thus, we can simply write it as
\[
 N(t,y) = y - y_0(t).
\]
Thus, the constant term is simultaneously $-y_0(t)$ and $-tyr(y)F_\kS(0;t)$, which we may solve for~$F_\kS(0;t)$. Thus, $N(t,y)$ is fully determined in terms of the small branch, and we get the MGF for meanders from the fundamental functional equation made entire:
\[
 F_\kS(y;t) = \frac{N(t,y)}{y(1 - tP(y))} = \frac{y - y_0(t)}{y(1 - tP(y))}.
\]
This gives the OGF of meanders
\[
 F_\kS(t) = \frac{1 - y_0(t)}{1 - tP(1)}.
\]
Since the algebraic functions form a field, and $N(t,y)$, $P(y)$ and $y_0(t)$ are algebraic functions, both the MGF and OGF for directed meanders are algebraic.
\end{proof}

\begin{ex}\label{dpgf}
We may apply Theorem \ref{dmgf} to Example \ref{dyckpref} to produce a generating function for the class. Recall that the directed inventory for the step set is $P(y) = y + \frac1y$ and the solution to the kernel equation is
\[
 y_0(t) = \frac{1 - \sqrt{1 - 4t^2}}{2t}.
\]
Theorem \ref{dmgf} tells us that the bivariate generating function for the class of Dyck prefixes is
\[
 F_\kS(y;t) = \frac{2ty - 1 + \sqrt{1 - 4t^2}}{y(1 - t(y + \frac1y))2t},
\]
and the OGF is
\[
 F_\kS(t) = \frac{2t - 1 + \sqrt{1 - 4t^2}}{(1 - 2t)2t}.
\]
\end{ex}

This concludes the classification of the generating functions associated with directed meanders. In the next section, we turn to the enumeration of these objects, both exact and asymptotic.

\section{Enumeration}\label{direnum}

Now that we have a method for finding the generating functions associated with directed meanders, enumeration seems an easy task. For closed coefficient formulae, we can simply apply the generalised binomial theorem and other coefficient extraction techniques to the generating functions given by Theorem \ref{dmgf}. We would also like a corresponding theorem about the asymptotic behaviour of the coefficient sequences. As is well understood, the asymptotic behaviour of counts is closely related to the singular properties of the corresponding generating functions \cite{FlSe09}. The application of the kernel method in Theorem~\ref{dmgf} gives a factorisation of the generating function under which the singular forms become manageable. This uses a small amount of analysis in the vicinity of what is known as the structural radius, $\rho$.

\begin{lem}
\cite[Lemma 2]{BaFl02} Let $P(u)$ be the polynomial inventory of a simple step set $\kS$. Then there exists a unique number $\tau$, called the \textit{structural constant}, such that
\[
 P'(\tau) = 0, \hspace{1em} \tau > 0.
\]
The \textit{structural radius} is by definition the quantity
\[
 \rho := \frac{1}{P(\tau)}.
\]
\end{lem}

\begin{proof}
Differentiating $P$ twice, notice that $P''(x) > 0$ for all $x > 0$, so viewed as a real function, $P(x)$ is strictly convex. Since it satisfies $P(0) = P(+\infty)=+\infty$, there must be a unique positive minimum attained at some $\tau$, thus $P'(\tau) = 0$.
\end{proof}

Finally, we can get to the asymptotic results for directed meanders. There are three cases, distinguished by a quantity known as the drift.

\begin{defn}
\cite[Definition 5]{BaFl02} Given a simple walk with polynomial inventory $P(y)$, the \textit{drift} is the quantity
\[
 \delta = P'(1).
\]
\end{defn}

In the combinatorial case, with or without weighted steps (see Chapter \ref{def} for how to weight a step set), the drift is then the sum of all the $y$-values of the steps multiplied by the weight of each step, which is an indicator of the `tendency' of the walk to increase or decrease in altitude. For a probabilistic case, the drift is the expected movement in the $y$ direction of any single step. Note that for a symmetric step set, the drift is $\delta = 0$ and the structural constant is $\tau = 1$.

\begin{thm}\label{dirasympt}
\cite[Theorem 4]{BaFl02} Let $\kS$ be a simple set of small steps, with polynomial inventory $P$. The asymptotic number of meanders taken on $\kS$ depends on the sign of the drift $\delta = P'(1)$ as follows.

\[
\begin{array}{r c c l}
  \delta = 0 : & [t^n]F_\kS(1;t) & \sim & \nu_0 \frac{P(1)^n}{\sqrt{\pi n}} \left( 1 + \frac{c_1}{n} + \frac{c_2}{n^2} + \cdots \right) \\
  &&&\nu_0 := \sqrt{2 \frac{P(1)}{P''(1)}}; \\
  &&&\\
  \delta < 0 : & [t^n]F_\kS(1;t) & \sim & \nu_0^{\pm} \frac{P(\tau)^n}{2\sqrt{\pi n^3}} \left( 1 + \frac{c'_1}{n} + \frac{c'_2}{n^2} + \cdots \right) \\
  &&&\nu_0^{\pm} := -\sqrt{2 \frac{P(\tau)^3}{P''(\tau)}}\frac{1}{P(\tau) - P(1)}; \\
  &&&\\
  \delta > 0 : & [t^n]F_\kS(1;t) & \sim & \xi_0 P(1)^n + \nu_0^{\pm} \frac{P(\tau)^n}{2\sqrt{\pi n^3}} \left( 1 + \frac{c''_1}{n} + \frac{c''_2}{n^2} + \cdots \right) \\
  &&&\xi_0 := 1 - y_0(\rho_1), \hspace{1em} \rho_1 = P(1)^{-1}.
\end{array}
\]
\end{thm}

We can view these formulae intuitively. For the positive drift cases, a fraction close to $\xi_0$ of the unconstrained walks is a meander, which makes sense since the walks have a natural tendency to gain in altitude. For negative drift, most paths tend to head down, and so the fraction of unconstrained walks which are meanders is exponentially small, roughly~$(P(\tau)/P(1))^n$. For zero drift, the proportion of walks which are meanders is as large as $1/\sqrt{n}$, with the walks tending to oscillate near the horizontal axis.

\begin{proof}
The discussion is based on an analysis of the generating function $F_\kS(1;t)$, in its factorisation
\[
 F_\kS(1;t) = \frac{1 - y_0(t)}{1 - tP(1)}
\]
provided by Theorem \ref{dmgf}. The method of proof examines the locations of the singularities and zeroes of the numerator in relation to the zero $1/P(1)$ of the denominator.

In the case that $\delta = 0$, we have $P'(1) = 0$, $\tau = 1$ and $\rho = 1/P(1)$. We must do some analysis of the solution to the kernel equation $y_0(t)$ in the vicinity of $\rho$. First, the kernel equation can be put into the form
\[
 t = \frac{1}{P(y)}.
\]
Since the inverse of an analytic function at a point where the derivative is nonzero is analytic, the above relation is invertible analytically in the neighbourhood of any point $t$ such that $P'(t) \neq 0$. The contrapositive then must hold: a singularity will occur at any value $\zeta$ such that $P'(\zeta) = 0$.

When $\delta = 0$, $P'(1) = 0$ by construction while $P''(1) > 0$. Thus the local form of $t = 1/P(y)$ reads
\[
 t = \frac1{P(1)} - \frac12 P''(1)(y - 1)^2 + O((y - 1)^3).
\]
This is readily inverted, leading to two local solutions
\[
 y(t) = 1 \pm \sqrt{2\frac{P(1)}{P''(1)}} \sqrt{1 - t/\rho} + ... \hspace{1em} t \rightarrow \rho^-.
\]
Therefore, $1 - y_0(t)$ contributes a term of the form $(1 - t/\rho)^{1/2}$ at $t = \rho$, while the denominator $(1 - tP(1))$ has a simple zero there. Globally, the singularity of $F_\kS(1;t)$ is of the square root type, giving the result.

For a negative drift meander, we have $P'(1) < 0$, so we must have $\tau >1$ since $P'(u)$ increases from $-\infty$ to $+\infty$ when $y$ ranges from $0^+$ to $+\infty$. With $\rho = 1/P(\tau)$ (the structural radius) and $\rho_1 = 1/P(1)$, we have $\rho_1 < \rho$. In this case, the prefactor $(1 - tP(1))^{-1}$ has a pole at $\rho_1$. This pole is cancelled, however, by a zero in the numerator induced by $(1 - y_0(t))$ (since $y_0(\rho_1) = 1)$ making $\rho_1$ a removable singularity of $F_\kS(1;t)$. This puts the dominant singularity of $F_\kS(1;t)$ at $\rho$, and it is of the square root type.

Finally, for positive drift, $\tau < 1$, so that the prefactor induces a pole at $\rho_1 = 1/P(1)$ before $1 - y_0(t)$ becomes singular. The argument conlcudes by subtracting singularities, since the function
\[
 F_\kS(1;t) - \frac{1 - y_0(\rho_1)}{1 - zP(1)}, \hspace{1em} \rho_1 := \frac{1}{P(1)},
\]
now has a dominant singularity of the square root type at $\rho$.
\end{proof}

\begin{ex}\label{asymptex}
We complete this chapter's running example, obtaining asymptotic results. Dyck prefixes have a zero drift step set, since
\[
 P'(y) = 1 - \frac1{y^2} \Rightarrow P'(1) = 0.
\]
In order to apply the results of Theorem \ref{dirasympt}, we need the second derivative of the inventory:
\[
 P''(y) = \frac{2}{y^3}.
\]
So, by the aforementioned theorem,
\[
 [t^n]F_\kS(1;t) \sim \sqrt{2\frac{P(1)}{P''(1)}}\frac{P(1)^n}{\sqrt{\pi n}} = \sqrt{\frac{2}{\pi}}\frac{2^n}{\sqrt{n}}.
\]
\end{ex}

This completes the material on directed lattice paths with small steps. The following chapter returns to undirected paths, starting with walks confined to a half plane. By showing a bijection between directed walks and half plane walks, we set up a useful enumerative result for our analysis of quarter plane walks in Chapter \ref{qpwalks}. In Chapter \ref{conclusion}, we discuss how we can use results on directed walks with more general step sets for a generalised version of the results given in this document.

\chapter{Walks confined to the half-plane}\label{hpwalks}

We turn back to undirected walks with small steps, but begin restricting the region. By adding a boundary to the unrestricted case considered in Chapter \ref{undir} we can begin to see how these combinatorial structures could model more than just polymers. For instance, a single boundary model could represent the interaction of a particle with some kind of container. While in general any line in $\RR^2$ could be used as the boundary of a half plane, we will only discuss the case where the $x$-axis is used.

\section{Definition}\label{hdef}

Fix a set of small steps $\kS$ and take the class $\cW_\kS$ of planar walks on $\kS$. We make the restriction that all elements of the sequence of points $(p_0,...,p_n)$ for each walk of length $n$ must lie inside the half plane $y \geq 0$. We call this subclass of walks $\cH_\kS$ and seek $h(n)$, the cardinality of $\cH_n$. Shown in Figure \ref{half} is a typical example of one of these.

\begin{figure}[h]
 \begin{center}
  \begin{tikzpicture}[scale = 1]
  \draw[->, thick] (0,0) -- (0,4.5);
  \draw[<->, thick] (-3,0) -- (3,0);
  \draw[-, dotted] (-3,1) -- (3,1);
  \draw[-, dotted] (-3,2) -- (3,2);
  \draw[-, dotted] (-3,3) -- (3,3);
  \draw[-, dotted] (-3,4) -- (3,4);
  \draw[-, dotted] (-3,0) -- (-3,4);
  \draw[-, dotted] (-2,0) -- (-2,4);
  \draw[-, dotted] (-1,0) -- (-1,4);
  \draw[-, dotted] (1,0) -- (1,4);
  \draw[-, dotted] (2,0) -- (2,4);
  \draw[-, dotted] (3,0) -- (3,4);
  \draw[->, very thick,red] (0,0) -- (1,1);
  \draw[->, very thick,red] (1,1) -- (2,2);
  \draw[->, very thick,red] (2,2) -- (1,3);
  \draw[->, very thick,red] (1,3) -- (0,2);
  \draw[->, very thick,red] (0,2) -- (-1,1);
  \draw[->, very thick,red] (-1,1) -- (-2,2);
  \draw[->, very thick,red] (-2,2) -- (-3,3);
  \draw[->, very thick,red] (-3,3) -- (-2,4);
  \draw[->, very thick,red] (-2,4) -- (-1,3);
  \draw[->, very thick,red] (-1,3) -- (-2,2);
  \draw[->, very thick,red] (-2,2) -- (-3,1);
  \end{tikzpicture}

  $n = 11$, $\kS = \diagr{NW,NE,SW,SE}$
  \end{center}
\caption{A walk of length 11 from $\cH_\kS$}\label{half}
\end{figure}

Given a step set $\kS$ with inventory $S(x,y)$, we can use recursion to find a functional equation that the generating function $H_\kS(x,y;t)$ satisfies. It is similar to before: a half plane walk is either the empty walk or a walk of shorter length incremented by adding a new step. We must take care here, however, and be sure that we don't wander out of the allowed region. We can exculde this possibility by considering the generating function $H_\kS(x,0;t)$, which counts walks ending on the $x$ axis.

So, the functional equation which $H_\kS(x,y;t)$ satisfies is
\[
 H_\kS(x,y;t) = 1 + tS(x,y)H_\kS(x,y;t) - \{y^{<0}\}tS(x,y)H_\kS(x,0;t).
\]
Notice the similarities to the functional equation for directed meanders? This is because these two classes are equinumerous when we make $x = 1$, referred to in Chapter \ref{dir} as taking the directed inventory, given as
\[
 S(x,y) \mapsto P(y) = S(1,y).
\]
This brings us to a new definition.

\begin{defn}
Fix a step set $\kS$, with inventory $S(x,y)$. When we take the directed inventory of $\kS$, $P(y) = S(1,y)$, we obtain a new, directed step set $\kS_P$ which is in general weighted (has multiple steps in the same direction), which we call the \textit{horizontal projection} of $\kS$. 
\end{defn}

We illustrate this new definition and the bijection it gives in the following example.

\begin{ex}
We show in Figure \ref{hpex} the weighted step set found by taking the directed inventory of $\kS = \{\textbf{NE,~ SE,~ SE,~NW}\}$. This is $\kS_P$, the horizontal projection of $\kS$.

\begin{figure}[h]
 \begin{center}
  \diagb{NW,NE,SW,SE}
    \begin{tikzpicture}[scale = 0.5]
	\draw[-,black,thick] (0,-0.1) -- (0,0.1);
	\draw[-,black,thick] (0,0) -- (1,0);
	\draw[-,black,thick] (1,0) -- (0.8,0.2);
	\draw[-,black,thick] (1,0) -- (0.8,-0.2);
	\draw[-,white] (0,-0.2) -- (0,-0.8);
    \end{tikzpicture}
    \begin{tikzpicture}[scale = 0.4]
      \draw[-,red, thick] (0,0) -- (1,1);
      \draw[-,red, thick] (0,0) -- (1,-1);
      \draw[-,blue, thick] (0.5,0) -- (1.5,1);
      \draw[-,blue, thick] (0.5,0) -- (1.5,-1);
    \end{tikzpicture}
  \end{center}
\caption{Taking the directed inventory of $\kS$ leads to the horizontal projection $\kS_P$}\label{hpex}
\end{figure}
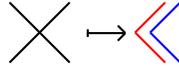

Now in the directed model, every $x$-negative step is a red step and every $x$-positive step a blue one. This gives the bijection between $\cH_\kS$ on $\kS$ and $\cF_{\kS_P}$ on $\kS_P$, and we show an example of the correspondence in Figure \ref{bij}.

\begin{figure}[h]
\begin{center}
  \begin{tikzpicture}[scale = 0.8]
  \draw[->, thick] (0,0) -- (0,4.5);
  \draw[<->, thick] (-3,0) -- (3,0);
  \draw[-, dotted] (-3,1) -- (3,1);
  \draw[-, dotted] (-3,2) -- (3,2);
  \draw[-, dotted] (-3,3) -- (3,3);
  \draw[-, dotted] (-3,4) -- (3,4);
  \draw[-, dotted] (-3,0) -- (-3,4);
  \draw[-, dotted] (-2,0) -- (-2,4);
  \draw[-, dotted] (-1,0) -- (-1,4);
  \draw[-, dotted] (1,0) -- (1,4);
  \draw[-, dotted] (2,0) -- (2,4);
  \draw[-, dotted] (3,0) -- (3,4);
  \draw[->, very thick,red] (0,0) -- (1,1);
  \draw[->, very thick,red] (1,1) -- (2,2);
  \draw[->, very thick,red] (2,2) -- (1,3);
  \draw[->, very thick,red] (1,3) -- (0,2);
  \draw[->, very thick,red] (0,2) -- (-1,1);
  \draw[->, very thick,red] (-1,1) -- (-2,2);
  \draw[->, very thick,red] (-2,2) -- (-3,3);
  \draw[->, very thick,red] (-3,3) -- (-2,4);
  \draw[->, very thick,red] (-2,4) -- (-1,3);
  \draw[->, very thick,red] (-1,3) -- (-2,2);
  \draw[->, very thick,red] (-2,2) -- (-3,1);
  \end{tikzpicture}
  \begin{tikzpicture}[scale = 0.5]
  \draw[<->,black,very thick] (-1,0) -- (1,0); \draw[white,thick] (-1,-3) -- (1,-3); \end{tikzpicture}
  \begin{tikzpicture}[scale = 0.8]
  \draw[->, thick] (0,0) -- (0,4.5);
  \draw[->, thick] (0,0) -- (11,0);
  \draw[-, dotted] (0,1) -- (11,1);
  \draw[-, dotted] (0,2) -- (11,2);
  \draw[-, dotted] (0,3) -- (11,3);
  \draw[-, dotted] (0,4) -- (11,4);
  \draw[-, dotted] (1,0) -- (1,4);
  \draw[-, dotted] (2,0) -- (2,4);
  \draw[-, dotted] (3,0) -- (3,4);
  \draw[-, dotted] (4,0) -- (4,4);
  \draw[-, dotted] (5,0) -- (5,4);
  \draw[-, dotted] (6,0) -- (6,4);
  \draw[-, dotted] (7,0) -- (7,4);
  \draw[-, dotted] (8,0) -- (8,4);
  \draw[-, dotted] (9,0) -- (9,4);
  \draw[-, dotted] (10,0) -- (10,4);
  \draw[-, dotted] (11,0) -- (11,4);
  \draw[->, very thick, blue] (0,0) -- (1,1);
  \draw[->, very thick, blue] (1,1) -- (2,2);
  \draw[->, very thick, red] (2,2) -- (3,3);
  \draw[->, very thick, red] (3,3) -- (4,2);
  \draw[->, very thick, red] (4,2) -- (5,1);
  \draw[->, very thick, red] (5,1) -- (6,2);
  \draw[->, very thick, red] (6,2) -- (7,3);
  \draw[->, very thick, blue] (7,3) -- (8,4);
  \draw[->, very thick, blue] (8,4) -- (9,3);
  \draw[->, very thick, red] (9,3) -- (10,2);

  \draw[->, very thick, red] (10,2) -- (11,1);
  \end{tikzpicture}
  \end{center}
\caption{A walk from $\cH_\kS$ and its horizontal projection in $\cF_{\kS_P}$}\label{bij}
\end{figure}
\end{ex}

This significantly simplifies classification and enumeration for us, as will be seen in the following two sections.

\section{Classification}\label{hclass}
We begin the section on classification by showing the isomorphism between $\cH_\kS$ relative to a step set $\kS$ with inventory $S(x,y)$ and $\cF$ relative to the horizontally projected step set $\kS_p$ with directed inventory $P(y) = S(1,y)$.

\begin{prop}\label{htof}
The class $\cH_\kS$ of undirected half plane walks on a step set $\kS$ and the class~$\cF_{\kS_P}$ of directed meanders on $\kS_p$, the horizontal projection of $\kS$, are in bijection.
\end{prop}

\begin{proof}
Fix a set of steps $\kS$, with inventory $S(x,y)$. We know that the generating function~$H_\kS(x,y;t)$ for the class $\cH_\kS$ of half plane walks taken on $\kS$ satisfies
\[
 H_\kS(x,y;t) = 1 + tS(x,y)H_\kS(x,y;t) - \{y^{<0}\}tS(x,y)H_\kS(x,0;t).
\]
By taking the horizontal projection, we are setting $x = 1$, giving the new functional equation
\[
 H_\kS(1,y;t) = 1 + tS(1,y)H_\kS(1,y;t) - \{y^{<0}\}tS(1,y)H_\kS(1,0;t).
\]
The inventory is now a function of $y$ only, and can be represented by the inventory of the horizontal projection, $\kS_p$. Making the replacement $P(y) = S(1,y)$, we find
\[
 H_\kS(1,y;t) = 1 + tP(y)H_\kS(1,y;t) - \{y^{<0}\}tP(y)H_\kS(1,0;t).
\]

Now, let $F_{\kS_P}(y;t)$ be the generating function for directed meanders on the set of steps $\cP$, the horizontal projection of $\kS$. Recall from Chapter \ref{dir} that $F_{\kS_P}(y;t)$ satisfies
\[
 F_{\kS_P}(y;t) = 1 + tP(y)F_{\kS_P}(y;t) - \{y^{<0}\}tP(y)F_{\kS_P}(0;t),
\]
which is simply the functional equation for the walks in $\cH_\kS$ taken on the original step set $\kS$ with the $x$ information suppressed (or with $x$ set to 1). Thus, $H_\kS(1,y;t)$ and $F_{\kS_P}(y;t)$ satisfy the same functional equation, and therefore the coefficients of both series satisfy the same recursion. Since the first few terms of each coefficient sequence agree, this proves that these two classes, complete with the stopping height parameter $y$, are in bijection. That is
\[
 F_{\kS_P}(1;t) \equiv H_\kS(1,1;t).
\]
\end{proof}

Thus, the generating MGF $H_\kS(1,y;t)$ and OGF $H_\kS(t)$ for half plane walks on a step set $\kS$ are the same as the MGF $F_{\kS_P}(y;t)$ and the OGF $F_{\kS_P}(t)$ for directed meanders on the projected set $\kS_p$. There is one detail we have glossed over here: the $x$-coordinate of the steps. Because the half plane walks we are working with here are on sets of small steps, there is no problem with simply suppressing the $x$ information. However, for more general sets of steps, the same holds. Since the only boundary is the $x$-axis, the $y$-coordinate of each step is what is important to allowing a walk to be `legal' or not, and therefore the property that both classes are equinumerous will hold. Therefore, we may apply the results of Chapter \ref{dir} to show that $\cH_\kS$ has algebraic generating functions.

\section{Enumeration}\label{henum}

In view of our goal of enumeration, exact or asymptotic, and a corresponding generating function this is ideal: we obtain all the information we're after by considering the directed model found by setting $x = 1$. The next example illustrates this for us.

\begin{ex}
Take as an example the step set shown in Figure \ref{half}, $\kS = \{ \textbf{NW,~NE,~SW,~SE}\}$. A typical walk is shown in the same Figure. If we let $H_\kS(x,y;t)$ be the generating function for the half plane walks taken on $\kS$, then $H_\kS$ will satisfy the functional equation
\[
 H_\kS(x,y;t) = 1 + t \left(xy + \frac{x}{y} + \frac1{xy} + \frac{y}{x} \right)H_\kS(x,y;t) - t\left(\frac1{xy} + \frac{x}{y}\right)H_\kS(x,0;t).
\]
We take the horizontal projection by setting $x = 1$ to obtain
\[
 H_\kS(1,y;t) = 1 + t(2y + \frac{2}{y})H_\kS(1,y;t) - t\left(\frac2{y}\right)H_\kS(1,0;t).
\]
According to Proposition \ref{htof} this is the functional equation for the class of directed meanders on $\cP$, the horizontal projection of $\kS$, allowing us to bring to bear the machinery given in Chapter \ref{dir}. We suppress the $x$ information and rearrange the equation to give
\[
 H_\kS(y;t) - t (2y + \frac{2}{y})H_\kS(y;t) = 1 - t\frac2y H_\kS(0;t),
\]
from which we may isolate the kernel in the denominator of the right hand side, like so:
\[
 H_\kS(y;t) = \frac{1 - t\frac2y H_\kS(0;t)}{1 - t (2y + \frac{2}{y})}.
\]
We let
\[
 K(y,t) = 1 - t (2y + \frac2y),
\]
and solve the characteristic equation $K(y,t) = 0$:
\begin{eqnarray*}
 y - t (2y^2 + 2) = 0 
  &\Rightarrow& 2ty^2 - y + 2t = 0, \\
  &\Rightarrow& y = \frac{1 \pm \sqrt{1 - 16t^2}}{4t};
\end{eqnarray*}
from which we take the small branch (the branch analytic at 0) with equation
\[
 y_1(t) = \frac{1 - \sqrt{1 - 16t^2}}{4t}.
\]
We only need the enumeration information, so we ignore the parameter $y$ by setting $y = 1$. Then the generating function is
\[
 H_\kS(1,1;t) = \frac{1 - y_1(t)}{1 - tP(1)} = \frac{4t - 1 + \sqrt{1 - 16t^2}}{4t(1 - 4t)}
\]
from which we may find a formula for $h(n)$ using the generalised binomial theorem. As in the directed case, this is a radical:
\[
 H_\kS^*(t) = \frac{4t - 1 + \sqrt{1 - 16t^2}}{4t}.
\]

Next, we need asymptotics. We may apply the black box Theorem \ref{dirasympt}. First, we need the drift of the walk. It should be obvious that this is zero, but the general method is to find~$\delta = P'(1)$, so in our case
\[
 P'(y) = 2 - \frac2{y^2} \Rightarrow P'(1) = 0
\]
as expected. This tells us that we use the 0 drift expansion for the asymptotics, given as
\[
 [t^n]F_{\kS_P}(1;t) \sim \nu_0 \frac{4^n}{\sqrt{\pi n}} \left(1 + \frac{c_1}{n} + \frac{c_2}{n^2} + \cdots \right)
\]
where
\[
 \nu_0 = \sqrt{2 \frac{P(1)}{P''(1)}} = \sqrt{2}.
\]
This gives the asymptotic development for half plane walks on $\kS$ as
\[
 [t^n]F_{\kS_P}(1;t) \sim  \frac{\sqrt{2} \cdot 4^n}{\sqrt{\pi n}} \left(1 + \frac{c_1}{n} + \frac{c_2}{n^2} + \cdots \right).
\]
\end{ex}

The ability to exactly enumerate half plane walks on a step set $\kS$ is a tool which we can use to find upper bounds on the exponential growth of the quarter plane walks on the same step set. After some background on the state of quarter plane walks, Chapter \ref{qpwalks} will show how this is done.

\chapter{Walks confined to the quarter plane}\label{qpwalks}

We have seen so far that walks in the whole plane have a rational generating function, easily found by an application of elementary analytic combinatorics. Walks in a half plane are solved using the bijection to a directed model which in turn is solved by the kernel method. These models have an algebraic generating function in every case. We finally turn our attention to the final region of the thesis: walks restricted to first quadrant of $\ZZ^2$. This class is much richer, and there are only the beginnings of a unified approach \cite{FaRa12,BoRaSa12}. There are models with transcendental D-finite generating functions, and models with non-D-finite generating functions.

In Section \ref{qpdef} we give the problem in analogy with the previous three cases, stating a functional equation that the MGF will satisfy. Section~\ref{qpclass} covers a reduction of the~$2^8~-~1$ possible step sets to 79 non-equivalent models, classifying their associated generating functions as D-finite or non D-finite and providing an asymptotic growth template for the D-finite cases. As in the half plane model and the directed model, results will be dependent upon the drift.

Section \ref{qpenum} covers the enumerative results for the D-finite cases, beginning with a survey of previous results. After the survey, we give in Section \ref{base} proofs of the exponential growth for 8 D-finite step sets which form the base cases in our case analysis. Some of these proofs are original, and those taken from other authors are marked as such. Sections \ref{ubs} and \ref{lbs} give an original systematic way of proving the exponential growth of each case by bounding the exponential growth above in a unified way and importing lower bounds from the base cases via a boot strapping method given in Lemmas \ref{onestep} and \ref{twostep}. The results are summarised and notation for each case are introduced in Table \ref{summ}.

\begin{table}
\begin{center}
\begin{tabular}{| c | c | c | c | c | c |}
  \hline
  $i$ & $\kS_i$ & Parent & $\delta_i$ & $\beta_i$ & Lower bounds \\ \hline
  &&&&&\\[-5pt] 
  1 &  \diagr{N,S,E,W} & \diagr{N,S} & 0 & 4 & Lemma \ref{twostep} on \diagr{N,S} \\ [+2mm]
  2 & \diagr{NW,SW,NE,SE} & Base Case & 0 & 4 & Proposition \ref{S02} (p \pageref{S02}) \\  [+2mm]
  3 & \diagr{N,S,NE,NW,SW,SE} & \diagr{NE,SE,SW,NW} & 0 &  6 & Lemma \ref{twostep} on $\kS_2$ \\ [+2mm]
  4 & \diagr{N,S,NE,NW,SW,SE,E,W} & \diagr{NE,SE,SW,NW,N,S} & 0 & 8 & Lemma \ref{twostep} on $\kS_3$ \\ [+2mm]
  5 & \diagr{S,NE,NW} & \diagr{W,N,SE} & \drn & 3 & Proposition \ref{S05} (p \pageref{S05}) \\ [+2mm]
  6 & \diagr{S,NE,NW,E,W} & \diagr{NW,NE,S} & \drn & 5 & Lemma \ref{twostep} on $\kS_5$ \\ [+2mm]
  7 & \diagr{S,NE,NW,N} & \diagr{NE,NW,S} & \drn & 4 & Lemma \ref{onestep} on $\kS_5$ \\ [+2mm]
  8 & \diagr{S,NE,NW,N,E,W} & \diagr{NW,NE,N,S} & \drn & 6 & Lemma \ref{twostep} on $\kS_7$ \\ [+2mm]
  9 & \diagr{SW,SE,NE,NW,N} & \diagr{NW,NE,SW,SE} & \drn & 5 & Lemma \ref{onestep} on $\kS_2$ \\ [+2mm]
  10 & \diagr{SW,SE,NE,NW,N,E,W} & \diagr{SE,SW,NE,NW,N} & \drn & 7 & Lemma \ref{twostep} on $\kS_9$ \\ [+2mm]
  11 & \diagr{S,SW,SE,N} & Base Case & \drs & $2\sqrt{3}$ & Proposition \ref{S11} (p \pageref{S11}) \\ [+2mm]
  12 & \diagr{S,SW,SE,N,E,W} & \diagr{S,SW,SE,N} & \drs & $2(1 + \sqrt3)$ & Lemma \ref{twostep} on $\kS_{11}$ \\ [+2mm]
  13 & \diagr{S,SW,SE,NW,NE} & Base Case & \drs & $2\sqrt6$ & Proposition \ref{S13} (p \pageref{S13}) \\ [+2mm]
  14 & \diagr{S,SW,SE,NW,NE,E,W} & \diagr{S,SW,SE,NW,NE} & \drs & $2(1 + \sqrt6)$ & Lemma \ref{twostep} on $\kS_{13}$ \\ [+2mm]
  15 & \diagr{SW,SE,N} & Base Case & \drs & $2\sqrt2$ & Proposition \ref{S15} (p \pageref{S15}) \\ [+2mm]
  16 & \diagr{SW,SE,N,E,W} & \diagr{SW,SE,N} & \drs & $2(1 + \sqrt2)$ & Lemma \ref{twostep} on $\kS_{15}$ \\ [+2mm]
  17 & \diagr{N,W,SE} & Base Case & 0 & 3 & Proposition \ref{S17} (p \pageref{S17}) \\ [+2mm]
  18 & \diagr{N,S,E,W,NW,SE} & \diagr{NW,E,SE,W} & 0 & 6 & Lemma \ref{twostep} on $\kS_{22}$ \\ [+2mm]
  19 & \diagr{S,W,NE} & Base Case & 0 & 3 & Proposition \ref{S19} (p \pageref{S19}) \\ [+2mm]
  20 & \diagr{N,E,SW} & Base Case & 0 & 3 & Proposition \ref{S19} (p \pageref{S19}) \\ [+2mm]
  21 & \diagr{N,S,E,W,NE,SW} & \diagr{N,E,S,W} & 0 & 6 & Lemma \ref{twostep} on $\kS_1$ \\ [+2mm]
  22 & \diagr{NW,SE,W,E} & Base Case & 0 & 4 & Proposition \ref{S22} (p \pageref{S22}) \\ [+2mm]
  23 & \diagr{NE,SW,W,E} & \diagr{E,W} & 0 & 4 & Lemma \ref{twostep} on \diagr{E,W} \\ \hline
\end{tabular}
\caption{The exponential growth factors for the D-finite cases and how lower bounds on the exponential growth are found. Upper bounds are given by Theorem \ref{bnd} for every case.}\label{summ}
\end{center}
\end{table}

\section{Definitions}\label{qpdef}

Fix a step set $\kS \subseteq \bar{\kS}$. We call the class of walks taken on $\kS$ restricted to the quarter plane~$\cQ_\kS$. It should be noted that on the step set $\kS$, we have $\cQ_\kS \subseteq \cH_\kS$. This fact will prove useful in the asymptotic enumeration of the quarter plane walks.

So, let $q(n)$ be the number of walks on $\kS$ confined to the quarter plane. Then the OGF for $\cQ_\kS$ is 
\[
 Q_\kS(t) = \sum_{n \geq 0} q(n)t^n.
\]
As with the half plane case, in order to find a functional equation, we need to parameterise the class. Let $q(i,j;n)$ be the number of walks with steps from $\kS$, confined to the quarter plane. Then the MGF for $\cQ_\kS$ with parameters $x$ and $y$ marking $x$ and $y$ coordinates, respectively, of the endpoint, is
\[
 Q_\kS(x,y;t) = \sum_{i,j,n \geq 0} q(i,j;n) x^iy^jt^n.
\]
The functional equation that $Q_\kS(x,y;t)$ satisfies is
\begin{eqnarray}
 Q_\kS(x,y;t) = 1 &+& t\left( S(x,y)Q_\kS(x,y;t) \right. \nonumber \\
    &-& \{x^{<0}\}S(x,y)Q_\kS(0,y;t) \nonumber \\
    &-& \left. \{y^{<0}\}S(x,y)Q_\kS(x,0;t) \right) \nonumber \\
    &+& \omega(\kS)\frac{t}{xy}Q_\kS(0,0;t),
\end{eqnarray}
where
\[
 \omega(\kS) = \left\{ \begin{array}{ll} 1 & \mbox{if } (-1,-1) \in \kS, \\ 0 & \mbox{otherwise.} \end{array} \right.
\]
Note the differences from the half plane functional equation. The first difference to note is the term $t\{x^{<0}\}S(x,y)Q_\kS(0,y;t)$, removing the walks of length $n+1$ which make their way into the region $x < 0$. We also have the term $\omega(\kS)\frac{t}{xy}Q_\kS(0,0;t)$, which corrects for undercounting if the \textbf{SW} step is included in $\kS$ and therefore removed twice since it is both~$x$-negative and~$y$-negative.

Again, we have a notion of drift of a step set, this time with two components. We recall it in the following definition.

\begin{defn}
For a given step set $\kS$, the \textbf{drift} is defined to be the vector sum of the steps in~$\kS$. 
\end{defn}

Note that since the drift is a sum of 2-vectors, it will be a 2-vector also. Thus we give the drift in terms of the sign of each component. This is illustrated by an example.

\begin{ex}
Take the step set $\diagr{N,E,S,W,NW}$. The horizontal and vertical steps will have vector sum~$0$, so the vector sum of the whole step set is in the \textbf{NW} direction, given by the ordered pair~$(-1,1)$ or the vector diagram \begin{tikzpicture}[scale = 0.2] \draw[->, very thick] (0,0) -- (-1,1); \end{tikzpicture}. In this case, we say the drift is $y$-positive and $x$-negative.
\end{ex}

This new functional equation is a little more complicated than before. First of all, it's in terms of at least three generating functions and in some cases four. The success of the kernel method in the directed and half plane models is due to the low number of unknowns that can be made dependent by solving the kernel equation, a property missing from this case. We can rearrange the functional equation to isolate~$Q_\kS(x,y;t)$, giving
\begin{equation}\label{qpfe}
 Q_\kS(x,y;t) = \frac{1 - t\left( \{x^{<0}\}S(x,y)Q_\kS(0,y;t) + \{y^{<0}\}S(x,y)Q_\kS(x,0;t) \right) + \omega(\kS)\frac{t}{xy}Q_\kS(0,0;t)}{1 - tS(x,y)}.
\end{equation}
In analogy with the directed model, the denominator is called the kernel of the model, and plays a part in the orbit-sum enumerative techniques used by Bousquet-M\'{e}lou and Mishna in \cite{BoMi10}. The following section gives an overview of the part the kernel plays in the classification of $Q_\kS(x,y;t)$ as algebraic, D-finite or non D-finite.

\section{Classification}\label{qpclass}

In the full plane model, there are $2^8 - 1$ possible step sets $\kS$ available. Of these, some will give rise to trivial classes $\cQ_\kS$. These are subsets of $\{ \textbf{SE,~S,~SW,~W,~NW}\}$. There are also step sets which will never interact with one boundary, making them automatically half plane models. Shown in Figure \ref{autohp} is one of the largest step sets that will behave in this way. Another is found by reflecting the first across $y = x$. By removing trivial step sets as well as step sets which are automatically half plane models, thus solvable by the kernel method, and considering symmetries, Bousquet-M\'{e}lou and Mishna in \cite{BoMi10} reduce this number to 79 non-equivalent true quarter plane models.

\begin{figure}[h]
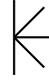

 \begin{center}
  \diagb{N,NE,E,SE,S}
  \caption{The largest step set which is automatically a half plane model.}\label{autohp}
 \end{center}
\end{figure}

Associated with each model is $G(\kS)$, some dihedral group of birational transformations which preserves the kernel $K(x,y;t) = 1 - tS(x,y)$ of a model - hence dependent on~$\kS$. Note that not all dihedral symmetries are in $G(\kS)$, the group depends on the generators defined in \cite{BoMi10}. For~23 step sets $\kS$, this group is finite while for the remaining 56 it is infinite. There is a lot of evidence for a correspondence between the cardinality of $G(\kS)$ and the classification of the associated generating function: it is shown in \cite{BoMi10} that 22 of the 23 finite group models have D-finite generating functions - with the 23rd requiring another approach in its own paper \cite{BosKau09} - and it is shown in \cite{BoRaSa12} that 51 of the 56 infinite group models have non D-finite generating functions. The remaining 5 models are known as the singular models, and are particularly stubborn when it comes to being shown they are non D-finite (see \cite{MiRe09}).

For the remainder of this chapter, we focus on those models associated with a finite group or those with a D-finite generating function. These models are listed in Table \ref{fintab}, along with experimentally found asymptotic expressions, first given in \cite{BoKa09}, which we shall discuss in the next section. For the remainder of this section, we give a stronger classification of the generating functions, which helps with the asymptotic enumeration later on. We recall a definition from \cite{BoKa09}.

\begin{defn}
\cite{BoKa09} A power series $Q(t) = \sum_{n \geq 0}a_nt^n$ in $\QQ[\![t]\!]$ is called a $G$-series if:
\begin{enumerate}
 \item it is D-finite;
  \item its radius of convergence in $\CC[\![t]\!]$ is positive;
  \item there exists a constant $C > 0$ such that for all $n \in \NN$ the common denominator of $a_0,...,a_n$ is bounded by $C^n$.
\end{enumerate}
\end{defn}

Since generating functions are elements of $\ZZ[\![t]\!]$, they trivially pass the third condition. Furthermore, the generating functions in question here are a priori D-finite, so the first condition is satisfied also. We can show the second condition to be satisfied by considering the final line of the previous paragraph: the number of walks of length $n$ cannot exceed $|\kS|^n$. We summarise this in the following proposition.

\begin{prop}\label{rocbnd}
\cite[Theorem 1]{BoKa09} For a given step set $\kS$, the radii of convergence of each of $Q_\kS(0,0;t)$, $Q_\kS(0,1;t)$, $Q_\kS(1,0;t)$, $Q_\kS(1,1;t)$ is bounded below by $\frac{1}{|\kS|}$.
\end{prop}

\begin{proof}
The total number of walks in $\cW$ taken on $\kS$ is $|\kS|^n$, thus this forms an upper bound on the number of walks in $\cQ_\kS$ taken on $\kS$, which is also the coefficient $q(n)$. By virtue of the exponential growth formula from Theorem \ref{expgrow}, this upper bound on the coefficients of the generating function forms a lower bound on the radius of convergence, proving the statement.
\end{proof}

This satisfies the second property of $G$-series, completing the proof of the following theorem from \cite{BoKa09}.

\begin{thm}\label{gseries}
\cite[Theorem 2]{BoKa09} The generating functions for the quarter plane models $\cQ_{\kS_i}$ for $i=1,...,23$ are $G$-series.
\end{thm}

In \cite{BoKa09}, Bostan and Kauers show Theorem \ref{gseries} along with a property of the asymptotic growth of the coefficients of the generating functions. It can be concluded from Theorem \ref{gseries} and results in \cite{Ga09} that the coefficients of any one of the generating functions in question fits the asymptotic template
\[
 q(n) \sim \kappa \beta ^n n^\alpha (\log n)^\gamma.
\]
Furthermore, Bostan and Kauers give compelling evidence that $\gamma=0$ for each of the cases we consider in this chapter, thus these models probably have the same growth template as an algebraic series. We discuss this further in the following section on enumeration.

\section{Enumeration}\label{qpenum}

We begin the enumeration discussion by surveying some of the more general previous enumerative results linked with the classification of the generating functions in Section \ref{qpprev}. These are split into rigourous results given by Bousquet-M\'{e}lou and Mishna in~\cite{BoMi10} and the experimental results of Bostan and Kauers in \cite{BoKa09}.

Sections \ref{base} through \ref{lbs} contain the original contribution of this thesis, where we prove the exponential growth factors for the D-finite models by bounding the exponential growth and applying a squeezing argument. Upper bounds are found in a unified way via a relaxation of the restrictions placed on the class. Lower bounds are given by a case analysis reduced from 23 to 8 cases thanks to Lemmas \ref{onestep} and \ref{twostep}.

\subsection{Previous Results}\label{qpprev}

In this section, we give some rigourous and experimental results of previous authors, highlighting how they are tied to the classification of the generating functions.

\subsubsection{Rigourous results using orbit sums}

In Chapter \ref{dir}, we used the kernel method, which takes a solution to the kernel equation and uses it to bind $t$ and the catalytic variable $y$, which marks the altitude of the endpoint, in such a way that we reduced the number of unknowns and found a solution. For 22 of the 23 finite group models, Bousquet-M\'{e}lou and Mishna apply a similar process to the functional equation satisfied by the generating functions of those models in \cite{BoMi10}. We briefly discuss the process here, leaving the details to the source material.

First, we rewrite Equation (\ref{qpfe}) as
\begin{eqnarray*}
  K(x,y;t)xyQ_\kS(x,y;t) = xy &-& xyt\left(\{x^{<0}\}S(x,y)Q_\kS(0,y;t) \right. \\
  && \hspace{2em} + \left. \{y^{<0}\}S(x,y)Q_\kS(x,0;t)\right) + t\omega(\kS)Q_\kS(0,0;t).
\end{eqnarray*}
The multiplication by $xy$ will clear $x$ and $y$ from the denominators in every term, thus $xy\{x^{<0}\}S(x,y)$ is now a polynomial in $y$ and similarly $xy\{y^{<0}\}S(x,y)$ is a polynomial in $x$. Now let $F(x) = xy\{y^{<0}\}S(x,y)Q_\kS(x,0;t)$ and $G(y) = xy\{x^{<0}\}S(x,y)Q_\kS(0,y;t)$, so the above equation becomes
\begin{equation}\label{osum}
 K(x,y;t)xyQ_\kS(x,y;t) = xy - tF(x) - tG(y) + \omega(\kS)tQ_\kS(0,0;t).
\end{equation}

Because the unknowns are no longer mixed as functions of $x$ and $y$, the resulting equation is well suited to the orbit sum method.

\subsubsection{The orbit sum method}

The group $G(\kS)$ is defined as a group of birational transformations of the plane preserving the kernel $K(x,y;t) = 1 - tS(x,y)$. This is a dihedral group with two generators
\[
 \Phi: (x,y) \mapsto (x',y) \mbox{ and } \Psi : (x,y) \mapsto (x,y'),
\]
where the terms $x'$ and $y'$ are rational functions in $x$ and $y$ and depend on the step set $\kS$. The fact that each generator acts only on one of the variables $x$ and $y$ is the key concept here. Acting on Equation (\ref{osum}) with $\Phi$, we find
\begin{equation}\label{osumphi}
 K(x,y;t)x'yQ_\kS(x',y;t) = x'y - tF(x') - tG(y) + \omega(\kS)tQ_\kS(0,0;t).
\end{equation}
Then we take the difference of Equations (\ref{osum}) and (\ref{osumphi}), eliminating the term $tG(y)$:
\begin{equation}\label{osum1}
 K(x,y;t)\left( xyQ_\kS(x,y;t) - x'yQ_\kS(x',y;t)\right) = xy - x'y - tF(x) + tF(x').
\end{equation}
The next step is to act on Equation (\ref{osumphi}) with $\Psi$, generating a new equation, and adding it to Equation (\ref{osum1}), continuing to alternately act with $\Phi$ and $\Psi$ and taking the alternating sum. If the group is finite, as it is for our 23 cases, this process will terminate and all unknown functions on the right hand side will vanish. This leads to the equation
\[
 \sum_{g \in G} \mbox{sgn}(g)g \cdot \left(xyQ_\kS(x,y;t)\right) = \frac{1}{K(x,y;t)}\sum_{g \in G} \mbox{sgn}(g) g\cdot(xy).
\]
The coefficient of each summand is $\mbox{sgn}(g) = (-1)^k$, where $k$ is the number of generators~$\Phi$ and $\Psi$ in the decomposition of $g$. For 19 cases, this is enough to show D-finiteness. Essentially, the generating function is expressed as the constant term of a multivariate rational function. Since D-finite functions are closed under Hadamard products \cite{Li89}, it follows quickly. For enumerative purposes, we can arrive at an expression for the generating function as a rational series extraction
\[
 xyQ_\kS(x,y;t) = \{x^{>0}y^{>0}\}\frac{1}{K(x,y;t)}\sum_{g \in G} \mbox{sgn}(g) g\cdot(xy).
\]
From this, we may derive the formula for the number of walks of length $n$ ending at $(i,j)$:
\[
 q(i,j;n) = [x^{i+1}y^{j+1}]\left(\sum_{g \in G} \mbox{sgn}(g) g\cdot(xy) \right)S(x,y)^n.
\]

For the remaining four models the approach isn't as useful because the right hand side for each vanishes, leaving
\[
 \sum_{g \in G} \mbox{sgn}(g)g \cdot \left(xyQ_\kS(x,y;t)\right) = 0.
\]
For one of these models known as Gessel's model, $\kS_{23}$, we have not mastered the use of rotational symmetry and so this model cannot be handled by this method in a straight forward manner. Bostan and Kauers treat this model in \cite{BosKau09}. For the remaining three, it is due to an $x/y$ symmetry: the presence of $(y,x)$ in the orbit of $(x,y)$. Bousquet-M\'{e}lou and Mishna approach this with an alternating sum over half of the orbit instead, and manage to solve these three models.

While these formulae and expressions provide invaluable tools to generate enumerative data and a classification of the generating functions for 22 of the 23 cases, the aim was not to provide asymptotic data. The rational series extraction gives no hint as to the location of the singularities of the generating function. Series generation is easy to do with help from computers and this leads to another approach. The next section will outline a very succesful experimental regime employed by Bostan and Kauers to this end.

\subsubsection{Bostan and Kauers' experimental approach}

Recall that the asymptotic growth of a counting sequence is in correspondence with the radius of convergence of the associated generating function. In \cite{BoKa09}, Bostan and Kauers use the D-finiteness, as well as other readily available information about the model, to conjecture the radius of convergence of each model and the asymptotic expressions.

Recall that a D-finite generating function satisfies a linear differential equation of order~$r$ with polynomial coefficients $c_i(t)$, where $c_0(t)$ is the leading coefficient. If we assume the existence of a simply connected domain on which the coefficients are analytic, the existence theorem on page 3 of \cite{Wa65} guarantees that in a neighbourhood of a point $t_0$ such that $c_0(t_0) \neq~0$,~there exist at most $r$ linearly independent analytic solutions to the above differential equation. This implies that singularities can only occur at points which are roots of the leading coefficient $c_0(t)$, giving an upper bound of $\deg(c_0(t))$ candidates for the dominant singularity.

Next, recall that Proposition \ref{gseries} bounds the radius of convergence of a G-series below by~$1/|\kS|$, further reducing the number of candidates for the dominant singularity.

So, to use these facts Bostan and Kauers required a differential equation satisfied by the generating function of a given model. By generating all walks on a given step set up to length $N$ (at that point, the maximum achievable was around $N = 1000$, but now with the help of Steve Melczer, it is around $N = 2000$) and searching the Weyl algebra $\QQ[t]\langle D_t \rangle$ of differential operators satisfied by the series modulo $t^N$, they were able to find a differential equation satisfied by the partial sum. The idea is that if the generating function happens to be D-finite, then for a sufficiently large $N$, a differential equation satisfied modulo $t^N$ will provide a differential equation which is really satisfied by the full power series. Analysis of the polynomial coefficients of the found differential equation yields the asymptotic information found in Table \ref{fintab}.

\begin{table}
\begin{center}
\begin{tabular}{| c | c | c | c | c |}
  \hline
  $i$ & $\kS_i$ & $\kappa$ & $\alpha$ & $\beta$ \\ \hline
  &&&&\\[-5pt] 
  1 &  \diagr{N,S,E,W} & $\frac4\pi$ & -1 & 4 \\ [+2mm]
  2 & \diagr{NW,SW,NE,SE} & $\frac2\pi$ & -1 & 4 \\  [+2mm]
  3 & \diagr{N,S,NE,NW,SW,SE} & $\frac{\sqrt{6}}\pi$ & -1 & 6 \\ [+2mm]
  4 & \diagr{N,S,NE,NW,SW,SE,E,W} & $\frac8\pi$ & -1 & 8 \\ [+2mm]
  5 & \diagr{S,NE,NW} & $\frac{\sqrt{3}}{\sqrt{\pi}}$ & $-\frac12$ & 3 \\ [+2mm]
  6 & \diagr{S,NE,NW,E,W} & $\frac{\sqrt{5}}{2\sqrt{2}\sqrt{\pi}}$ & $-\frac12$ & 5 \\ [+2mm]
  7 & \diagr{S,NE,NW,N} & $\frac4{3\sqrt{\pi}}$ & $-\frac12$ & 4  \\ [+2mm]
  8 & \diagr{S,NE,NW,N,E,W} & $\frac{2\sqrt{3}}{3\sqrt{\pi}}$ & $-\frac12$ & 6 \\ [+2mm]
  9 & \diagr{SW,SE,NE,NW,N} & $\frac{\sqrt{5}}{3\sqrt{2}\sqrt{\pi}}$ & $-\frac12$ & 5 \\ [+2mm]
  10 & \diagr{SW,SE,NE,NW,N,E,W} & $\frac{\sqrt{\frac73}}{3\sqrt{\pi}}$ & $-\frac12$ & 7 \\ [+2mm]
  11 & \diagr{S,SW,SE,N} & $\frac{12\sqrt{3}}\pi$ & $-2$ & $2\sqrt{3}$ \\ [+2mm]
  12 & \diagr{S,SW,SE,N,E,W} & $\frac{\sqrt{3}(1 + \sqrt3)^{\frac72}}{2\pi}$ & $-2$ & $2(1 + \sqrt3)$ \\ [+2mm]
  13 & \diagr{S,SW,SE,NW,NE} & $\frac{12\sqrt{30}}{\pi}$ & $-2$ & $2\sqrt6$ \\ [+2mm]
  14 & \diagr{S,SW,SE,NW,NE,E,W} & $\frac{\sqrt{6(376+156\sqrt6)}(1 + \sqrt6)^{\frac72}}{5\sqrt{95}\pi}$ & $-2$ & $2(1 + \sqrt6)$ \\ [+2mm]
  15 & \diagr{SW,SE,N} & $\frac{24\sqrt{2}}{\pi}$ & $-2$ & $2\sqrt2$ \\ [+2mm]
  16 & \diagr{SW,SE,N,E,W} & $\frac{\sqrt{8}(1 + \sqrt2)^{\frac72}}{\pi}$ & $-2$ & $2(1 + \sqrt2)$ \\ [+2mm]
  17 & \diagr{N,W,SE} & $\frac{3\sqrt{3}}{2\sqrt{\pi}}$ & $-\frac32$ & 3 \\ [+2mm]
  18 & \diagr{N,S,E,W,NW,SE} & $\frac{3\sqrt{3}}{2\sqrt{\pi}}$ & $-\frac32$ & 6 \\ [+2mm]
  19 & \diagr{S,W,NE} & $\frac{2\sqrt{2}}{\Gamma(\frac14)}$ & $-\frac34$ & 3 \\ [+2mm]
  20 & \diagr{N,E,SW} & $\frac{3\sqrt{3}}{\sqrt2\Gamma(\frac14)}$ & $-\frac34$ & 3 \\ [+2mm]
  21 & \diagr{N,S,E,W,NE,SW} & $\frac{\sqrt{2}3^{\frac34}}{\Gamma(\frac14)}$ & $-\frac34$ & 6 \\ [+2mm]
  22 & \diagr{NW,SE,W,E} & $\frac{8}{\pi}$ & $-2$ & 4 \\ [+2mm]
  23 & \diagr{NE,SW,W,E} & $\frac{4\sqrt3}{3\Gamma(\frac13)}$ & $-\frac23$ & 4 \\ \hline
\end{tabular}
\caption{The number of walks $q_i(n)$ is conjectured to grow like $\kappa \beta^n n^\alpha$ (\cite[Table 1]{BoKa09}).}\label{fintab}
\end{center}
\end{table}

The column of exponential growth factors $\beta$ provide us with strong candidates for the rigorous results. In fact, while these results are not proved rigorously in \cite{BoKa09}, the approach to do so is similar to that used in a paper by the same authors on Gessel's walks \cite{BosKau09}. They are omitted from \cite{BoKa09} because the computations required are quite large, motivating the development of another approach with more elementary techniques. 

Our approach, developed over the next three sections, uses restrictions and relaxations of the constraints on a class to bound the number of walks by known quantities, and a small amount of analysis to squeeze the exponential growth. Thanks to Lemmas \ref{onestep} and~\ref{twostep} in Section \ref{lbs}, we are able to reduce the cases analysis required to 8 base cases. We begin by proving the exponential growth of these base cases in the following section.

\subsection{Base cases}\label{base}

In Section \ref{lbs}, we give a way of importing lower bounds on the exponential growth of a step set $\kS$ from proper subsets of $\kS$. This gives us a way of reducing our 23 cases down to~8 base cases which are true quarter plane walks. There are two models which each import exponential growth from a half plane model, but in both cases the model is a class of Dyck prefixes, so we don't prove these base cases here. This section will prove the exponential growth factors for each quarter plane base case, first giving the base cases we can fully enumerate, then giving those for which we must enumerate a sufficiently large subset.

\subsubsection{Full counts}\label{basefc}
There are two base case step sets which have quarter plane models susceptible to direct counting techniques. These are $\kS_2$ and $\kS_{17}$, pictured below in Figure \ref{dc}

\begin{figure}[h]
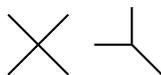

 \begin{center}
  \diagb{NE,SE,SW,NW} \diagb{N,SE,W}
 \end{center}
  \caption{The step sets which we count directly.}\label{dc}
\end{figure}

Let us move through them in order. The first, $\kS_2 = \diagr{NE,SE,SW,NW}$ is pretty straightforward.

\begin{prop}\label{S02}
The exponential growth factor for the counting sequence $q_2(n)$ is 4.
\end{prop}

\begin{proof}
For $\kS_2$, we have a simpler count. If we take a walk on $\kS_2$ and look at the projection onto each axis, we see that the endpoint of the walk traces out a Dyck prefix on either axis. Moreover, since every step is positive or negative in both $x$ and $y$, the endpoint projection moves on both axes at every step. So, if $q_2(n)$ is the number of walks of length $n$ on $\kS_2$ and~$d(n)$ is the number of Dyck prefixes, then
\[
 q_2(n) = d(n)^2 = \binom{n}{\lfloor \frac{n}{2} \rfloor}^2.
\]
An application of Stirlings approximation
\[
 n! \sim \sqrt{2\pi n} \left( \frac{n}{e} \right)^n
\]
gives the result.
\end{proof}

Finally, we move on to $\kS_{17} = \diagr{N,SE,W}$. The proof is taken from an informal discussion in \cite{BoMi10}, which can also be found in \cite[Section 3]{Mi09}.

\begin{prop}\label{S17}
\cite[Section 5.2]{BoMi10} The exponential growth factor for the counting sequence $q_{17}(n)$ of walks on $\kS_{17} = \diagr{N,SE,W}$ is 3.
\end{prop}

\begin{proof}
To count walks on $\kS_{17} = \diagr{N,W,SE}$, note that before we may use $\textbf{W}$, which is $x$-negative, we must use an $x$-positive step. But the only $x$-positive step is $\textbf{SE}$, which is~$y$-negative, and before we may use it we must use the $\textbf{N}$ step. Thus, the number of $\textbf{N}$ steps must outweigh the number of $\textbf{SE}$ steps which must in turn outweigh the number of $\textbf{W}$ steps. This puts the walks in bijection with Standard Young tableaux of height at most 3, which we explain presently.

Standard Young tableaux are labelled Ferrer’s diagrams of a partition such that the boxes are labelled in a strictly increasing manner from left to right and from bottom to top. We construct a standard Young tableau of height three of size n from a walk $w = (s_1 , s_2 , . . . , s_n)$ as follows. If $s_i = \textbf{N}$ (resp. \textbf{SE, W}), then place label $i$ in the next available space to the right on the bottom row (resp. second, top). The final tableau is increasing from left to right by construction, and the prefix condition $\#\textbf{N} \geq \#\textbf{SE} \geq \#\textbf{W}$ ensures that it is increasing along the columns.

Standard Young tableaux are well understood and may be enumerated via the hook-length formula. The number of standard Young tableaux of height 3 and size $n$ is the $n$th Motzkin number. An application of Stirling's approximation will give the result.
\end{proof}

Next, we move on to the longer list of cases for which we must count a sufficiently large subset.

\subsubsection{Counting a subset}\label{basess}
For the majority of the base case step sets, closed form enumeration of all walks of a given length is beyond the reach of current techniques. We look instead for easier to count subsets with the same exponential growth. Some will be walks returning to the origin, as in step set 11, others will be walks ending on a boundary. Often, we appeal to an upper bound given by the half plane walks taken on the same step set. For the proof that these upper bounds have the same exponential growth, see Theorem \ref{bnd} in Section \ref{ubs}. We begin by counting a subset of $\kS_5 = \{ \textbf{NE,~S,~NW}\}$ by injecting walks on another step set. This isn't really a base case, but the method of proof is different enough that it warrants inclusion here.

\begin{prop}\label{S05}
The counting sequence $q_5(n)$ of walks on $\kS_5 = \diagr{NE,NW,S}$ has exponential growth factor $3$. 
\end{prop}

\begin{proof}
Note that we may not count walks returning to the origin: these have the same count as walks taken on the negative drift counterpart of this steps set,  $\kS_{15}$, so we must be a little clever. Consider for a moment the step set $\cT = \{\textbf{NW,~E,~S}\}$, shown in Figure \ref{17'}. This set is a reflection of $\kS_{17}$ across the line $y = x$, so the number of walks of length $n$ is the $n$th Motzkin number.

\begin{figure}
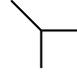

 \begin{center}
  \diagb{NW,E,S}
 \end{center}
  \caption{The helper set}\label{17'}
\end{figure}

Note that if we rotate a step so it takes us further from any boundary we are doing no harm. So if we rotate the \textbf{E} step so it becomes a \textbf{NE} step, all allowed walks from $\cT$ are still allowed on the new step set. This is the same as injecting the walks taken on $\cT$ as a subset of the walks taken on $\kS_{5}$ with the function
\begin{eqnarray*}
 f: \cT &\rightarrow& \kS_{5} \\
  \textbf{NW} &\mapsto& \textbf{NW} \\
  \textbf{E} &\mapsto& \textbf{NE} \\
  \textbf{S} &\mapsto& \textbf{S}.
\end{eqnarray*}

Thus, we have
\[
 M_n \leq q_5(n) \leq h_5(n).
\]
Taking $\lim \frac1n \log$ of each term gives
\[
 \beta_5 = 3.
\]
\end{proof}

Moving on to the next base case, $\kS_{11} = \{ \textbf{N,~SE,~S,~SW} \}$, which is the first negative drift base case.

\begin{prop}\label{S11}
The exponential growth factor of the counting sequence $q_{11}(n)$ for quarter plane walks on $\kS_{11} = \diagr{N,S,SW,SE}$ is $2\sqrt3$.
\end{prop}

\begin{proof}
Consider the walks on this step set which return to the origin. Notice that any walk returning to the $x$-axis must be of even length, since every step is either $y$-positive or $y$-negative. Thus, any walk returning to the origin must also be of even length. We begin by considering a walk of length $2n$ taken only on the \textbf{N} and \textbf{S} steps, returning to the origin. This is simply a projection of some Dyck path onto the $y$-axis, and these are counted by the Catalan numbers.

We replace some of the \textbf{S} steps by \textbf{SW} and \textbf{SE} steps. Since the walk must return to the origin, the walk on $\{\textbf{SW,~SE}\}$ must form a Dyck path with respect to the $y$-axis, which must also be of even length. These too are counted by the Catalan numbers. Moreover, we must choose some even number of the $y$-negative steps to replace, so there will be a summation of binomials also. Explicitly, the number of walks returning to the origin on $\kS_{11}$ is given by
\[
 q_{11}(0,0;2n) = C_n \sum_{k = 0}^{\lfloor n/2 \rfloor} \binom{n}{2k} C_k.
\]
The quantity
\[
 \sum_{k = 0}^{\lfloor n/2 \rfloor} \binom{n}{2k} C_k
\]
is an expression for the $n$th Motizkin number $M_n$, so we may rewrite the above as
\[
 q_{11}(0,0;2n) = C_nM_n.
\]

By applying Stirling's approximation, we know that for $C,D,r,s \in \RR$, we have
\[
 M_n \sim C 3^n n^r \mbox{ and } C_n \sim D 4^n n^s,
\]
so given some $0 < \epsilon < 1$, we know there exists some $N$ sufficiently large such that
\[
 1 - \epsilon \leq \frac{q(0,0;2n)}{CD(2\sqrt3)^{2n}n^{r+s}} \leq 1 + \epsilon
\]
for all $n > N$, that is, the exponential growth factor of $q_{11}(0,0;2n)$ is $2\sqrt3$.

Thus
\[
 q_{11}(0,0;2n) \bowtie (2\sqrt3)^{2n},
\]
and we may apply $\lim\frac1n\log$ to each term in the inequality
\[
 q_{11}(0,0;n) \leq q_{11}(n) \leq h_{11}(n),
\]
taking care to take the limit supremum on the left, to prove the result.
\end{proof}

For $\kS_{13} = \diagr{NE,NW,SE,SW,S}$, another negative drift step set, we look to the projection of the walk on to either axis, similar to the argument for $\kS_2$, but with a bit more analysis.

\begin{prop}\label{S13}
The counting sequence $q_{13}(n)$ for QP walks on $\kS_{13} = \diagr{NE,NW,SE,SW,S}$ has exponential growth factor~$2\sqrt6$.
\end{prop}

\begin{proof}
As in $\kS_2$, we look at the projection of a typical walk on the $x$ and $y$ axes. Since all steps are either $y$-positive or negative, the projection onto the $y$ axis forms a Dyck prefix. Let us consider the projection onto the $x$-axis. There is one step with no horizontal component, and so the projection will form a Motzkin prefix, but with the stationary steps forced to correpsond to the downward steps in the projected Dyck prefix. Figure \ref{S13ex} shows an example of a walk returning to the $x$-axis with the steps in blue highlighting this correpsondence.

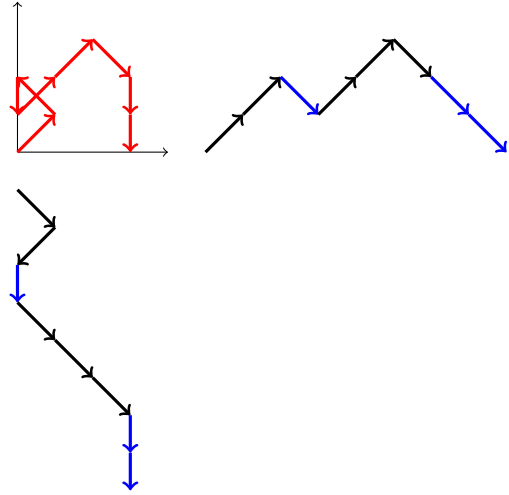
\begin{figure}
  \begin{center}
    \begin{tikzpicture}[scale = 0.5]
      \draw[->] (0,0) -- (4,0);
      \draw[->] (0,0) -- (0,4);
      \draw[->,very thick,red] (0,0) -- (1,1);
      \draw[->,very thick,red] (1,1) -- (0,2);
      \draw[->,very thick,red] (0,2) -- (0,1);
      \draw[->,very thick,red] (0,1) -- (1,2);
      \draw[->,very thick,red] (1,2) -- (2,3);
      \draw[->,very thick,red] (2,3) -- (3,2);
      \draw[->,very thick,red] (3,2) -- (3,1);
      \draw[->,very thick,red] (3,1) -- (3,0);
      \draw[->, very thick, black] (5,0) -- (6,1);
      \draw[->, very thick, black] (6,1) -- (7,2);
      \draw[->, very thick, blue] (7,2) -- (8,1);
      \draw[->, very thick, black] (8,1) -- (9,2);
      \draw[->, very thick, black] (9,2) -- (10,3);
      \draw[->, very thick, black] (10,3) -- (11,2);
      \draw[->, very thick, blue] (11,2) -- (12,1);
      \draw[->, very thick, blue] (12,1) -- (13,0);
      \draw[->, very thick, black] (0,-1) -- (1,-2);
      \draw[->, very thick, black] (1,-2) -- (0,-3);
      \draw[->, very thick, blue] (0,-3) -- (0,-4);
      \draw[->, very thick, black] (0,-4) -- (1,-5);
      \draw[->, very thick, black] (1,-5) -- (2,-6);
      \draw[->, very thick, black] (2,-6) -- (3,-7);
      \draw[->, very thick, blue] (3,-7) -- (3,-8);
      \draw[->, very thick, blue] (3,-8) -- (3,-9);
    \end{tikzpicture}
    \caption{A walk taken on $\kS_{13}$ shown with the projection onto each axis.}\label{S13ex}
  \end{center}

\end{figure}

In order to simplify computations, we consider the walks ending on the $x$-axis. Thus the projection onto the $y$-axis forms a Dyck path, something that we are quite familiar with counting: let the number of projections onto the $y$-axis be $p_y(n)$. Then $p_y(n) = C_n$. The projection onto the $x$-axis forms a Motzkin prefix, and as already mentioned, the `stationary' steps correspond only to $y$-negative steps in the original walk.

Now count the number of Motzkin prefixes with this restriction. The walks return to the $x$-axis, and since all steps are either $y$-positive or negative the walks must be of even length. Of the $n$ downward steps, $k$ are chosen to be the steps with no $x$ component. The remainder of the $2n$ steps must form a Dyck prefix, giving the formula
\[
 p_x(n) = \sum_{k = 0}^n \binom{n}{k} d(2n-k).
\]

From Chapter \ref{dir}, we know the asymptotic growth for Dyck prefixes is
\[
 d(2n-k) \sim C 2^{2n-k}(2n-k)^r,
\]
where $C$ is a positive real constant and $r \leq 0$. Thus for $0 < \epsilon <1$, there is some $N$ such that~$n > N$ implies
\begin{eqnarray*}
  \sum_{k=0}^N \binom{n}{k} d(2n-k) &+& \sum_{k = N+1}^n \binom{n}{k} (1 - \epsilon) C 2^{2n-k} (2n-k)^r \\
  && \leq p_x(n) \leq \sum_{k=0}^N \binom{n}{k} d(2n-k) + \sum_{k = N+1}^n \binom{n}{k} (1 + \epsilon) C 2^{2n-k} (2n-k)^r.
\end{eqnarray*}

Let's look at the lower bound first. We have
\begin{eqnarray*}
 p_x(n) 
  &\geq& \sum_{k=0}^N \binom{n}{k} d(2n-k) + \sum_{k = N+1}^n \binom{n}{k} (1 - \epsilon) C 2^{2n-k} (2n-k)^r, \\
  &\geq& \sum_{k=0}^N \binom{n}{k} d(2n-k) + (1 - \epsilon) C 2^{2n} (2n)^r \sum_{k = N+1}^n \binom{n}{k} \left(\frac12\right)^k \mbox{since } r \leq 0, \\
  &\geq& \sum_{k=0}^N \left( \binom{n}{k} d(2n-k) - C 2^{2n} (2n)^r \left(\frac12\right)^k \right) + (1 - \epsilon) C 2^{2n} (2n)^r \sum_{k = 0}^n \binom{n}{k} \left(\frac12\right)^k \\
  &\geq& \sum_{k=0}^N  \binom{n}{k} \left( d(2n-k) - C 2^{2n} (2n)^r \left(\frac12\right)^k \right) + (1 - \epsilon) C 2^{2n} (2n)^r \left(\frac32 \right)^n \mbox{binomial theorem,} \\
  &\geq& \sum_{k=0}^N \binom{n}{k}\left( d(2n-k) - C 2^{2n} (2n)^r \left(\frac12\right)^k \right) + (1 - \epsilon) C (\sqrt6)^{2n} (2n)^r.
\end{eqnarray*}

Now, the sum with upper bound of summation $N$ is the error term. Note that the constant upper limit of summation makes the growth as $n \rightarrow \infty$ of the positive portion polynomial, and the negative portion has its exponential growth dominated by $(\sqrt6)^{2n}$. Thus there exists some new $N^-$ such that $n > N^-$ forces
\[
 \sum_{k=0}^N \binom{n}{k}\left( d(2n-k) - C 2^{2n} (2n)^r \left(\frac12\right)^k \right) \leq \frac{\epsilon}2 C (\sqrt6)^{2n} (2n)^r,
\]
so for $n > N^-$
\[
 (1 - \epsilon/2) C (\sqrt6)^{2n} (2n)^r \leq p_x(n).
\]
The upper bound is similar, giving
\[
 p_x(n) \leq (1 - \epsilon) C (\sqrt{6})^{2n}
\]
for $n$ sufficiently large. This proves
\[
 p_x(n) \bowtie (\sqrt6)^{2n}.
\]

Recall that the number of walks returning to the $x$-axis on $\kS_{13}$ is the product
\[
 \sum_{i \geq 0} q_{13}(i,0;2n) = p_y(n)p_x(n).
\]
We already know that $p_y(n) \sim \frac{2^{2n}}{n^{3/2}\sqrt\pi}$, and we just showed that $p_x(n) \sim D (\sqrt6)^{2n} (2n)^{s}$ for some positive $D$ and $s \leq 0$, so for some $D'$ and $s'$ we have
\[
 \sum_{i \geq 0} q_{13}(i,0;2n) \sim D'(2\sqrt6)^{2n}(2n)^{s'}.
\]
Thus given the inequality
\[
 \sum_{i \geq 0} q_{13}(0,1;n) \leq q_{13}(n) \leq h_{13}(n),
\]
we know that for $n$ sufficiently large we substitute each quantities respective growth function, and working in the usual way (taking the $\limsup$ on the left) we find
\[
 \beta_{13} = 2\sqrt6.
\]
\end{proof}

This brings us to $\kS_{15} = \diagr{N,SW,SE}$. We find the exponential growth by considering walks returning to the origin.

\begin{prop}\label{S15}
 The exponential growth factor for the counting sequence $q_{15}(n)$ for QP walks on $\kS = \diagr{N,SW,SE}$ is $2\sqrt2$.
\end{prop}

\begin{proof}
The walks returning to the origin are easy to count. Every walk which returns to the origin must be of length $4n$, since there must be the same number of $y$-positive steps as $y$-negative, and the same number of $x$-positive as $x$-negative. Now, if we look at the projection of a walk that returns to the origin onto the $y$-axis, the endpoint traces out a Dyck path of length $4n$. Then we look at the subwalk taken on the $x$-negative steps. This is a Dyck path of length $2n$. So, we have
\[
 q_{15}(0,0;4n) = \cC_{2n}\cC_n \sim C4^{2n}4^{n}(4n)^r = C(2\sqrt2)^{4n}(4n)^r
\]
for some $C \in \RR^+$ and $r \leq 0$. We then get
\[
 q_{15}(0,0;n) \leq q_{15}(n) \leq h_{15}(n),
\]
which we may work on in the same way as number $\kS_{11}$: by taking the $\limsup$ on the left, and getting
\[
 \beta_{15} = 2\sqrt2.
\]
\end{proof}

The next base case is $\kS_{19} = \{ \textbf{NE,~S,~W}\}$. This is a very famous model of a quarter plane walk, and is known as Kreweras' model, since he was the first to study it. We prove this factor along with the factor for $\kS_{20}$, since one step set is the reverse of the other.

\begin{prop}\label{S19}\label{S20}
\cite{Kr65} The exponential growth factor of the counting sequences $q_{19}(n)$ and $q_{20}(n)$, for QP walks on $\kS_{19} = \diagr{NE,S,W}$ and $\kS_{20} = \diagr{N,E,SW}$ respectively, is~$\beta_{19} = \beta_{20} = 3$.
\end{prop}

\begin{proof}
In 1965, Kreweras proved in \cite{Kr65} that the number of walks on $\kS_{19}$ of length $3n$, terminating at the origin, is given by
\[
 k_n = \frac{4^n}{(n+1)(2n+1)}\binom{3n}{n}.
\]
An application of Stirling's approximation tells us that the exponential growth of this quantity is $3^{3n}$, as desired. This provides a lower bound on the exponential growth of the unrestricted quarter plane model, giving the first part of the result.

The next base case to count is the step set $\kS_{20} = \{ \textbf{N,~E,~SW}\}$. Note that this is the previous case in reverse, so the number of walks returning to the origin of length $3n$ is $k_n$, which we just showed has an exponential growth factor $\beta_{20} = 3$, giving a lower bound on the quarter plane model and proving the second part of the result.
\end{proof}

The next base case, $\kS_{22} = \{ \textbf{E,~SE,~W,~NW}\}$ is an interesting case. It is known as the Gouyou-Beauchamps model. There are a few different ways to obtain the lower bound we need here. The proof is taken from \cite{BoMi10}.

\begin{prop}\label{S22}
\cite[Proposition 11]{BoMi10} The counting sequence $q_{22}(n)$ of QP walks on $\kS_{22} = \diagr{E,SE,W,NW}$ has exponential growth factor $\beta_{22} = 4$.
\end{prop}

\begin{proof}
First of all, we notice a bijection between this model and pairs of non-intersecting Dyck prefixes. To pass from a pair to a quadrant walk, we must assign a direction to each possible pair of steps. Hence, the number of walks ending at a given position $(i,j)$ is given as a 2-by-2 Gessel-Viennot determinant, and we can expect a closed form expression for this number.

Next, note that we may make the transformation $(i,j) \mapsto (i+j,j)$ to take quarter plane walks injectively to walks in $\kS_1 = \{ \textbf{N,~S,~E,~W}\}$ lying between the $x$-axis and the line~$y = x$. This is the form they were studied in by Gouyou-Beauchamps, who provided a formula for the number of $n$ step walks ending on the $x$-axis as a product of Catalan numbers in~\cite{Go86}. This interest grew from a far from obvious bijection between these walks and Young tableuax of height at most 4. The Gouyou-Beauchamps closed form expression for walks ending at the origin is
\[
 q(0,0;2m) = \frac{6(2m)!(2m + 2)!}{m!(m+1)!(m+2)!(m+3)!} \sim \frac{24 \cdot 4^{2m}}{\pi m^5}.
\]
This has the desired exponential growth, and so may be used as a lower bound to prove~$\beta_{22}=4$.
\end{proof}

This concludes the proof of the exponential growth of the base cases. In the following section, we show how to tightly bound above the exponential growth of each step set by relaxing the regional constraints on the class.

\subsection{Upper bounds on the exponential growth}\label{ubs}

We wish to use the $\lim\frac1n\log$ technique to squeeze the exponential growth of the counting sequences for each $\kS_i$. In this section, the main result is the following.

\begin{thm}\label{bnd1}
(Upper bound) For $i = 1,...,23$, the exponential growth factor for the number of walks taken on $\kS_i$ is bounded above by the value $\beta_i$ given in Table \ref{summ}.
\end{thm}

Now, we know intuitively (and we proved in Section \ref{qpclass}) that since the number of walks on a step set $\kS$ cannot exceed $|\kS|^n$, if $\beta$ is the exponential growth factor for the number of walks on $\kS$, we have
\[
 \beta \leq |\kS|.
\]
It turns out that for most of the 23 D-finite cases this bound is tight. We were intrigued by the fact that for those~$\kS_i$ with $y$-negative drift, this bound is too large. This is because bounding by the number of steps is the same as removing both boundaries, and considering $\cQ_\kS$ as a subset of~$\cW_{\kS}$, and the enumeration of $\cW_{\kS}$ doesn't depend on the drift, and as can be seen from Chapters~\ref{dir} and \ref{hpwalks}, when a boundary is introduced drift influences the enumerative properties of a model.

As an alternative bound, we consider removing only one boundary, as shown in Figure \ref{rembnd}. Note that for $i \in \{1,...,23\}$,~if the drift is non-zero it is in the $y$-direction only and therefore will be preserved by a horizontal projection. Thus, considering $\cQ_\kS$ as a subset of $\cH_\kS$ will give drift dependent upper bounds on the exponential growth.

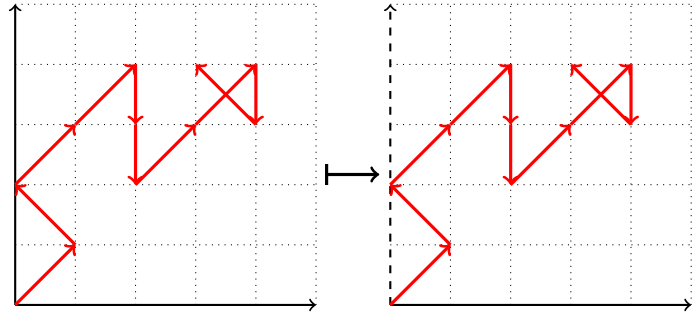
\begin{figure}
\begin{center}
  \begin{tikzpicture}[scale = 0.8]
    \draw[->, thick, black] (0,0) -- (5,0);
    \draw[->, thick, black] (0,0) -- (0,5);
    \draw[-, dotted, black] (1,0) -- (1,5);
    \draw[-, dotted, black] (2,0) -- (2,5);
    \draw[-, dotted, black] (3,0) -- (3,5);
    \draw[-, dotted, black] (4,0) -- (4,5);
    \draw[-, dotted, black] (5,0) -- (5,5);
    \draw[-, dotted, black] (0,1) -- (5,1);
    \draw[-, dotted, black] (0,2) -- (5,2);
    \draw[-, dotted, black] (0,3) -- (5,3);
    \draw[-, dotted, black] (0,4) -- (5,4);
    \draw[-, dotted, black] (0,5) -- (5,5);
    \draw[->, very thick, red] (0,0) -- (1,1);
    \draw[->, very thick, red] (1,1) -- (0,2);
    \draw[->, very thick, red] (0,2) -- (1,3);
    \draw[->, very thick, red] (1,3) -- (2,4);
    \draw[->, very thick, red] (2,4) -- (2,3);
    \draw[->, very thick, red] (2,3) -- (2,2);
    \draw[->, very thick, red] (2,2) -- (3,3);
    \draw[->, very thick, red] (3,3) -- (4,4);
    \draw[->, very thick, red] (4,4) -- (4,3);
    \draw[->, very thick, red] (4,3) -- (3,4);
  \end{tikzpicture}
  \begin{tikzpicture}[scale = 0.7]
      \draw[-,very thick,black] (0,-0.2) -- (0,0.2);
      \draw[-,white] (0.5,-0.1) -- (0.5,-2.5);
      \draw[->,very thick,black] (0,0) -- (1,0);
  \end{tikzpicture}
  \begin{tikzpicture}[scale = 0.8]
    \draw[->, thick, black] (0,0) -- (5,0);
    \draw[->, thick, dashed, black] (0,0) -- (0,5);
    \draw[-, dotted, black] (1,0) -- (1,5);
    \draw[-, dotted, black] (2,0) -- (2,5);
    \draw[-, dotted, black] (3,0) -- (3,5);
    \draw[-, dotted, black] (4,0) -- (4,5);
    \draw[-, dotted, black] (5,0) -- (5,5);
    \draw[-, dotted, black] (0,1) -- (5,1);
    \draw[-, dotted, black] (0,2) -- (5,2);
    \draw[-, dotted, black] (0,3) -- (5,3);
    \draw[-, dotted, black] (0,4) -- (5,4);
    \draw[-, dotted, black] (0,5) -- (5,5);
    \draw[->, very thick, red] (0,0) -- (1,1);
    \draw[->, very thick, red] (1,1) -- (0,2);
    \draw[->, very thick, red] (0,2) -- (1,3);
    \draw[->, very thick, red] (1,3) -- (2,4);
    \draw[->, very thick, red] (2,4) -- (2,3);
    \draw[->, very thick, red] (2,3) -- (2,2);
    \draw[->, very thick, red] (2,2) -- (3,3);
    \draw[->, very thick, red] (3,3) -- (4,4);
    \draw[->, very thick, red] (4,4) -- (4,3);
    \draw[->, very thick, red] (4,3) -- (3,4);
  \end{tikzpicture}
  \caption{Removing the $y$-axis as a boundary to produce an upper bound on the exponential growth}\label{rembnd}
\end{center}
\end{figure}

For a given step set $\kS_i$, viewing a quarter plane walk as a restricted half plane walk will give the inequality
\[
 q(n) \leq h(n),
\]
where $q(n)$ and $h(n)$ count quarter and half plane walks, respectively, of length $n$. Then, for the right $C,D,\beta,\gamma,r,s$, we have
\[
 q(n) \sim C \beta^n n^s \mbox{ and } h(n) \sim D \gamma^n n^r,
\]
and for sufficiently large $n$ we will have
\[
 C \beta^n n^s \leq D \gamma^n n^r.
\]
Once again, taking $\lim \frac1n \log$ on both sides will yield
\[
 \log \beta \leq \log \gamma \Rightarrow \beta \leq \gamma.
\]
Now recall that the exponential growth factor for half-plane walks with zero or positive drift is equal to the cardinality of the step set, so this bound is as good as the previous one for these cases. Also, we expect it to improve in the negative drift cases since for the class~$\cH_\kS$ the drift affects the asymptotic growth. Let us consider an example.

\begin{ex}\label{ubex}
We showed in Proposition \ref{S11} that $\kS_{11} = \diagr{N,S,SE,SW}$ has a lower bound of~$2\sqrt{3}$ on the exponential growth, and appealed to an upper bound given by the half plane walks on the same step set. Here we produce this upper bound by appealing to the theorems on directed walks from Chapter \ref{dir}, applied through the bijection given in Chapter \ref{hpwalks} to half plane walks.

First, we represent $\kS_{11}$ by its inventory
\[
 S(x,y) = y + \frac{x}{y} + \frac1y + \frac{1}{xy},
\]
and take the horizontal projection by setting $x = 1$, taking the directed inventory. This gives us the weighted inventory for the directed meander on the horizontal projection $\cP$,
\[
 P(y) = y + \frac3y.
\]
We find the drift to be negative:
\[
 P'(y) = 1 - \frac3{y^2} \Rightarrow P'(1) = -2 < 0;
\]
so Theorem \ref{dirasympt} tells us that the exponential growth factor of $h_{11}(n)$ is equal to $P(\tau)$, where~$\tau$ is the positive solution to the equation
\[
 P'(\tau) = 0.
\]
So, we let
\[
 1 - \frac3{y^2} = 0
\]
which is easily solved for
\[
 \tau = \sqrt3.
\]
Plugging this back into $P(y)$ gives
\[
 \gamma = P(\tau) = 2\sqrt3,
\]
which is tight according to Proposition \ref{S11}.
\end{ex}

So in Example \ref{ubex} using the relaxation to a half plane model, we have found an upper bound for $\beta_{11}$ which is an improvement on the upper bound provided by the number of steps. It is an elementary application of the method given in Example \ref{ubex} to the remaining~5 negative drift cases to verify that it is a tight upper bound in all cases. This leads to the following result, a restatement of Theorem~\ref{bnd1} from the beginning of Section \ref{ubs} using a little more precision.

\begin{thm}\label{bnd}
The exponential growth factor of a quarter plane walk with a set of small steps $\kS_i$, $i=1,2,3,...,23$, is bounded above by
\begin{enumerate}
  \item the number of steps for a walk with non-negative drift,
  \item the quantity $P(\tau)$ for a walk with negative drift,
\end{enumerate}
where $P(y) = S(1,y)$ is the inventory of the horizontal projection of $\kS_i$ and $\tau$ is the structural constant of the half plane model.
\end{thm}

The proof of Theorem \ref{bnd} is simply an application of the method of Example \ref{ubex} in each of the cases and so is omitted. Now that we have a method for providing upper bounds which agree with the findings of Bostan and Kauers in all cases, we move on to finding lower bounds. This is a little more complicated than the upper bounds, and is detailed in the following section.

\subsection{Lower bounds on the exponential growth}\label{lbs}

In this section, we show how to find lower bounds on the exponential growth for the given models. In some cases, we are able to use a restriction, such as the walks returning to the origin or those returning to an axis, but this doesn't work for all cases. We are then forced to construct lower bounds from step subsets that we know the exponential growth of. We illustrate the intuition of this in the following example.

\begin{ex}\label{int}
Take $\kS_5 = \diagr{NE,S,NW}$, which we shall call the positive drift fork. Notice that if we take any walk of a given length on $\kS_5$ and insert a north step at any internal point, the remainder of the walk moves no closer to the $y$-axis and further from the $x$ axis, giving us a walk from $\kS_7 = \diagr{NE,NW,N,S}$. If $q_5(i)$ is the number of walks of length $i$ on $\kS_5$, and $q_5(i) \sim C3^ii^s$, we can build walks of length $n$ in~$\kS_7$ by taking walks of length $i$ on $\kS_5$ and inserting $n-i$ north steps. See Figure \ref{onefig} for an example, where we insert $\textbf{N}$ steps at the positions marked by a blue circle. The number of walks on $\kS_7$ built in such a way is
\[
 q'_7(n) = \sum_{i \geq 0} \binom{n}{i} q_5(i),
\]
where we denote it by $q'$ since in general, this method might not give all possible walks on the larger step set. Now, intuitively, since $q_5(n)$ has dominant growth $3^n$, we would like to get $q'_7(n)$ with dominant growth $4^n$ by applying the binomial theorem, like so
\[
 q'_7(n) \sim \sum_{i \geq 0} \binom{n}{i} 3^i = 4^n.
\]

\begin{figure}[h]
  \begin{center}
    \begin{tikzpicture}
      \draw[->,very thick,black] (0,0) -- (5,0);
      \draw[->,very thick,black] (0,0) -- (0,5);
      \draw[-,dotted,black] (0,1) -- (5,1);
      \draw[-,dotted,black] (0,2) -- (5,2);
      \draw[-,dotted,black] (0,3) -- (5,3);
      \draw[-,dotted,black] (0,4) -- (5,4);
      \draw[-,dotted,black] (0,5) -- (5,5);
      \draw[-,dotted,black] (1,0) -- (1,5);
      \draw[-,dotted,black] (2,0) -- (2,5);
      \draw[-,dotted,black] (3,0) -- (3,5);
      \draw[-,dotted,black] (4,0) -- (4,5);
      \draw[-,dotted,black] (5,0) -- (5,5);
      \draw[blue] (0,0) circle (4pt);
      \draw[blue] (1,2) circle (4pt);
      \draw[->,very thick,red] (0,0) -- (1,1);
      \draw[->,very thick,red] (1,1) -- (2,2);
      \draw[->,very thick,red] (2,2) -- (2,1);
      \draw[->,very thick,red] (2,1) -- (1,2);
      \draw[->,very thick,red] (1,2) -- (0,3);
      \draw[->,very thick,red] (0,3) -- (0,2);
    \end{tikzpicture}
    \begin{tikzpicture}
      \draw[-,very thick,black] (0,-0.2) -- (0,0.2);
      \draw[-,white] (0.5,-0.1) -- (0.5,-2.5);
      \draw[->,very thick,black] (0,0) -- (1,0);
    \end{tikzpicture}
    \begin{tikzpicture}
      \draw[->,very thick,black] (0,0) -- (5,0);
      \draw[->,very thick,black] (0,0) -- (0,5);
      \draw[-,dotted,black] (0,1) -- (5,1);
      \draw[-,dotted,black] (0,2) -- (5,2);
      \draw[-,dotted,black] (0,3) -- (5,3);
      \draw[-,dotted,black] (0,4) -- (5,4);
      \draw[-,dotted,black] (0,5) -- (5,5);
      \draw[-,dotted,black] (1,0) -- (1,5);
      \draw[-,dotted,black] (2,0) -- (2,5);
      \draw[-,dotted,black] (3,0) -- (3,5);
      \draw[-,dotted,black] (4,0) -- (4,5);
      \draw[-,dotted,black] (5,0) -- (5,5);
      \draw[->,very thick,blue] (0,0) -- (0,1);
      \draw[->,very thick,red] (0,1) -- (1,2);
      \draw[->,very thick,red] (1,2) -- (2,3);
      \draw[->,very thick,red] (2,3) -- (2,2);
      \draw[->,very thick,red] (2,2) -- (1,3);
      \draw[->,very thick,blue] (1,3) -- (1,4);
      \draw[->,very thick,red] (1,4) -- (0,5);
      \draw[->,very thick,red] (0,5) -- (0,4);
      \draw[-,white] (0,-0.1) -- (0,-0.15);
    \end{tikzpicture}
    \caption{Inserting north steps at points of a walk on $\kS_5$ to create a walk on $\kS_7$.}\label{onefig}
 \end{center}
\end{figure}
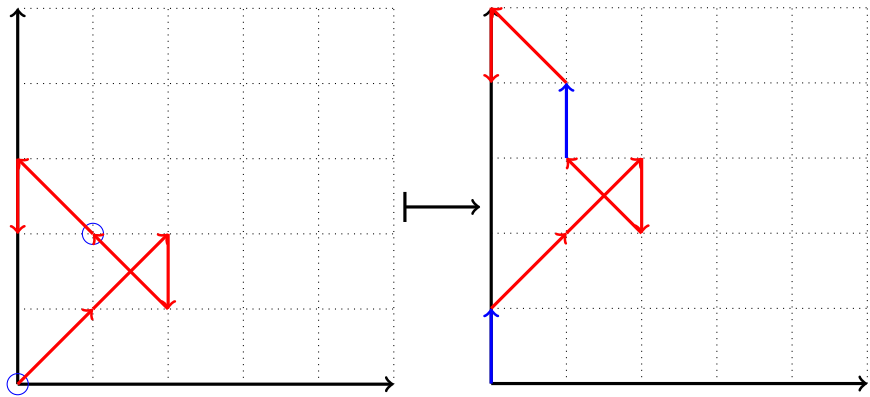

Furthermore, we could take $\kS_7$, the trident, and increase the number of steps by two, inserting a pair of steps in opposite directions. The only thing we need to make it work is that the subwalk taken on the new steps is a meander with respect to the relevant axes. If~$d(i)$ is the number of Dyck prefixes of length $i$, and we are adding the steps $\{\textbf{E,W}\}$ to the step set $\kS_7$, we get $q''_8(n)$, the number of legal walks from $\kS_8 = \{\textbf{N,~NE,~E,~S,~W,~NW}\}$ built in this way, counted by
\[
 q''_8(n) = \sum_{i \geq 0} \binom{n}{i} q_7(i)d(n-i).
\]
Since $q_7(n)$ has exponential growth $4^n$ and $d(n)$ has exponential growth $2^n$, we'd like to replace the functions with their dominant growth and apply the binomial theorem, as before, to get a set with dominant growth $6^n$. Again, doing this requires some care. Figure \ref{twofig} shows an example of this process, where we insert horizontal steps at the positions marked by a blue circle in the graphic on the left. Note that the subwalk taken on the blue steps is a meander.

\begin{figure}[h]
  \begin{center}
    \begin{tikzpicture}
      \draw[->,very thick,black] (0,0) -- (5,0);
      \draw[->,very thick,black] (0,0) -- (0,5);
      \draw[-,dotted,black] (0,1) -- (5,1);
      \draw[-,dotted,black] (0,2) -- (5,2);
      \draw[-,dotted,black] (0,3) -- (5,3);
      \draw[-,dotted,black] (0,4) -- (5,4);
      \draw[-,dotted,black] (0,5) -- (5,5);
      \draw[-,dotted,black] (1,0) -- (1,5);
      \draw[-,dotted,black] (2,0) -- (2,5);
      \draw[-,dotted,black] (3,0) -- (3,5);
      \draw[-,dotted,black] (4,0) -- (4,5);
      \draw[-,dotted,black] (5,0) -- (5,5);
      \draw[blue] (0,0) circle (4pt);
      \draw[blue] (0,1) circle (4pt);
      \draw[blue] (2,3) circle (4pt);
      \draw[blue] (0,5) circle (4pt);
      \draw[blue] (1,4) circle (4pt);
      \draw[blue] (2,2) circle (4pt);
      \draw[->,very thick,red] (0,0) -- (0,1);
      \draw[->,very thick,red] (0,1) -- (1,2);
      \draw[->,very thick,red] (1,2) -- (2,3);
      \draw[->,very thick,red] (2,3) -- (2,2);
      \draw[->,very thick,red] (2,2) -- (1,3);
      \draw[->,very thick,red] (1,3) -- (1,4);
      \draw[->,very thick,red] (1,4) -- (0,5);
      \draw[->,very thick,red] (0,5) -- (0,4);
    \end{tikzpicture}
    \begin{tikzpicture}
      \draw[-,very thick,black] (0,-0.2) -- (0,0.2);
      \draw[-,white] (0.5,-0.1) -- (0.5,-2.5);
      \draw[->,very thick,black] (0,0) -- (1,0);
    \end{tikzpicture}
    \begin{tikzpicture}
      \draw[->,very thick,black] (0,0) -- (5,0);
      \draw[->,very thick,black] (0,0) -- (0,5);
      \draw[-,dotted,black] (0,1) -- (5,1);
      \draw[-,dotted,black] (0,2) -- (5,2);
      \draw[-,dotted,black] (0,3) -- (5,3);
      \draw[-,dotted,black] (0,4) -- (5,4);
      \draw[-,dotted,black] (0,5) -- (5,5);
      \draw[-,dotted,black] (1,0) -- (1,5);
      \draw[-,dotted,black] (2,0) -- (2,5);
      \draw[-,dotted,black] (3,0) -- (3,5);
      \draw[-,dotted,black] (4,0) -- (4,5);
      \draw[-,dotted,black] (5,0) -- (5,5);
      \draw[->,very thick,blue] (0,0) -- (1,0);
      \draw[->,very thick,blue] (1,0) -- (2,0);
      \draw[->,very thick,red] (2,0) -- (2,1);
      \draw[->,very thick,blue] (2,1) -- (1,1);
      \draw[->,very thick,red] (1,1) -- (2,2);
      \draw[->,very thick,red] (2,2) -- (3,3);
      \draw[->,very thick,blue] (3,3) -- (4,3);
      \draw[->,very thick,red] (4,3) -- (4,2);
      \draw[->,very thick,blue] (4,2) -- (3,2);
      \draw[->,very thick,red] (3,2) -- (2,3);
      \draw[->,very thick,red] (2,3) -- (2,4);
      \draw[->,very thick,blue] (2,4) -- (3,4);
      \draw[->,very thick,blue] (3,4) -- (4,4);
      \draw[->,very thick,red] (4,4) -- (3,5);
      \draw[->,very thick,blue] (3,5) -- (2,5);
      \draw[->,very thick,blue] (2,5) -- (1,5);
      \draw[->,very thick,red] (1,5) -- (1,4);
      \draw[-,white] (0,-0.1) -- (0,-0.15);
    \end{tikzpicture}
    \caption{Inserting a meander on $\{\textbf{E,~W}\}$ into a walk from $\kS_7$ to obtain a walk on $\kS_8$.}\label{twofig}
  \end{center}

\end{figure}

\end{ex}

The intuitive idea has merit, but the asymptotic expressions contain a subexponential factor. We show the details in Lemmas \ref{onestep} and \ref{twostep}. The following section outlines how restrictions on the stopping conditions for a path work for some step sets, but can fail for others.

\subsubsection{Restricting stopping conditions}

Given that upper bounds were obtained in a unified way by considering the relaxation of the problem to a half plane model, a good first idea for finding lower bounds would be finding a restriction with the same exponential growth. For example, this was the technique we used in Proposition \ref{S11} to prove the exponential growth for $\kS_{11}$. Unfortunately, restricting the stopping condition fails to provide tight lower bounds for the positive drift cases $\kS_i$, $i \in \{ 5,6,7,8,9,10\}$. This is because each positive drift case is the reflection of a corresponding negative drift case, as shown in Figure \ref{711}.

\begin{figure}[h]
 \begin{center}
  \diagb{N,S,NW,NE} 
  \begin{tikzpicture}[scale = 0.5]
  \draw[<->,black,very thick] (-1,0) -- (1,0); \draw[white,thick] (-1,-1) -- (1,-1); \end{tikzpicture}
  \diagb{N,SE,SW,S}
  \caption{$\kS_7$ and $\kS_{11}$.}\label{711}
 \end{center}
\end{figure}

What this symmetry implies is a bijection between the walks returning to the origin on either step set. Following a walk ending at the origin on one step set in reverse will give you a walk returning to the origin on the other. Given this bijection, the exponential growth factors must be the same, and we cannot use walks returning to the origin for a tight lower bound on the exponential growth factor of step set 7.

In fact, this is true of every pair of step sets which are reflections across a horizontal axis of symmetry. To see this, we note that walks returning to the origin are in bijection for every pair, and so the exponential growth factor for such walks is bounded above by the exponential growth factor for general quarter plane walks on the negative drift version of each pair. We already showed in section \ref{ubs} that these factors are in turn bounded above by the exponential growth factors of the number of half plane walks on the same step set, each of which is strictly smaller than the number of steps. Thus, for every positive drift step set, we need a better lower bound than the number of walks returning to the origin. The following section will provide these.

\subsubsection{Lower bounds from subsets}

Recall Example \ref{int}. By carefully inserting new steps (see Figures \ref{onefig} and \ref{twofig}) into walks on a step set, we were able to bootstrap a lower bound on the exponential growth of a larger step set. This relied on the models having only exponential growth, so the proof reduced to an application of the binomial theorem. In our 23 examples, the growth has a constant factor and a subexponential factor. While the analysis does get a little messier the relevant part of the result is still the same.

\begin{lem}\label{onestep}
Let $q(i) \sim C\beta^ii^s$, where $s \leq 0, C \in \RR^+, \beta \in [1,\infty)$, and define
\[
 q'(n) = \sum_{i \geq 0} \binom{n}{i} q(i).
\]
Then $q'(n)$ has exponential growth factor $\beta + 1$.
\end{lem}

\begin{proof}
Fix $\epsilon >0$. Then since $q(n) \sim C \beta^n n^s$, we know that there exists some $N$ such that~$n > N$ gives
\[
 (1 - \epsilon) \leq \left| \frac{q(n)}{C \beta^n n^s} \right| \leq (1 + \epsilon).
\]
We take $n > N$, and consider the quantity
\[
 q'(n) = \sum_{i \geq 0} \binom{n}{i} q(i).
\]

We may decompose the summation into two parts, those terms with summation index smaller than $N$, and those larger. We look first at the lower bound given by the growth properties:
\begin{eqnarray*}
 q'(n) 
  &\geq& \sum_{i = 0}^N \binom{n}{i}q(i) + \sum_{i = N+1}^n\binom{n}{i} ((1 - \epsilon) C \beta^i i^s), \\
  &\geq& \sum_{i = 0}^N \binom{n}{i}q(i) + (1 - \epsilon)Cn^s\sum_{i = N+1}^n\binom{n}{i} \beta^i \hspace{2em} (\mbox{since } n^s \leq i^s) \\
  &\geq& \sum_{i = 0}^N \binom{n}{i}(q(i) - Cn^s\beta^i) + (1 - \epsilon)Cn^s\sum_{i = 0}^n \binom{n}{i} \beta^i \\
  &\geq& \sum_{i = 0}^N \binom{n}{i}(q(i) - Cn^s\beta^i) + (1 - \epsilon)C(\beta+1)^nn^s \hspace{1em} \mbox{(by the binomial theorem)}.
\end{eqnarray*}

We must now show that the error terms become insignificant as $n$ increases. We do so similarly to before: we divide throughout by $C(\beta+1)^nn^s$ and take the limit. Since the upper limit of summation is a constant, the growth of the numerator is polynomial, and thus dominated by the growth of~$(\beta + 1)^n$. That is, we may find an $N^- > N$ such that for $n > N^-$ we have
\[
 \left| \frac{\sum_{i = 0}^N \binom{n}{i}(q(i) - Cn^s\beta^i)}{C(\beta+1)^nn^s} \right| \leq \frac{\epsilon}{2}.
\]
So taking $n > N^-$, we find
\[
 (1 - 2\epsilon) < \left(1 - \frac{\epsilon}{2}\right) \leq \left| \frac{q'(n)}{C(\beta + 1)^nn^s} \right|.
\]

We work similarly on the upper bound. The main difference is that we wish to increase the value of the bound when making simplifications, so rather than replacing the polynomial factors with something smaller, we remove them entirely.
\begin{eqnarray*}
 q'(n) 
  &\leq& \sum_{i = 0}^N \binom{n}{i}q(i) + \sum_{i = N+1}^n\binom{n}{i} ((1 + \epsilon) C \beta^i i^s), \\
  &\leq& \sum_{i = 0}^N \binom{n}{i}q(i) + (1 + \epsilon) C \sum_{i = N+1}^n\binom{n}{i} ( \beta^i ), \\
  &\leq& \sum_{i = 0}^N \binom{n}{i}(q(i) - C\beta^i) + (1 + \epsilon)C(\beta+1)^n.
\end{eqnarray*}
We then use the same analysis as before to show that the error term becomes smaller than $\frac{\epsilon}{2} C(\beta + 1)^n$. That is, there is some $N^+$ such that taking $n > N^+$ gives
\[
 \left| \frac{q'(n)}{(\beta+1)^n} \right| \leq (1 + \frac{3\epsilon}{2}) < (1 + 2\epsilon).
\]

Putting it all together, we have that for $n > \max \left\{ N^+, N^- \right\}$, the quantity $q'(n)$ satisfies the inequality
\[
 (1-2 \epsilon) C (\beta+1)^n n^s \leq q'(n) \leq (1 + 2\epsilon) C(\beta+1)^n.
\]
Consider what happens when we take the logarithm of the bounds, divide by $n$ and take the limit as $n$ tends to infinity, our standard method of isolating the exponential growth factors. This will cancel all but the exponential growth in each term. Since the upper bound and lower bound both have exponential growth factor $(\beta+1)$, we get
\[
 \log(\beta+1) \leq \lim_{n\rightarrow \infty} \frac1n \log q'(n) \leq \log(\beta+1),
\]
so the exponential growth factor of $q'(n)$ is $(\beta+1)$.
\end{proof}

So, we may now apply this to our example from before, getting a lower bound for the number of walks on $\kS_7$.

\begin{ex}\label{osex}
Take $\kS_5$ and $\kS_7$, as we did in Example \ref{int}. As discussed already, the number of walks on $\kS_7$ obtained by inserting a north step at internal points of a walk from $\kS_5$ is
\[
 q'_7(n) = \sum_{i \geq 0} \binom{n}{i} q_5(i).
\]
We showed in Section \ref{base} that $\beta_5 = 3$, and it follows from $q_5(n) \leq 3^n$ that $q_5(n) \sim C 3^n n^s$, where $C \in \RR^+$ and $s \leq 0$, therefore we may apply proposition \ref{onestep} to find that
\[
 q'_7(n) \sim C' 4^n n^{s'},
\]
where $C'$ must be positive real, since $q'_7(n)$ counts something, and $s' \leq 0$ since $q'_7(n) \leq C 4^n$, as shown in Proposition \ref{onestep}. Thus, if $q_7(n)$ is the number of all possible walks on $\kS_7$, and~$h_7(n)$ is the number of half plane walks taken on $\kS_7$, we have
\[
 q'_7(n) \leq q_7(n) \leq h_7(n).
\]
We may assume that $n$ is sufficiently large, and replace each count with its asymptotic growth, giving
\[
 C'4^nn^{s'} \leq D \beta_7^n n^r \leq E4^n n^q.
\]
Using our standard method of isolating the exponential growth factor with a limit of a logarithm, we find
\[
 \log 4 \leq \log \beta_7 \leq \log 4,
\]
or
\[
 \beta_7 = 4.
\]
\end{ex}

The previous example shows that the exponential growth factor for the sequence $\left(q_7(n)\right)_n$ is 4. The next proposition gives us a method for adding two steps to arrive at a similar conclusion.

\begin{lem}\label{twostep}
Let $q(n) \sim C \beta^n n^s$ where $C$ is a positive real constant, $\beta \in [1,\infty)$ and $s \leq 0$. Let $d(n)$ be the number of Dyck prefixes on $n$ steps. Then the quantity
\[
 q''(n) = \sum_{i = 0}^n \binom{n}{i} d(i)q(n-i)
\]
has growth
\[
 q''(n) \sim C'' (\beta + 2)^n n^{s''},
\]
where $C''$ is a positive real constant and $s'' \leq 0$.
\end{lem}

\begin{proof}
Fix $\epsilon > 0$. Then thanks to the asymptotic form of $q(n)$, we know there is some~$M > 0$ such that for all $n >M$, we have
\[
 (1 - \epsilon) \leq \left|\frac{q(n)}{C\beta^nn^s} \right| \leq (1 + \epsilon).
\]
For $d(n)$, which has growth $d(n) \sim D2^nn^r$ with $r \leq 0$, there is some $M' > 0$ such that for $n > M'$ we have
\[
 (1 - \epsilon) \leq \left|\frac{d(n)}{D2^nn^r} \right| \leq (1 + \epsilon).
\]
Take $N = \max \{M',M\}$ and let $n > 2N$. Then we define
\[
 q''(n) = \sum_{i \geq 0} \binom{n}{i} d(i)q(n-i).
\]

The argument here follows the same lines as Proposition \ref{S13}, splitting the sum into three intervals, and bounding. Since it is similar, we move somewhat faster here. Going through the motions leads to the following inequality.

\begin{eqnarray*}
  (1 - 2\epsilon) \left(Cn^s \sum_{i<N} \binom{n}{i} (d(i)\beta^{n-i} \right. &-& Dn^r2^i\beta^{n-i})  + CDn^{r+s}(\beta + 2)^n \\
  && \left. + Dn^r \sum_{i=n-N+1}^n \binom{n}{i} (2^iq(n-i) - Cn^s2^i\beta^{n-i}) \right) \leq q''(n).
\end{eqnarray*}

Now we analyse the growth. The argument from here follows the same lines as the single step case. We divide throughout by $CD(\beta + 2)^nn^{r + s}$ to find
\begin{eqnarray*}
  (1 - 2\epsilon) && \left(\frac{\sum_{i<N} \binom{n}{i} \big(d(i)\beta^{n-i} - Dn^r2^i\beta^{n-i}\big)}{D(\beta + 2)^nn^r}  + 1 \right. \\
  && \left. +  \frac{ \sum_{i=n-N+1}^n \binom{n}{i} \big(2^iq(n-i) - Cn^s2^i\beta^{n-i}\big)}{C(\beta + 2)^nn^{s}} \right) \leq \frac{q''(n)}{CD(\beta + 2)^nn^{r+s}}.
\end{eqnarray*}
The error terms both tend to zero, so we can find some $N^-$ large enough such that for $n > N^-$ both will be smaller than $\epsilon/2$. So, taking $n > N^-$ means
\[
 (1 - \epsilon) \leq \frac{q''(n)}{CD(\beta+2)^nn^{r+s}}.
\]

The upper bound is similar, so again we omit some details. Since $(1 + 3\epsilon)$ is larger in modulus then both $(1 + \epsilon)$ and $(1 + \epsilon)^2$, we replace both of the smaller terms by the larger, increasing the upper bound. We also replace the monomials in $i$ and $n$ by 1, since the exponents are at most 0. We then do the same analysis, allowing us to find some $N^+$ such that for $n> N^+$,
\[
 (1 + 4\epsilon) \geq \left| \frac{q''(n)}{CD(\beta + 2)^n} \right|.
\]

Putting this all together, we find that for $n > \max\{N^-,N^+\}$ we have
\[
 (1 - 4\epsilon)CD(\beta + 2)^nn^{r+s} \leq q''(n) \leq (1 + 4 \epsilon)CD(\beta + 2)^n.
\]
Again, by using our method of isolating the exponential growth factor, we find that the exponential growth factor of $q''(n)$ is $\beta + 2$ as desired. The constant portion of the growth must be real and positive, since $q''(n)$ is positive, and finally, the polynomial part must have exponent at most zero, since the upper bound has no polynomial portion. That is
\[
 q''(n) \sim C''(\beta + 2)^n n^{s''},
\]
where $C''$ is positive and real and $s'' \leq 0$.
\end{proof}

Let us see this applied to an example.

\begin{ex}\label{tsex}
Take again $\kS_7$ and $\kS_8$ as we did in Example \ref{int}. Now we wish to build walks of length $n$ taken on the step set $\kS_8$ by inserting a meander (or Dyck prefix) of length $n-i$ on the $\{ \textbf{E,~W}\}$ into a walk of length $i$ on $\kS_7$. If $q_7(i)$ is the number of walks of length $i$ on $\kS_7$ and $d(n-i)$ the number of Dyck prefixes of length $n-i$, the number of such walks is
\[
 q_8''(n) = \sum_{i = 0}^n \binom{n}{i} q_7(i)d(n-i).
\]
Now that we have the proofs of the base cases and Lemmas \ref{onestep} and \ref{twostep}, we know that the growth of $q_7(n)$ is suitable for the application of Proposition \ref{twostep}, which tells us that
\[
 q''_8(n) \sim C'' 6^n n^{s''},
\]
where $C''$ is real and positive, and $s''$ is at most zero.

Next, we know that $q''_8(n)$ forms a lower bound for $q_8(n)$
\[
 q''_8(n) \leq q_8(n) \leq h_8(n),
\]
since we can not get all walks of length $n$ on $\kS_8$ in this way. We take $n$ sufficiently large, so that we may make the substitution of each quantity with its growth, giving
\[
 C'' 6^n n^{s''} \leq D \beta^n n^r \leq E 6^n n^{t}.
\]
Once again, we may isolate the exponential growth factor, finding
\[
 \log 6 \leq \log \beta \leq \log 6,
\]
or simply
\[
 \beta = 6.
\]
Thus, we have shown the exponential growth factor for the sequence $(q_8(n))_n$ is $\beta_8 = 6$.
\end{ex}

Thus we have a way of adding one step and two steps at a time to step sets and increasing the exponential growth. This allows us to give lower bounds on the growth of a lot of step sets, using fewer than half of the 23 D-finite cases as groundwork. Lemmas \ref{onestep} and \ref{twostep}, along with the results of Section \ref{base}, give the results in Table \ref{summ}. Before we conclude this chapter, we mention some other uses for our new lemmas.


There is another way to use Lemma \ref{twostep} that has not yet been discussed. So far, all examples show the insertion of a cardinality two step set with zero drift (steps in the opposite direction), but we can also insert two element step sets with positive drift, as long as no component of the drift is towards a boundary. Take $\kS_7 = \diagr{N,S,NE,NW}$ as an example. Rather than using the step set $\diagr{NE,NW,S}$ as a parent for this step set, we could use Lemma \ref{twostep} to insert a Dyck prefix on the steps $\diagr{NE,NW}$ into a Dyck prefix on the vertical steps \diagr{N,S}, as counting the ways to do this would give formula from the hypothesis of Lemma \ref{twostep}. This gives a large subset of the walks on $\kS_7$ as a shuffling of Dyck prefixes (walks on $\kS_1$), which combined with Theorem \ref{bnd} gives the exponential growth factor of 4.

This concludes the material of Chapter \ref{qpwalks} and Part \ref{plp}. Part \ref{conc} will conclude the thesis with a discussion of the results and future work.

\part{Conclusion}\label{conc}

\chapter{Summary and open problems}\label{conclusion}

The main result of this thesis is the proof of the exponential growth factors for the quarter plane models with either zero drift of a drift vector parallel to a boundary through combinatorial means. This was done by bounding the number of walks on a step set $\kS_i$ by known quantities with exponential growth factors equal to the conjectured exponential growth of each model. By then taking a limit which isolated the exponential growth of each bound, we were able to squeeze the exponential growth and prove the conjectured value in each case.

The significance of these results can be expressed in two ways. First, it answers any question about the possibility of a human proof of the exponential growth of quarter plane walks with single component drift in the positive. Secondly, up until now there has been no treatment of the exponential growth of walks taken on these 23 step sets in one piece of work, without either a level of computation similar to that given in \cite{BoKa09,BosKau09}, or some very complicated analysis as in the forthcoming results announced by Fayolle and Raschel in \cite{FaRa12}. Our method not only gives human driven, rigorous results using some fairly elementary techniques, but it also introduces Lemmas \ref{onestep} and \ref{twostep}, which give a systematic method for boot strapping exponential growth from smaller step sets.

\begin{sidewaystable}
  \begin{center}
    \begin{tabular}{|c|r|l|l|}
      \hline
      Class & Functional Equation & Classification of GF & Enumerative Results \\
      \hline
      $\cW_\kS$ & $W_\kS(x,y;t) = 1 + tS(x,y)W_\kS(x,y;t)~~$ & ~~Rational &  \begin{tabular}{l} Exact and asymptotic \\ formulae for coefficients. \end{tabular} \\
      \hline
      $\cF_\kS$ & \begin{tabular}{r} $F_\kS(y;t) = 1 + tP(y)F_\kS(y;t) $ \\ $ - \{y^{<0}\}tP(y)F_\kS(0;t)$ \end{tabular} & ~~Algebraic & \begin{tabular}{l} Coefficient formulae from coefficient \\ extraction of explicit GF, \\ asymptotic results from Theorem \ref{dirasympt} \end{tabular} \\
      \hline
      $\cH_\kS$ & \begin{tabular}{r} $H_\kS(x,y;t) = 1 + tS(x,y)H_\kS(x,y;t) $ \\ $ - \{y^{<0}\}tS(x,y)H_\kS(x,0;t)$ \end{tabular} & ~~Algebraic & ~~Results via bijection with $\cF_\kS$ \\
      \hline
      $\cQ_\kS$ & \begin{tabular}{r} $Q_\kS(x,y;t) = ~~~~~ 1 + tS(x,y)Q_\kS(x,y;t) $ \\ $- t\left( \{x^{<0}\}S(x,y)Q_\kS(0,y;t)\right.$ \\ $\left. + \{y^{<0}\}S(x,y)Q_\kS(x,0;t) \right)$ \\ $ + \omega(\kS)\frac{t}{xy}Q_\kS(0,0;t)$ \end{tabular} & \begin{tabular}{l} 23 D-finite \\ 51 non D-finite \\ 5 conjectured non D-finite \end{tabular} & \begin{tabular}{l} For D-finite cases: \\ enumerative results for 19 \\ cases via orbit sums, some direct \\ counting results, experimental \\ asymptotics via long sum \\ generation, exponential  growth \\ factors from bounding arguments \end{tabular} \\
      \hline
    \end{tabular}
    \caption{A table summarising each class, its functional equation, classification of GFs and summary of enumerative results.}\label{fetab}
  \end{center}
\end{sidewaystable}

Our results are by no means comprehensive and for the remainder of this chapter we discuss extensions of these methods to more general settings. The first extension, and perhaps the lowest hanging fruit, is to work on the remaining 56 step sets - those with non-zero drift which is not parallel to an axis - using the method of bounding and squeezing developed in Chapter \ref{qpwalks}. Next, we consider bounding subexponential growth factors, though also address the prospect of limited success. Finally, we consider generalisations of Theorem \ref{bnd} and Lemmas \ref{onestep} and \ref{twostep} to more general models: those with larger steps and those on higher dimensional lattices.

\subsection*{Results for the remaining sets of small steps}

For the remaining sets of small steps, finding the exponential growth factors is harder. These models are known in many cases to have a non-D-finite generating function, meaning that they are not $G$-series. Therefore we don't have a nice template for the asymptotic growth of the counting sequence as in the 23 cases in Chapter \ref{qpwalks}. We can still apply the methodology as long as we assume a generic growth template
\[
 q_\kS(n) \sim \kappa \beta^n \theta(n),
\]
where $\theta(n)$ is some subexponential factor, that is, $\lim_{n \rightarrow \infty} |\theta(n)|^{1/n} = 1$. Again, here drift is the key in determining whether a model will have full exponential growth equal to the number of steps, or exponential growth diminished by increased interaction with a boundary.

\subsubsection{Positive drift - tight bounds from Lemma \ref{onestep}}

We begin by showing a family of step sets known as the singular sets, shown in Figure \ref{singfig}, for which our method will give exponential growth factors. Notice that all step sets have positive drift in each component, so the drift is into the quarter plane, and it would be reasonable to expect that in each case $\beta_\kS = |\kS|$.

\begin{figure}[h]
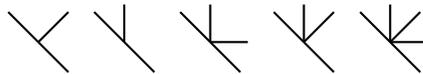

 \begin{center}
  \diagb{NW,NE,SE} \diagb{NW,N,SE} \diagb{NW,N,E,SE} \diagb{NW,N,NE,SE} \diagb{NW,N,NE,E,SE}
  \caption{The singular step sets.}\label{singfig}
 \end{center}
\end{figure}

These are all positive drift step sets, so the exponential growth for any counting sequence is bounded above by the number of steps, found by taking the half plane walks bounded below by the $x$-axis. We can bound the two step sets with cardinality 3 below by a subset of all walks, in both cases by injecting walks taken on step set 
$\kS_{17}$ (or its $y = x$ reflection). Then since we know the results for the smallest two step sets, we can get the other three by applying Lemma \ref{onestep}.

\subsubsection{Negative drift - finding bounds requires care}

Now, we show an example $\kS$ with more interesting drift, shown in Figure \ref{negfig}. The drift is positive in the $y$-component but negative in $x$. We can expect a diminished exponential growth factor due to this drift, but we must take care in bounding it.

\begin{figure}[h]
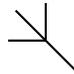

 \begin{center}
  \diagb{N,SE,W,NW}
  \caption{A non-singular step set $\kS$ with a negative drift component.}\label{negfig}
 \end{center}
\end{figure}

The drift is key to finding a least upper bound on the exponential growth. In Theorem \ref{bnd}, it was shown that an upper bound on the exponential growth could be found by removing the $y$-axis as a boundary and using the results of Chapter \ref{hpwalks} to find the exponential growth of the relaxation. What made these upper bounds the lowest available from this type of relaxation was that the drift was always parallel to the $y$-axis, and thus a negative drift quarter plane walk was also a negative drift half plane walk.

However, as already noted, the drift is $x$-negative and $y$-positive. Thus, removing the $y$-axis will give a positive drift half plane model, which won't give a least upper bound. Taking instead the relaxation to the $x \geq 0$ half plane model, we preserve the negative drift and find the least upper bound available from this relaxation.

Following the methods of Chapter \ref{qpwalks}, the next step would be to attempt to remove steps from $\kS$ in order to find a more easily enumerated smaller model to which we could apply Lemma \ref{onestep} or \ref{twostep}. There is only one choice of step to remove here, though, and it leads to the trivial model \diagr{NW,W,SE}. Injection of walks on $\kS_{17}$ would once again work, but we cannot expect this to be tight.

Some future work on this topic could focus on other ways of producing lower bounds to overcome the problems just discussed. One way was already touched on in the proof of Proposition \ref{S05}, where a lower bound on the $\beta_5$ was imported from $\beta_{17}$ by rotating a step away from a boundary. While this didn't help in the example just given, we could find families related by this rotational relation to help cut down any case analysis needed. A long term goal would be finding an automatic way to bound growth constants below, as we did with upper bounds in Theorem \ref{bnd}. Full asymptotic expressions may be on the way, however. Very recently, Fayolle and Raschel have announced `general asymptotic results' in a forthcoming paper in \cite{FaRa12}, giving as an example the asymptotic expression for the counting sequences of walks on $\kS_1$ ending anywhere, on an axis and returning to the origin. They do so by using integral equations found by viewing the model as a boundary value problem, requiring some heavy analysis to find solutions.

\subsection*{Predicting polynomial growth}

Next, we move on to the prediction of the subexponential factors. To simplify the discussion, we focus on our 23 models with simple drift. We know that for these cases, the subexponential factors are of the form
\[
 \theta(n) = \kappa n^{\alpha},
\]
where $\kappa$ is a positive real constant and $\alpha$ is, according to Bostan and Kauers, a negative rational number. For some models, such as $\kS_2 = \{\textbf{NE,~SE,~SW,~NW}\}$, we can find the polynomial factors by an application of Stirling's approximation, since we can count the whole set. For others, where we need to appeal to the asymptotics of a subset, we need a way to isolate the polynomial factor. The following is a very informal sequence of calculations intended to show how we might do this.

So, taking a step set $\kS$, we have our template for the asymptotics of the quarter plane walks on $\kS$
\[
 q(n) \sim \kappa \beta^n n^{\alpha}.
\]
For sufficiently large $n$ we can replace $q(n)$ with its asymptotic estimate and divide out the known exponential growth $\beta^n$:
\[
 \frac{q(n)}{\beta^n} \sim \kappa n^{\alpha};
\]
then taking a logarithm:
\[
 \log \left( \frac{q(n)}{\beta^n} \right) \sim \log \kappa + \alpha \log n;
\]
and the correct limit:
\[
 \lim_{n \rightarrow \infty} \frac{1}{\log n} \log \left( \frac{q(n)}{\beta^n} \right) \sim \alpha;
\]
we can isolate $\alpha$. 

Given a method for isolation involving limits, it would be reasonable to expect to be able to bound and squeeze yet again. Unfortunately, the argument breaks down. Intuitively, it's precisely because the relaxations and restrictions have the same exponential growth: it's then the role of the polynomial factor to modulate the asymptotic growth and keep it close to the actual count. For the walks in the half plane, according to Theorem \ref{dirasympt}, if the drift is positive the polynomial factor is of degree 0. Similarly, the negative drift cases have polynomial factors of degree $-\frac32$, and the zero drift have degree $-\frac12$. In all cases, this is out by a difference of at least $\frac12$ with the results of Table \ref{fintab}, and will not give a tight upper bound on the polynomial factor. It appears, therefore, that new methods are necessary for finding tight bounds on the polynomial growth. Again, according to the announcement of Fayolle and Raschel in \cite{FaRa12}, their approach using boundary value problems and integral equations will produce full asymptotic expressions for classes of walks with small steps.

\subsection*{Generalisations}

The material of Chapter \ref{dir} is taken from a paper \cite{BaFl02}, in which Banderier and Flajolet determine explicit enumerative formulae for directed paths with any step size. These more general results should give upper bounds on the exponential growth of the coefficients of generating functions for more general step sets in the quarter plane, meaning that Theorem \ref{bnd} should translate directly. Furthermore, Lemmas \ref{onestep} and \ref{twostep} make no mention of the size of steps, and so they might be translated into a more general setting as well. This discussion leads to the following conjecture.

\begin{conj}\label{bigstepsboot}
Let $\kS$ be as above. Let $\kT = \kS \cup \{s_0\}$ and $\kU = \kS \cup \{s_1,s_2\}$, where $s_0$ is a step of any size away from any boundary and $\{s_1,s_2\}$ is a step set without drift towards any boundary. If $\beta_\kS$ (respectively $\beta_\kT,\beta_\kU$) is the exponential growth of walks in the quarter plane on $\kS$ (resp. $\kT,\kU$), then Lemma \ref{onestep} will import from $\kS$ the lower bound
\[
 \beta_\kS + 1 \leq \beta_\kT,
\]
and similarly, Lemma \ref{twostep} will import the lower bound
\[
 \beta_\kS + 2 \leq \beta_\kU.
\]
\end{conj}

This conjecture is subject to the same considerations as its counterpart in Chapter \ref{qpwalks}: the application of Lemmas \ref{onestep} and \ref{twostep} correspond to the injection of single steps away from a boundary, or meanders on a pair of steps with drift away from a boundary.

Another generalisation is the application of these results to higher dimensional lattice paths. This is an active area of research \cite{BoKa09}, with much experimental work being done at the time of writing. There should be no trouble in boot strapping as we have done for planar walks with small steps: as long as we are inserting either a single step away from any boundary or a pair of steps with either zero drift of non-zero drift away from any boundary, we should get analogous results. We summarise in the following conjecture.

\begin{conj}\label{hdlemm}
Let $\kS$ be a set of small steps in $\ZZ^d$ for some dimension $d$, let $R$ be some open region in $\ZZ^d$ bounded by rational hyperplanes, let $\kT = \kS \cup \{s_0\}$ and $\kU = \kS \cup \{s_1,s_2\}$ where $s_0$ is a step away from any boundary of $R$ and $s_1,s_2$ is a pair of steps without drift towards any boundary. If $\beta_\kS$ (respectively $\beta_\kT,\beta_\kU$) is the exponential growth of walks in $R$ on $\kS$ (resp. $\kT,\kU$), then Lemma \ref{onestep} will import from $\kS$ the lower bound
\[
 \beta_\kS + 1 \leq \beta_\kT,
\]
and similarly, Lemma \ref{twostep} will import from $\kS$ the lower bound
\[
 \beta_\kS + 2 \leq \beta_\kU.
\]
\end{conj}

Finally, we have a conjecture about producing upper bounds. In the remaining quarter plane models, experimental work has shown that sometimes we must find other half planes - which contain the quarter plane -  to produce upper bounds which are reasonable candidates for tight upper bounds. The conjecture is for step sets of any size in the quarter plane.

\begin{conj}
If a step set $\kS$ is symmetric about an axis $L$, with $L$ passing through the first quadrant (potentially on a boundary), then the half plane walks bounded by the line $L^\perp$ have the same exponential growth as the quarter plane walks.
\end{conj}


\backmatter
\newpage

\addcontentsline{toc}{chapter}{\bibname}

\bibliographystyle{plain}
\bibliography{thesis}


\end{document}